\documentclass[screen,acmsmall]{acmart}

\usepackage{
  amsmath
, stmaryrd
, mathpartir
, mathtools
, xspace
, enumitem
, subcaption
, comment
, booktabs
, ebproof
, todonotes
, xparse
, hyperref
, colortbl
, multirow
, adjustbox
, bbold
}

\usepackage[nameinlink]{cleveref}
\usepackage[all]{xy}

\usepackage{iwilare}

\newtheorem*{remark}{Remark}

\begin{document}

\setcopyright{cc}
\setcctype{by}
\acmDOI{10.1145/3776703}
\acmYear{2026}
\acmJournal{PACMPL}
\acmVolume{10}
\acmNumber{POPL}
\acmArticle{61}
\acmMonth{1}
\received{2025-07-09}
\received[accepted]{2025-11-06}

\title{Di- is for Directed: \\ First-Order Directed Type Theory via Dinaturality}

\renewcommand{\shorttitle}{Di- is for Directed: First-Order Directed Type Theory via Dinaturality}

\author{Andrea Laretto}
\orcid{0000-0002-6413-5794}
\email{andrea.laretto@taltech.ee}
\author{Fosco Loregian}
\orcid{0000-0003-3052-465X}
\email{fosco.loregian@taltech.ee}
\author{Niccolò Veltri}
\orcid{0000-0002-7230-3436}
\email{niccolo.veltri@taltech.ee}
\affiliation{%
  \institution{TalTech}
  \city{Tallinn}
  \country{Estonia}
}

\renewcommand{\shortauthors}{Andrea Laretto, Fosco Loregian, Niccolò Veltri}

\begin{abstract}
  We show how dinaturality plays a central role in the interpretation of directed type theory where types are given by (1-)ca\-te\-go\-ries and directed equality by $\hom$-functors. We introduce a first-order directed type theory where types are semantically interpreted as categories, terms as functors, predicates as dipresheaves, and proof-relevant entailments as dinatural transformation. This type theory is equipped with an elimination principle for directed equality, motivated by dinaturality, which closely resembles the $\J$-rule used in \ML type theory. This directed $\J$-rule comes with a simple syntactic restriction which recovers all theorems about symmetric equality, except for symmetry. Dinaturality is used to prove properties about transitivity (composition), congruence (functoriality), and transport (co\-Yo\-ne\-da) in exactly the same way as in \ML type theory, and allows us to obtain an internal ``naturality for free''. We then argue that the quantifiers of directed type theory should be ends and coends, which dinaturality allows us to capture formally. Our type theory provides a formal treatment to (co)end calculus and Yoneda reductions, which we use to give distinctly logical proofs to the (co)Yoneda lemma, the adjointness property of Kan extensions via (co)ends, exponential objects of presheaves, and the Fubini rule for quantifier exchange. Our main theorems are formalized in Agda.
\end{abstract}



\maketitle

\definecolor[named]{ACMPurple}{cmyk}{0.55,1,0,0.15}

\newcommand{\CiteAppendix}[2]{\Cref{#2}}

\createrule{exp}{\textsf{exp}}
\createrule{expinv}{\textsf{exp}$^{-1}$}
\createrule{var}{\textsf{var}}
\createrule{prod}{\textsf{prod}}
\createrule{cutassoc}{\textsf{assoc-nat-din-nat}}
\createrule{cutnat}{\textsf{cut-nat}}
\createrule{cutdin}{\textsf{cut-din}}
\createrule{op}{\textsf{op}}
\createrule{reindex}{\textsf{idx}}
\createrule{refl}{$\refl$}
\createrule{top}{$\top$}
\createrule{dinathomnat}{\textsf{dinat-$\hom$nat}}
\createrule{J}{$\J$}
\createrule{Jinv}{$\J^{-1}$}
\createrule{hom}{$\hom$}
\createrule{Jcomp}{\textsf{$\J$-comp}}
\createrule{Jop}{\textsf{$\J$-op}}
\createrule{Jeq}{\textsf{$\J$-eq}}
\createrule{opeq}{\textsf{$\op$-eq}}
\createrule{eq}{\textsf{eq}}
\createrule{classicJ}{\textsf{$\J$=}}
\createrule{classicrefl}{\textsf{$\refl$=}}
\createrule{end}{\textsf{end}}
\createrule{endinv}{\textsf{end}$^{-1}$}
\createrule{coend}{\textsf{coend-without-frobenius}}
\createrule{coendinv}{\textsf{coend}$^{-1}$}
\createrule{coyoneda}{\textsf{coYoneda}}
\createrule{yoneda}{\textsf{Yoneda}}
\createrule{homfunc}{\textsf{$\Rightarrow$-func}}
\createrule{coendfrobenius}{\textsf{coend}}
\createrule{weaken}{\textsf{wk}}
\createrule{enddinr}{\textsf{end-din}$_r$}
\createrule{enddinl}{\textsf{end-din}$_l$}
\createrule{endnatr}{\textsf{end-nat}$_r$}
\createrule{endnatl}{\textsf{end-nat}$_l$}
\createrule{natcutcoherence}{\textsf{cut-coherence}}
\createrule{cutnatidl}{\textsf{cut-nat-id}$_l$}
\createrule{cutnatidr}{\textsf{cut-nat-id}$_r$}
\createrule{cutdinidl}{\textsf{cut-din-id}$_l$}
\createrule{cutdinidr}{\textsf{cut-din-id}$_r$}
\createrule{contract}{\textsf{contr}}
\createrule{coendunit}{\textsf{coend-unit}}
\createrule{endcounit}{\textsf{end-counit}}

\def\agdabase{https://github.com/iwilare/dinaturality/blob/popl-2026/}

\section{Introduction}\label{sec:introduction}

Homotopy type theory~\cite{UnivalentFoundationsProgram2013homotopy,Awodey2009homotopy,Berg2010types} revolutionized the way we think about types. One of the fundamental insights that inspired this revolution was first given in a seminal paper by Hofmann and Streicher~\cite{Hofmann1998groupoid}, with a remarkably simple idea: rather than viewing types just as \emph{sets} of inhabitants, they give an interpretation of \MLTT where types are taken to be \emph{groupoids}, i.e., categories in which every morphism is an isomorphism. The inhabitants of a type become the objects of a groupoid, and the morphisms in a groupoid represent the \emph{equalities} between inhabitants, of which there can be more than a unique one. The reason why morphisms need to be invertible is because of the inherently \emph{symmetric} nature of equality: given a proof of equality $e : x = y$, there is always a proof of the equality $e' : y = x$.

A natural question follows: why not \emph{categories}, rather than groupoids? Can there be a type theory where types are interpreted as \emph{categories}, where morphisms need not be invertible?
Such a system should take the name of \emph{directed type theory}~\cite{Ahrens2023bicategorical,North2019towards,Altenkirch2024synthetic,Licata20112,Weaver2020constructive,Gratzer2024directed}, where the directed aspect comes precisely from this asymmetric interpretation of ``equality''.

\begin{figure}[ht]
\begin{tabular}{r|l}
Types $C$ & Categories $\C$ \\
Functions $f : C \to D$ & Functors $F : \C \to \D$ \\
Relations $R : C \times D \to \textsf{Bool}$ & Profunctors $P : \Cop\times\D \> \Set$ \\
Predicates $P : C \to \textsf{Bool}$ & Presheaves $P : \Cop \> \Set$ \\
Points of a type & Objects of a category \\
Equalities $e : a =_C b$ & Morphisms $e : \hom_\C(a,b)$ \\
Equality types $=_C : C \times C\to\Type$ & Hom functors $\hom_\C : \Cop \x \C \> \Set$ \\
Universal quantifiers & Ends $\Endf{x:\C}P(\nxx)$ \\
Existential quantifiers & Coends $\Coendf{x:\C}P(\nxx)$ \\
\end{tabular}
\Description{A table with two columns comparing logical concepts with their directed generalizations. The left column lists types, functions, relations, predicates, points of a type, equalities, equality types, universal quantifiers, and existential quantifiers. The right column lists categories, functors, profunctors, presheaves, objects of a category, morphisms, hom functors, ends, and coends respectively.}
\caption{The directed generalization of logical concepts.}
\label{table:directed_generalization}
\end{figure}

Directed type theory has been a hot topic of type-theoretical research for the past decade~\cite{Nuyts2023higher,New2023formal,Gratzer2025yoneda,Altenkirch2024synthetic,Chu2024directed,Chu2025dependent,Weaver2024bicubical,Neumann2025generalized}. This quest for the directed generalization has a specific application in mind: in the same way that HoTT can be used to study homotopy theory in a type-theoretical way, directed type theory promises the study of \emph{category theory} in a type-theoretical way.

Category theory has proven to be a fundamental topic in the semantics of programming languages~\cite{Scott2000some,Moggi1991notions,Crole1994categories,Lambek1986introduction}, where it shines as the common framework that ties together logic, proofs, and types in the Curry-Howard-Lambek correspondence~\cite{Jacobs1999categorical,Casadio2021Joachim,Harper2016practical}. The \emph{unifying} role of category theory stretches even beyond computer science, in algebraic topology~\cite{MacLane1998categories}, universal algebra~\cite{Lawvere1963functorial}, quantum mechanics~\cite{Heunen2019categories}, and physics~\cite{Baez2010physics}.

This compelling series of applications comes at a cost: category theory can be overwhelming for newcomers, with overly abstract results and seemingly complicated ideas (e.g., the Yoneda lemma~\cite{Boisseau2018what}, Kan extensions~\cite{Hinze2012kan}). Even worse, these abstractions come baggaged with a plethora of naturality and functoriality side conditions that need to be checked~\cite{New2023formal}.

Directed type theory promises to reinterpret category theory \emph{itself} under a logical perspective, taking the Curry-Howard-Lambek correspondence to the next level: what once were abstract yet overarching results in category theory become \emph{simple type-theoretical statements}, which one can then prove in a system that takes care of naturality and functoriality bureaucracy \emph{for free}.

One of the ultimate goals of directed type theory is to capture this multitude of directed phenomena under a single, unified type-theoretical framework: since morphisms of a category can be viewed just as (directed) equalities, one can use directed type theory as a tool to represent and reason about programs, processes, rewrites, transitions~\cite{Ahrens2022semantics}, concurrency via directed spaces~\cite{North2019towards,Fajstrup2016directed}, types and terms of type theories (e.g., via ``directed higher inductive types''~\cite{Kavvos2019quantum,Weaver2024bicubical}), \emph{all internally} to the same type theory.

What is currently missing from the current conception of directed type theory is a direct description of what such a system should look like in the elementary case of 1-categories. Taking inspiration from the simplicity of the groupoid model in Hofmann and Streicher's approach,

\medskip
\begin{center}
\emph{{We introduce a first-order directed type theory with simple, straightforward semantics in 1-categories:} \\ proving theorems about directed equality follows the same exact steps of Martin-L\"of type theory, \\ and non-trivial theorems in category theory can be captured in a concise and distinctly logical way.}
\end{center}
\medskip

\noindent How should type-theoretical ideas change under the view of directed type theory? Category theorists have long known what the most natural path for the directed generalization should be~\cite{Lawvere1970equality}: functions between types should be \emph{functors} (i.e., functions which respect directed equalities), relations are naturally interpreted as \emph{profunctors}~\cite{Borceux1994handbook}, and \emph{(co)presheaves} can be thought of as generalized predicates~\cite{Bainbridge1976feedback}. We summarize the main ideas of the \emph{directed generalization} in \Cref{table:directed_generalization}.

Under this directed lens, familiar type-theoretical statements of equality become elementary definitions in category theory: we give a few simple examples in \Cref{table:symmetric_directed} in the canonical setting of \emph{first-order logic}, which is closely connected to the formal system later explored in this paper.

\begin{figure}[ht]
\begin{tabular}{c|l}
\multirow{2}{*}{
 $\begin{array}{r@{\ }c@{\ }l@{\ }c@{\ }c}
{x = y \hspace{0.6em}} & \land & \hspace{0.6em} y = z & \vdash & x = z \\[0.2em]
\hom_\C(x,y) & \times & \hom_\C(y,z) & \to & \hom_\C(x,z)
 \end{array}$} & Transitivity of equality \\[0.2em]
 & Composition in a category \\[0.5em]\hline\\[-0.8em]

\multirow{2}{*}{
 $\begin{array}{c@{\ }c@{\ }c}
x = y & \vdash & f(x) = f(y) \\[0.2em]
\hom_\C(x,y) & \to & \hom_\D(F(x),F(y))
 \end{array}$} & Congruence / functions respect equality \\[0.2em]
 & Action on morphisms of functors \\[0.5em]\hline\\[-0.8em]

\multirow{2}{*}{
 $\begin{array}{r@{\ }c@{\ }c@{\ }c@{\ }c}
x = y & \land & P(x) & \vdash & P(y) \\[0.2em]
\hom_\C(x,y) & \times & P(x) & \to & P(y)
 \end{array}$} & Substitution / transport along equality \\[0.2em]
& Action on morphisms of copresheaves
\end{tabular}
\caption{Elementary statements for symmetric equality and their directed counterparts.}
\Description{A table with two columns, with each row alternating an elementary theorems for symmetric equality and its directed counterpart. On the left side, three rows show the statements: transitivity of equality, congruence/functions respect equality, and substitution/transport along equality. On the right side, each row describes the corresponding concept in category theory: composition in a category, action on morphisms of functors, and action on morphisms of copresheaves.}
\label{table:symmetric_directed}
\end{figure}

However, directed type theory is not so straightforward. We list some fundamental challenges:

\textbf{Challenge 1. How to change rules for equality.} One can use their favorite proof assistant or logical system to prove the theorems in \Cref{table:symmetric_directed}: in the case of symmetric equality, typically this is done using an introduction rule $\Ruleclassicrefl$ and an elimination rule $\RuleclassicJ$ called \emph{$\J$-rule}~\cite{Hofmann1997syntax}, shown in \Cref{fig:j_and_refl} again for first-order logic. The introduction rule simply states that equality is reflexive. The elimination rule $J$ intuitively says that, if we assume an equality $e : a = b$ and we want to prove a predicate $P(a,b)$ for some variables $a,b : C$, it is sufficient to consider the case ``on the \emph{diagonal}'' $P(x,x)$, where $a$ and $b$ are identified with the same $x$. These two rules allow all of the above statements about symmetric equality to be derived almost ``for free'' just by contracting away equalities. However, $\RuleclassicJ$ allows for symmetry of equality to be derived, simply by picking $P(a,b) := b = a$. This is incompatible with the directed case, as not every morphism has an inverse.

The fundamental question then becomes: \emph{how can we tweak the rules of equality to disallow symmetry, and yet be able to derive ``for free'' the above theorems also in the case of directed equality?}

\begin{figure}[ht]%
\LinkRuleclassicJ%
\LinkRuleclassicrefl%
\centering
\begin{prooftree}
\infer0[\Ruleclassicrefl]{[x:C]~\Phi \vdash \refl : x = x}
\end{prooftree}
\quad
\begin{prooftree}
\hypo{[x : C]~\takespace{a = b, \, \Phi(a,b) }{\Phi(x,x)} \vdash & \takespace{J(h)}{h} : P(x,x)}
\infer1[\RuleclassicJ]{[a : C, b : C]\ a = b, \, \Phi(a,b) \vdash & J(h) : P(a,b)}
\end{prooftree}
\caption{Introduction and elimination rules for symmetric equality in first-order logic.}
\label{fig:j_and_refl}
\vspace{0.7em}
\begin{prooftree}
\infer0[\Rulerefl]{[x : \C]~\Phi \vdash & \textsf{refl} : \hom_\C(\nxx)}
\end{prooftree}
\quad
\begin{prooftree}
\hypo{[x : \C]~\takespace{\hom(a,b), \, \Phi(\overline a, \overline b)}{ \Phi(x,\n x)} \vdash & \takespace{J(h)}{h} : P(\nxx)}
\infer1[\RuleJ]{[a : \Cop, b : \C]\ {\hom}(a,b), \, \Phi(\overline a, \overline b) \vdash & J(h) : P(a,b)}
\end{prooftree}
\Description{A figure showing the introduction and elimination rules for directed equality in first-order dinatural directed type theory.}
\caption{Introduction and elimination rules for directed equality in first-order dinatural directed type theory.}
\label{fig:directed_j}
\end{figure}%

\textbf{Challenge 2. Polarity problems.} Another issue arises in the first example of \Cref{table:symmetric_directed}: since types are now categories, with each type $\C$ there should be a type $\C^\op$ (the opposite category) of the opposite ``polarity'', where the inhabitants are the same but all directed equalities are reversed. The \emph{type of directed equalities} $\hom_\C(x,y)$ then is \emph{asymmetric}, and receives a ``negative'' argument $x\!:\!\C^\op$ and a ``positive'' one $y\!:\!\C$, and provides the \emph{set} (i.e., a category with only trivial directed equalities) of morphisms between objects $x,y$ of $\C$.

The problem is that in the statement for transitivity of directed equality (i.e. composition) the variable $y$ appears both on the right side of $\hom_\C(x,y)$, with type $\C$, and at the same time on the left side of $\hom_\C(y,z)$, with seemingly different type $\Cop$! The same problem arises in $\Rulerefl$, since $x$ is used on both sides of $\hom$, and in $\RuleJ$ because in $P(x,x)$ the same $x$ needs to be used with both polarities. One solution first considered by North~\cite{North2019towards} and later revisited by Altenkirch~and~Neumann~\cite{Altenkirch2024synthetic} is to revert back to the undirected case of \emph{groupoids}. This solution may feel unsatisfactory, since one does not intuitively expect groupoids to appear in the semantics of a type theory where types are \emph{categories}. \emph{How do we solve these polarity problems without having to resort to groupoids?}

\textbf{Challenge 3. Directed quantifiers.} Another fundamental yet unexplored question is \emph{what the quantifiers of directed type theory should be in the 1-categorical case}. Because of the above polarity issues, this is not a trivial question: should the variable $y$ in the statement of transitivity be bound as a variable of type $y:\C$ or $\n y:\Cop$? A natural expectation is that quantifiers should be able to bind \emph{both} occurrences of $y$ at once.

\medskip
\begin{center}
\emph{This paper proposes a simple solution that addresses \\ all of the above challenges for directed type theory: {dinaturality}~\cite{Dubuc1970dinatural}.}
\end{center}
\medskip

The intuition behind dinaturality and dinatural transformations is that the same variable is allowed to appear both positively and negatively at the same time, irrespectively of polarity.

Not only do we deal with the variance problems without ever having to mention groupoids, but dinaturality also tells us what a \emph{directed} $\J$ rule should look like, which we illustrate in \Cref{fig:directed_j} next to the symmetric case. Curiously, this rule is reminiscent of the elimination rule for equality of standard \MLTT, but it comes equipped with a precise syntactic restriction that does not allow symmetry of directed equality to be derived.

What about quantifiers? Dinaturality comes again to the rescue, hinting at a possible answer: intimately connected to the notion of dinatural transformation are the notions of \emph{end} and \emph{coend}~\cite{Loregian2021coend}. Ends and coends, respectively denoted as $\Endf{x:\C}P(\nxx)$ and $\Coendf{x:\C}P(\nxx)$ for some functor $P:\Cop\x\C\to\Set$, are to be thought of as a sort of universal and existential quantifiers on $P$, respectively. Just like a quantifier, the integral sign of (co)ends binds positive and negative occurrences of variables, indicated as $x:\C$ and $\overline x:\C^\op$.

The main application of (co)ends is that they allow non-trivial statements in category theory to be formulated in a concise way~\cite{Loregian2021coend}: for example, one can use ends to characterize the set of natural transformations as the end $\Nat(F,G) \iso \Endf{x:\C}\!\hom_\D(F(\overline x),G(x))$; note the resemblance between this end and the universal quantification used in the usual definition of natural transformation. With this, we can rephrase the well-known Yoneda lemma~\cite{Leinster2014basic} as a simple isomorphism, shown in \Cref{fig:intro_yoneda} next to its logical ``decategorified'' interpretation. A similar statement holds for the case of existential quantifiers and coends, shown in \Cref{fig:intro_coyoneda}, which often takes the slogan of ``presheaves are colimits of representables''~\cite{Leinster2014basic} or ``co\-Yo\-ne\-da lemma''~\cite{Loregian2021coend,Clarke2022profunctor}.

\begin{figure}[ht]
\begin{subfigure}{0.45\textwidth}
    \centering
\def\arraystretch{1.2}
\caption{\begin{tabular}{c@{\ }c@{\ }c@{}c@{\ }c@{\ }c}
    $P(a)$ & $\iso$ & $\Endf{x:\C}    $ & $\hom_\C(a,\overline x)$ & $\Rightarrow$ & $ P(x) \vspace{0.2em}$ \\ \arrayrulecolor{lightgray}\hline\hline
    $P(a)$ & $\Leftrightarrow$ & $\forall(x:C).$ & $a =_C x$ & $\Rightarrow$ & $ P(x)$
\end{tabular}}
\label{fig:intro_yoneda}
\end{subfigure}
\begin{subfigure}{0.45\textwidth}
    \centering
\def\arraystretch{1.2}
\caption{\begin{tabular}{c@{\ }c@{\ }c@{}c@{\ }c@{\ }c}
    $P(a)$ & $\iso$ & $\Coendf{x:\C}    $ & $\hom_\C(\n x,a)$ & $\x$ & $ P(x) \vspace{0.2em}$ \\ \arrayrulecolor{lightgray}\hline\hline
    $P(a)$ & $\Leftrightarrow$ & $\exists(x:C).$ & $x =_C a    $ & $\land$ & $ P(x)$
\end{tabular}}
\label{fig:intro_coyoneda}
\end{subfigure}
\Description{A figure showing the Yoneda lemma and coYoneda lemma using (co)ends and their corresponding logical statements.}
\caption{Yoneda and coYoneda lemma using (co)ends and their corresponding logical statements.}\end{figure}

The first-order formula behind the (co)Yoneda lemma can be proven using any formal system: our directed type theory is the first elementary treatment of a formal system for the \emph{directed} case, where one can modularly use rules for quantifiers and equality as done in logic, e.g., with suitable introduction/elimination rules specific to directed equality and (co)ends. To give a taste of how closely our approach follows that of a standard logical proof, we show in \Cref{fig:yoneda_lemma_proof_intro} a proof of the Yoneda lemma in our type theory next to its ``decategorified'' proof in first-order logic.

\begin{figure}
\centering
\hspace{-1.5em}
\begin{minipage}{0.40\textwidth}
\begin{prooftree}
   \infer[no rule]0{[a\!:\!C]~\Phi(a) \vdash \forall(x:C).~a =_C x \Rightarrow P(x)}
   \infer[double]1[\seqrule{$\forall$}]{[a\!:\!C,x\!:\!C]~\Phi(a) \vdash a =_C x \Rightarrow P(x)}
   \infer[double]1[\seqrule{$\Rightarrow$}]{[a\!:\!C,x\!:\!C]\hspace{1.6pt}\ a =_C x \land \Phi(a) \hspace{1.3pt} \vdash P(x)}
   \infer[double]1[\seqrule{$=$}]{[z:C]~\Phi(z) \vdash P(z)}
\end{prooftree}
\end{minipage}
\hspace{2.5em}
\begin{minipage}{0.40\textwidth}
\begin{prooftree}
   \infer[no rule]0{[a\!:\!\C]~\Phi(a) \vdash \Endf{x:\C}\ {\hom}_\C(a,\overline x) \Rightarrow P(x)}
   \infer[double]1[\Ruleend]{[a\!:\!\C,x\!:\!\C]~\Phi(a) \vdash \hom_\C(a,\overline x) \Rightarrow P(x)}
   \infer[double]1[\Ruleexp]{[a\!:\!\C,x\!:\!\C]\hspace{0.86pt}\ \hom_\C(\overline a,x)\hspace{0.8pt}\x\hspace{0.8pt}\Phi(a) \hspace{1.3pt} \vdash P(x)}
   \infer[double]1[\RuleJ]{[z:\C]~\Phi(z) \vdash P(z)}
\end{prooftree}
\end{minipage}
\Description{A figure showing a proof of the Yoneda lemma in first-order logic and its proof in dinatural directed type theory, using bidirectional derivation rules.}
\caption{A proof of the Yoneda lemma in first-order logic, and its proof in dinatural directed type theory.}
\label{fig:yoneda_lemma_proof_intro}
\vspace{-1.5em}
\end{figure}

{ \textbf{(Co)end calculus.}} It is common knowledge among category theorists that there is a formal aspect to the manipulation of ends and coends, outlined in~\cite{Loregian2021coend}, that allows such non-trivial theorems to be proven using simple ``mechanical'' rules. This ``(co)end calculus'' has proven to be particularly useful in theoretical computer science, for example in the context of profunctor optics~\cite{Clarke2022profunctor,Boisseau2018what} and their string diagrams~\cite{Roman2020open,Boisseau2020string}, strong monads and functional programming~\cite{Asada2010arrows,Asada2010categorifying,Hinze2012kan,Uustalu2020eilenberg}, quantum circuits~\cite{Hefford2023coend}, and logic~\cite{Pistone2018proof,Pistone2021yoneda,Galal2020profunctorial}. Our work gives a \emph{logical interpretation} to (co)end calculus by reconceptualizing it just as a first-order instance of directed type theory, which is what motivates our focus on a non-dependent presentation of directed type theory.

{\textbf{Dinaturality.}} Dinaturality is not a novel concept: dinatural transformations are a generalization of natural transformations for functors $F,G : \Cop\x\C \> \D$ with mixed-variances~\cite{Dubuc1970dinatural}.

Serendipitously, the ``\emph{di}'' in dinatural stands for \emph{diagonal}: a dinatural is a family of maps $\alpha_x : F(x,x) \to G(x,x)$ which is required to be given only on the \emph{diagonal} of $F,G$ by equating the contravariant and covariant variables with the same value $x:\C$. Such family of maps is required to satisfy a certain equational property, which generalizes the usual square of natural transformations.

Famously, however, dinatural transformations \emph{do not always compose}: a well-known sufficient condition for the composability of dinaturals is the absence of loops in a suitably associated graph~\cite{Eilenberg1966generalization,McCusker2021composing}. This loop-freeness similarly arises in linear logic with the Danos-Regnier criterion \cite{Guerrini2004proof,Blute1996linear,Blute1998shuffle,Blute1993linear}, and more in general in logic where composition corresponds to cut elimination~\cite{Petric2003g,Girard1992normal}.

There is a particularly deep connection between dinaturality and parametricity in programming languages~\cite{Pistone2017dinaturality,Pistone2019completeness,Plotkin1993logic,Voigtlaender2020free} and realizable models for System F~\cite{Bainbridge1990functorial,Freyd1992dinaturality} where all dinaturals compose.
Dinaturality has remained somewhat of an understudied subject, partly because this lack of general compositionality has proven to be particularly hard to explain in full generality~\cite{Santamaria2019towards}: an in-depth review on dinaturality and its importance for computer science can be found in~\cite{Santamaria2019towards},~\cite[Sec. 3]{Scott2000some}.

\subsection{Contribution}

In this work, we connect for the first time dinatural transformations to directed type theory, showing how they turn out to be the key technical notion needed to capture directed type theory in an elementary and straightforward way.

Our general approach to directed type theory is to take the simplicity of the groupoid model of Hofmann and Streicher~\cite{Hofmann1998groupoid} and generalize it to the directed case with a first-order (yet expressive) system aimed at capturing two specific aspects of directed type theory: first, the ability to construct and prove properties about theorems of directed equality by following precisely the same steps as in \MLTT; second, the ability to exploit the power of (co)ends-as-quantifiers~\cite{Loregian2021coend} to give simple and concise logical proofs of well-known theorems in category theory.

We summarize the main contributions and technical aspects of this paper:
\begin{enumerate}[leftmargin=*,itemsep=0.5em]
\item {\emph{Setting.}} We introduce a first-order (non-dependent) directed type theory where types are semantically interpreted as (small) 1-categories, terms as functors, predicates as dipresheaves (i.e. functors $\C^{\op} \times \C \> \Set$), directed equality predicates as $\hom$-functors, and proof-relevant entailments as dinatural transformations which are not required to compose in the usual sense.
\item {\emph{First-order type theory.}} Our directed type theory builds on the well-known canonical setting of first-order logic, with judgments structured in a similar way~\cite[4.1]{Jacobs1999categorical}: we have simply-typed types and terms, on which we build a proof-relevant logic with predicates, entailments, and \emph{equality of entailments}. This last aspect is typically absent in usual accounts of first-order logic, but it is crucial in our presentation because it is precisely the point in which we use dinaturality.

Our system is a \emph{type theory} in the sense of Jacobs~\cite[p.~9,~(iii)]{Jacobs1999categorical}: proofs have an explicit computational content, e.g., the proof of transitivity of directed equality is a bona-fide family of functions that can be used to compose equality witnesses (i.e., morphisms) in the type theory.
\item {\emph{Directed equality elimination.}} In our 1-categorical setting, the rules for directed equality are straightforward: the directed equality introduction rule is essentially the same as the usual $\refl$, which we validate using identities in $\hom$-sets. We identify a directed equality elimination rule which is again syntactically reminiscent of the $\J$-rule, but equipped with a syntactic restriction that does not allow for symmetry to be derived. This syntactic restriction is not ad-hoc, but it is justified by a precise semantic fact: the connection between dinaturality and ordinary naturality. In short, the syntactic requirement to contract a directed equality in context $\hom_\C(x,y)$ for $x:\Cop,y:\C$ is that both $x$ and $y$ must appear only covariantly (i.e., with the ``correct polarity'') in the conclusion and only contravariantly (i.e., with the ``wrong polarity'') in the assumptions in context. The non-derivability of symmetry, aside from the syntactic restriction of $J$, follows by soundness and the existence of a countermodel. %
\item {\emph{Directed theorems.}} The rules for directed equality allow us to recover in \Cref{sec:examples_directed_equality} the same type-theoretic definitions about symmetric equality derivable in standard \MLTT, {except for symmetry}: e.g., transitivity of directed equality (composition in a category), congruences of terms along directed equalities (the action of a functor on morphisms), transport along directed equalities (i.e., the coYoneda lemma).
\item {\emph{Directed properties.}} In our type theory one can also prove \emph{properties} of these maps using a \emph{dependent} version of directed $\J$ specific to the judgment of equality of entailments: for example, one can show that the composition of directed equalities is automatically associative and unital on both sides (one of the two sides is definitionally unital on the equality that is being contracted). The semantic notion of dinaturality is not used to \emph{construct} such maps (functoriality suffices), but to validate this {dependent} directed $\J$ rule.
With this rule one can internally prove that functoriality and naturality follow ``for free'', again, by a simple directed equality contraction.
\item {\emph{Polarity.}} Our type theory is equipped with a precise notion of polarity and variance which is used to implement the syntactic restriction behind the $\J$ rule. Even in our non-dependent case the treatment of variables is non-trivial, since dinaturality requires a precise definition of variance/polarity that differs from the approaches described in other works~\cite{North2019towards,Altenkirch2024synthetic,Nuyts2015towards,Gratzer2024directed}.
\item  {\emph{Category theory, logically.}} Our type theory allows us to give direct, concise, and distinctly logical proofs of well-known (yet non-trivial) theorems in category theory by using $\hom$ as a directed equality: e.g., the (co)Yoneda lemma, Kan extensions computed via (co)ends are adjoint to precomposition, presheaves form a closed category, $\hom$-functors preserve (co)limits, and the Fubini rules; each of these follows by modularly using the logical properties of each connective.
\item {\emph{(Co)end calculus.}}  The approach used to prove these theorems is to combine the perspective of $\hom$ as directed equality with the ideas of ``(co)end calculus''~\cite{Loregian2021coend}, viewing (co)ends as the \emph{directed quantifiers} of directed type theory. (Co)end calculus as treated in \cite{Loregian2021coend} uses various semantic properties of (co)ends, which are however selected \emph{ad-hoc} and not modularly organized in a precise set of rules: our type theory gives a formal treatment to these techniques, approaching proofs in a different and more logical fashion. The choice of a first-order (hence non-dependent) type theory is to capture (co)end calculus, which is typically first-order in practical applications.
\item {\emph{Yoneda technique.}} Our proofs are logical, yet mirror the way that (co)end calculus is used in practice (e.g.,~\cite{Boisseau2018what,Hinze2012kan,Roman2020open}), i.e., via a ``Yoneda-like'' series of \emph{natural} isomorphisms of sets: to prove that two objects $A, B\!:\!\C$ are isomorphic, one assumes a generic object $\Phi$ and then applies a series of isomorphisms of sets \emph{natural} in $\Phi$ to establish that $\C(\Phi,A) \iso \C(\Phi,B)$, from which $A \iso B$ follows by the fully faithfulness of the Yoneda embedding~\cite{Boisseau2018what,Leinster2014basic}. The same technique can be used to show adjunctions, and that \emph{functors} are naturally isomorphic.  %
\item {\emph{Adjoint-form rules.}} In typical syntactic presentations of type theory, rules for connectives are formulated to make cut admissible~\cite{Hofmann1997syntax,Shulman2016categorical}. In our case, we cannot have in the semantics that all entailments (i.e. dinaturals) compose, and therefore our rules must be stated in such a way that cut is not admissible. In his seminal paper~\cite{Lawvere1969adjointness}, Lawvere introduced the categorical understanding of logic by viewing quantifiers/connectives as adjoints: we formulate (some of) the rules of our type theory with dinaturals precisely in Lawvere's ``adjoint-form'' (e.g.~\cite[4.1.7, 4.1.8]{Jacobs1999categorical}), i.e., as natural bijections between entailments. In standard accounts of logic this adjoint-form is equivalent to the usual intro/elim.\@ rules for connectives, but only in the presence of \emph{cut}; the key observation is that, despite the absence of a general cut rule, the rules for quantifiers/connectives in adjoint-form \emph{can} be validated in our semantics with dinaturals. %
\item {\emph{(Co)ends-as-quantifiers.}} The rules for ends and coends are reminiscent of the quantifiers-as-adjoints paradigm by Lawvere~\cite{Lawvere1969adjointness}, which we captured as ``right and left adjoint'' operations to weakening~\cite[1.9.1]{Jacobs1999categorical}. This adjointness relation should be only interpreted suggestively, since (co)ends are functorial operations for naturals but in general not dinaturals~\cite[1.1.7]{Loregian2021coend}. Our approach has the advantage that several properties of quantifiers, e.g., that they can be exchanged and permuted, follow automatically from certain \emph{structural properties of contexts}. For example, in first-order logic the formulas $\forall x.\forall y.P \Leftrightarrow  \forall y.\forall x.P \Leftrightarrow \forall (x,y). P$ are logically equivalent for any predicate $P$: this is indeed also verified for ends (and coends with existentials), and takes the name of ``Fubini rule''~\cite[IX.8]{MacLane1998categories},~\cite[1.3.1]{Loregian2021coend}, which we prove in \Cref{ex:fubini}. More details on (co)ends and their calculus can be found in~\cite[IX.5-6]{MacLane1998categories},~\cite[Ch. 1]{Loregian2021coend}.
\item {\emph{Dinaturality.}} Dinatural transformations do not compose in general~\cite{Santamaria2019towards}: this lack of general composition turns out not to be a problem in practice, since they \emph{do} compose in all examples of interest. In such cases, dinaturals compose essentially because they compose with other \emph{natural} transformations~\cite{Dubuc1970dinatural}, and we capture this in our system by providing two \emph{restricted} cut rules.
\end{enumerate}%
\noindent Because of the lack of general compositionality, we do not consider a categorical semantics of our type theory using standard categorical models, e.g., fibrations~\cite{Jacobs1999categorical} or categories with families~\cite{Castellan2020categories}, since they all ask for full composition, which cannot be guaranteed in our semantics. %
Hence, our approach is to simply consider the main rules described in \Cref{fig:syntax:entailments} (which have \emph{restricted} rules for composition of entailments) and prove soundness w.r.t. the category model with dinaturals.

We formalize the soundness theorems given in this paper about dinaturality using the Agda proof assistant and the \href{https://github.com/agda/agda-categories}{\texttt{agda-categories}} library. Whenever present, the symbol \agda\ links to formal proofs, reporting here just the core idea. The full Agda formalization is accessible at \cite{Laretto2025artifact} and \href{https://github.com/iwilare/dinaturality}{\texttt{https://github.com/iwilare/dinaturality}}. We refer to \cite{Laretto2025di} for the Appendix of this paper.

\subsection{Related Work}

Directed type theory has been approached in several (mutually incompatible) ways, with different methodological choices regarding semantics and rules for directed equality, but without ever investigating the connection to dinaturality.

\noindent \textbf{Directed type theory with groupoids.}
North~\cite{North2019towards}, Altenrkirch and Neumann~\cite{Altenkirch2024synthetic} describe a dependent directed type theory with semantics in the category of (small) categories $\Cat$, but using
groupoidal structure to deal with the problem of variance in both introduction and elimination rules for directed equality. This line of research has been recently expanded in~\cite{Chu2024directed,Chu2025dependent} by extending judgments with variance annotations.

We focus on non-dependent semantics, and avoid groupoids by tackling the variance issue with dinatural transformations; using dinaturality and (co)ends-as-quantifiers allow us to capture naturality for free and characterize natural transformations inside of the type theory.

\noindent \textbf{Directed type theory, judgmental models.} Another approach to modeling directed equality is at the judgmental level. This line of research started with Licata and Harper~\cite{Licata20112} who introduced a directed type theory with a model in $\Cat$. Since directed equality is treated judgmentally, there are no rules governing its behavior in terms of elimination and introduction principles, although variances are similarly used in context as in our approach. Ahrens et al.~\cite{Ahrens2023bicategorical} similarly identify a type theory with judgmental directed equalities and semantics in comprehension bicategories, and extensively compare previous works on both judgmental and propositional directed type theories.

\noindent \textbf{Logics for category theory.} New and Licata~\cite{New2023formal} give a sound and complete presentation for the internal language of (hyperdoctrines of) certain virtual equipments. These models capture enriched, internal, and fibered categories, and have an intrinsically directed flavor. In these contexts, the type theory can give synthetic proofs of Fubini, Yoneda, and Kan extensions as adjoints. This generality however comes at the cost of a non-standard syntactic structure of the logic, especially when compared to standard \MLTT. Directed equality elimination takes the shape of the (horizontal) identity laws axiomatized in virtual equipments~\cite{Cruttwell2010unified}, which in the $\Prof$ model is essentially the coYoneda lemma. Their quantifiers are given by tensors and (left/right) internal homs, which in $\Prof$ correspond to certain restricted (co)ends which always come combined with the tensors and internal homs of $\Set$.

Our work is similar in spirit in that we provide a formal setting for proving category theoretical theorems using logical methods; we only focus on the elementary 1-categorical model of categories and do not yet capture enriched and internal settings.
However, we treat (co)ends as quantifiers \emph{directly}, viewing them as operations which act on the variables of the context, without the need for them to include any conjunction or implication. Our rules for directed equality are more direct and reminiscent of standard Martin-L\"of type theory, and specifically focus on the semantic justification of dinaturality. Since we consider less general models, our contexts do not have any linear nor ordered restriction and the same variable can appear multiple times both in equalities and contexts: for example, this allows us to \emph{write down} the statement of symmetry (without being able to prove it), and to consider profunctors of arbitrary variables, as typically needed in (co)end calculus.

\noindent \textbf{Geometric models of directed type theory.} Riehl and Shulman~\cite{Riehl2017type} introduce a simplicial type theory for synthetic $(\infty,1)$-categories. A directed interval type is axiomatized in a style reminiscent of cubical type theory~\cite{Cohen2015cubical}, which allows a form of (dependent) Yoneda lemma to be derived using such identity type. This type theory has been implemented in practice in the Rzk proof assistant~\cite{Kudasov2024formalizing}. On this line of research, Weaver and Licata~\cite{Weaver2020constructive} present a \emph{bicubical} type theory with a directed interval and investigate a directed analog of the univalence axiom; the same objectives were recently advanced in Gratzer et al.~\cite{Gratzer2024directed,Gratzer2025yoneda} with triangulated type theory and modalities.

In comparison with the above works, we do not explore the geometrical interpretation of directedness and focus on ``algebraic'' 1-categorical semantics; moreover, our treatment of directed equality is done intrinsically with elimination rules as in \MLTT rather than with synthetic intervals, with semantics directly provided by $\hom$-functors.

\noindent \textbf{Coend calculus, formally.} Caccamo and Winskel~\cite{Caccamo2001higher} give a formal system in which one can work with coends and establish non-trivial theorems with a few syntactical rules. The flavor is explicitly that of an axiomatic system, and does not take inspiration from type-theoretic rules: for instance, presheaves are \emph{postulated} to be equivalent under the swapping of quantifiers (Fubini), so this principle is not derived from structural rules as typically done in a logical presentation. %

\subsection{Structure of the Paper}

We start in \Cref{sec:syntax_and_semantics} by describing syntax and judgmental structure of the type theory, and give examples of directed type theory in \Cref{sec:examples_directed_equality}. We recall notions about dinaturals in \Cref{sec:dinaturality}, which we then use for the semantics in \Cref{sec:semantics}. We then apply our type theory to give logical proofs of theorems in category theory in \Cref{sec:examples_coend_calculus}, concluding in \Cref{sec:conclusion} with future works.

\section{Syntax}
\label{sec:syntax_and_semantics}

We introduce the main syntactic judgments of our proof-relevant first-order directed type theory, for which we describe the main ideas and notation in \Cref{sec:notation,sec:rules}.

Our type theory is structured in a similar way to first-order logic~\cite[4.1]{Jacobs1999categorical}, with judgments for types and terms (i.e., sorts and function symbols), and predicates indexed by a term context.

We will omit several uninteresting equality judgments for contexts, terms, propositional contexts, as well as usual congruence and equivalence rules. We list here the main judgments of our type theory alongside a brief description of their semantics to aid intuition, with details in \Cref{sec:semantics}.

\vspace{0.6em}

\noindent
\ \textbf{\Cref{fig:syntax:types}:}
$\left\{
\begin{minipage}{0.86\linewidth}
\begin{itemize}[leftmargin=10pt]
    \item \fbox{$\C\type$}~\textbf{types} $\C,\D$ are interpreted in the semantics as small categories. Types can have $-^\op$, and include the terminal $\top$, product $\C \times \D$, and functor categories $[\C,\D]$.
    \item \fbox{$\C = \D$}~\textbf{judgmental equality of types}, interpreted as isomorphisms of categories; we use this to simplify $({\C^\op})^\op = {\C}$ and propagate the $\op$ inside types.
\end{itemize}
\end{minipage}
\right.$

\noindent
\ \textbf{\Cref{fig:syntax:terms}:}
$\left\{
\begin{minipage}{0.85\linewidth}
\begin{itemize}[leftmargin=*]
    \item \fbox{$\Gamma \ctx$}~\textbf{contexts} $\Gamma,\Delta$ are finite lists of categories, interpreted as \emph{products in $\Cat$s};
    \item \fbox{$\Gamma \ni x:\C$}~\textbf{variable in context}, which captures the de Bruijn indices of variables in context $\Gamma$; for us variable names are irrelevant, and we always identify variables with these judgments. Semantically, these are the projections out of $\sem{\Gamma}$.
    \item \fbox{$\Gamma \vdash F : \C$}~\textbf{terms} $F,G$ as \emph{functors} $\sem{\Gamma}\>\sem{\C}$, which are similar to terms in STLC;
\end{itemize}
\end{minipage}
\right.$

\noindent
\ \textbf{\Cref{fig:syntax:predicates}:}
$\left\{
\begin{minipage}{0.85\linewidth}
\begin{itemize}[leftmargin=*]
    \item \fbox{$[\Gamma]\ P \prop$}~\textbf{predicates} $P,Q$ as \emph{dipresheaves}, i.e., functors $\sem P\!:\!\sem{\Gamma}^\op\x\sem{\Gamma}\!\>\!\Set$;
    \item \fbox{$[\Gamma]~\Phi \propctx$}~\textbf{propositional contexts} $\Phi,\Phi'$ are finite lists of predicates, which we interpret via the \emph{pointwise product of dipresheaves in $\Set$};
\end{itemize}
\end{minipage}
\right.$

\noindent
\ \textbf{\Cref{fig:syntax:entailments}:}
$\left\{
\begin{minipage}{0.85\linewidth}
\begin{itemize}[leftmargin=*]
    \item \fbox{$[\Gamma]~\Phi \vdash \alpha : P$}~\textbf{entailments} $\alpha,\beta,\gamma$ are interpreted semantically as \emph{dinatural transformations} $\sem{\Phi}~\dinat \sem{P}$; we axiomatize composition/cut only with \emph{natural} transformations, not requiring general composition;
    \item \fbox{$[\Gamma]~\Phi \vdash \alpha = \beta : P$}~\textbf{equality of entailments}, i.e. \emph{equality of dinaturals in $\Set$}.
\end{itemize}
\end{minipage}
\right.$

\vspace{0.7em}

For predicates we consider the following logical connectives, which we denote syntactically with the same symbol later used in the semantics:
\begin{itemize}[leftmargin=*]
    \item \textbf{conjunction} $\Product$, interpreted as the pointwise product of dipresheaves in $\Set$;
    \item \textbf{polarized implication} $\Implication$, by postcomposing dipresheaves with $\hom_\Set : \Set^\op\x\Set\>\Set$;
    \item \textbf{propositional directed equality} $\hom_\C$ is interpreted by \emph{$\hom$-functors} $: \Cop\x\C\>\Set$;
    \item \textbf{universal and existential quantifiers} $\Endf{x:\C} P(\nxx)$, $\Coendf{x:\C}\!P(\nxx)$ are given by \emph{ends} and \emph{coends}. %
\end{itemize}
The judgments for types, terms, propositions and entailments are given in \Cref{fig:syntax:types,fig:syntax:terms,fig:syntax:predicates,fig:syntax:entailments}.

Our directed type theory is equipped with an equational theory for entailments, which we describe the key features of in \Cref{sec:rules} without spelling it out in detail. The most important cases are given in \Cref{fig:syntax:entailments} for directed equality, with details in \CiteAppendix{A}{appendix:syntax} for cuts/adjoint rules.

\begin{figure}[H]
  \centering
  \captionsetup{justification=centering}
  \[
  \begin{array}{c}
  \fbox{$\C \type$}
  \quad \inference{C \in \Sigma_B}{C \type}
  \quad \inference{\C \type}{\C^\op \type}
  \quad \inferenceTwo{\C \type}{\D \type}{\C \times \D \type}
  \quad \inferenceTwo{\C \type}{\D \type}{[\C,\D] \type}
  \quad \inference{\phantom{\C\!\!\type}}{\top \type}
  \\[1.5em]
  \fbox{$\C = \D$}
  \quad \inferenceZero{(\C^\op)^\op = \C}
  \quad \inferenceZero{(\C \times \D)^\op = \C^\op \times \D^\op}
  \quad \inferenceZero{[\C,\D]^\op = [\C^\op,\D^\op]}
  \quad \inferenceZero{\top^\op = \top}
  \quad \cdots
  \end{array}
  \]
  \vspace{-0.8em}
  \caption{Syntax of first-order dinatural directed type theory -- types  and judgmental equality.}
  \label{fig:syntax:types}
  \centering
  \vspace{0.9em}
  \begin{adju}[1.0]\hspace{-1.8em}$\begin{array}{c}
  \fbox{$\Gamma \ctx$}
  \quad \inference{\phantom{\C\!\!\type}}{[] \ctx}
  \quad \inferenceTwo{\Gamma \ctx}{\C \type}{\Gamma,\C \ctx}
  \quad \inference{\Gamma \ctx}{\Gamma^\op \ctx}
  \\[1.5em]
  \quad \fbox{$\Gamma = \Gamma'$}
  \quad \inferenceZero{[]^\op = []}
  \quad \inferenceZero{(\Gamma,\C)^\op = \Gamma^\op,\C^\op}
  \quad \inferenceTwo{\C = \C'}{\Gamma = \Gamma'}{\Gamma,\C = \Gamma',\C'}
  \\[1.5em]
  \fbox{$\Gamma \ni x:\C$}
  \quad \inference{\phantom{\phantom{\C\!\!\type}}}{\Gamma,x:\C \ni x:\C}
  \quad \inference{\Gamma \ni x:\C}{\Gamma,y:\D \ni x:\C}
  \\[1.5em]
  \fbox{$\Gamma \vdash t : \C$}
  \quad \inference{\Gamma \ni x:\C}{\Gamma \vdash x:\C}
  \quad \inference{\Gamma \vdash t : \C}{\Gamma^\op \vdash t^\op : \C^\op}
  \quad \inferenceTwo{f \in \Sigma_T}{\Gamma \vdash t : \textsf{dom}(f)}{\Gamma \vdash f(t) : \textsf{cod}(f)}
  \\[1.5em]\quad \inference{\phantom{\Gamma\!\!\!\vdash s : \C}}{\Gamma \vdash\ ! : \top}
  \quad \inferenceTwo{\vphantom{p}\Gamma \vdash s : \C}{\Gamma \vdash t : \D}{\Gamma \vdash \ang{s,t} : \C \times \D}
  \quad \inference{\Gamma \vdash p : \C \times \D}{\Gamma \vdash \pi_1(p) : \C}
  \quad \inference{\Gamma \vdash p : \C \times \D}{\Gamma \vdash \pi_2(p) : \D}
  \\[1.5em]
  \quad \inferenceTwo{\Gamma \vdash s : [\C,\D]}{\Gamma \vdash t : \C}{\Gamma \vdash s \cdot t : \D}
  \quad \inference{\Gamma, x:\C \vdash t(x) : \D}{\Gamma \vdash \lambda x.t(x) : [\C,\D]}
  \\[1.5em]
  \fbox{$\Gamma \vdash t = t' : \C$}
  \quad \inferenceTwo{\Gamma, x:\C \vdash f(x) : \D}{\Gamma \vdash t : \C}{\Gamma \vdash (\lambda x.f(x)) \cdot t = f[x \mapsto t] : \D}
  \quad \inference{\Gamma, x:\C \vdash f(x) : \D}{\Gamma, x:\C \vdash (\lambda x.f(x)) \cdot x = f(x) : \D}
  \\[1.5em]
  \quad \inference{\Gamma \vdash p : \C \times \D}{\Gamma \vdash \ang{\pi_1(p),\pi_2(p)} = p : \C \times \D}
  \quad \inference{\Gamma \vdash t : \top}{\Gamma \vdash t =\ ! : \top}
  \quad \inferenceTwo{\Gamma \vdash s : \C}{\Gamma \vdash t : \D}{\Gamma \vdash \pi_1(\ang{s,t}) = s : \C}
  \quad \inferenceTwo{\Gamma \vdash s : \C}{\Gamma \vdash t : \D}{\Gamma \vdash \pi_2(\ang{s,t}) = t : \D}
  \\[1.5em]
  \nowidth[c]{
  \quad
  \quad
  \inference{\Gamma \vdash t : \C}{\Gamma \vdash (t^\op)^\op = t : \C}
  \quad \inferenceTwo{\Gamma \vdash s : \C}{\Gamma \vdash t : \D}{\Gamma^\op \vdash \langle s, t \rangle^\op = \langle s^\op , t^\op \rangle : \C^\op \times \D^\op}
  \quad \inference{\Gamma^\op,x:\C \vdash t : \D}{\Gamma \vdash (\lambda x. t(x))^\op = \lambda x.t^\op(x) : [\C^\op,\D^\op]}
  }
  \end{array}$
  \end{adju}
  \caption{Syntax of first-order dinatural directed type theory -- contexts, variables, terms and their equality.}
  \label{fig:syntax:terms}
  \[
  \begin{array}{c}
\fbox{$[\Gamma]\ P \prop$}
\hspace{1em}
\inferenceTwo{[\Gamma]\ P \prop}{[\Gamma]\ Q \prop}{[\Gamma]\ P \times Q \prop}
\quad \inferenceTwo{[\Gamma^\op]\ P \prop}{[\Gamma]\ Q \prop}{[\Gamma]\ P \Rightarrow Q \prop}
\quad \inference{}{[\Gamma]~\top \prop}
\\[1.5em]
\quad \inferenceTwo{\vphantom{\Sigma_P}\Gamma^\op,\Gamma \vdash s : \Cop}{\Gamma^\op,\Gamma \vdash t : \C}{[\Gamma]\ {\hom}_\C(s, t) \prop}
\quad \inferenceThree{P \in \Sigma_P}{\vphantom{\Sigma_P}\Gamma^\op,\Gamma \vdash s : \textsf{neg}(P)^\op}{\Gamma^\op,\Gamma \vdash t : \textsf{pos}(P)}{[\Gamma]\ {P(s \mid t) \prop}}
\\[1.5em]
\quad \inference{[\Gamma, x:\C]\ P(\n x,x) \prop}{[\Gamma]~\Endf{x:\C} P(\n x,x) \prop}
\quad \inference{[\Gamma, x:\C]\ P(\n x,x) \prop}{[\Gamma]~\Coendf{x:\C} P(\n x,x) \prop}
\\[1.5em]
\fbox{$\Phi \propctx$} \quad \inferenceZero{\emptyctx \propctx} \quad \inferenceTwo{\Phi \propctx}{P \prop}{P, \Phi \propctx}
\hspace{1em}
\end{array}
\]
\vspace{-0.4em}
\Description{Main rules for the syntax of first-order dinatural directed type theory: types and their equality; contexts, variables, terms and their equality; predicates and propositional contexts.}
\caption{Syntax of first-order dinatural directed type theory -- predicates and propositional contexts.}
\label{fig:syntax:predicates}
\end{figure}

\begin{figure}[H]
  \[
  \begin{array}{c}
\fbox{$\Gamma \ni x:\A \mathsf{\ cov\ in\,}P$}
\\[1em]
\inferenceTwo{\Gamma \ni x:\A \mathsf{\ cov\ in\,}P}{\Gamma \ni x:\A \mathsf{\ cov\ in\,}Q}{\Gamma \ni x:\A \mathsf{\ cov\ in\,}P \times Q}
\quad \inferenceTwo{\Gamma^\op \ni x:\A^\op \mathsf{\ cov\ in\,} P}{\Gamma \ni x:\A \mathsf{\ cov\ in\,}Q}{\Gamma \ni x:\A \mathsf{\ cov\ in\,}P \Rightarrow Q}
\\[1em]
\quad \inferenceTwo{\Gamma^\textsf{op},\Gamma \ni \n x:\Aop \mathsf{\ unused\ in\,} s : \Cop}{\Gamma^\textsf{op},\Gamma \ni \n x:\A^\textsf{op} \mathsf{\ unused\ in\,} t : \C}{\Gamma \ni x:\A \mathsf{\ cov\ in\,}\hom_\C(s, t)}
\\[1em]
\inferenceTwo{\Gamma^\textsf{op},\Gamma \ni \n x:\Aop \mathsf{\ unused\ in\,} s : \textsf{neg}(P)^\op}{\Gamma^\textsf{op},\Gamma \ni \n x:\A^\textsf{op} \mathsf{\ unused\ in\,} t : \textsf{pos}(P)}{\Gamma \ni x:\A \mathsf{\ cov\ in\,}P(s \mid t)}
\\[1em]
\fbox{$\Gamma \ni x:\A \mathsf{\ contra\ in\,}P$}
\quad
\inference{\Gamma^\op \ni x:\A^\op \mathsf{\ cov\ in\,} P^\op}{\Gamma \ni x:\A \mathsf{\ contra\ in\,} P}
\\[1em]
\fbox{$\Gamma \ni x:\A \mathsf{\ unused\ in\,}t:\C$}
\quad \inferenceTwo{\Gamma \ni x : \C}{x \not = y}{\Gamma \ni y:\C \mathsf{\ unused\ in\,} x:\C}
\quad \inference{\Gamma,z:\C \ni x:\A \mathsf{\ unused\ in\,} t(z) : \D}{\Gamma \ni x:\A \mathsf{\ unused\ in\,} \lambda z.t(z) : [\C,\D]}
\\[1em]
\quad \inference{\Gamma \ni x:\A \mathsf{\ unused\ in\,} t : \dom(f)}{\Gamma \ni x:\A \mathsf{\ unused\ in\,} f(t) : \cod(f)}
\quad \inference{\Gamma \ni x:\A \mathsf{\ unused\ in\,} t : \C}{\Gamma^\op \ni x:\Aop \mathsf{\ unused\ in\,} t^\op : \C^\op}
\end{array}
\]
\caption{Syntax of first-order dinatural directed type theory -- syntactic conditions for covariant/contravariant variables in predicates. Full rules in \CiteAppendix{A}{appendix:syntax}.} %
\vspace{-0.8em}
\label{fig:syntax:pos_neg_conditions_formulas}
\vspace{0.9em}%
\begin{adju}[1.0]%
\def\arraystretch{2.3}%
$\begin{array}{c}
    \hspace{-3.7em}
    {\fbox{$[\Gamma]~\Phi \vdash \alpha : P$}}
    \quad
    \begin{prooftree}
    \infer0[\Rulevar]
         {[\Gamma]~\Phi,a:P,\Phi' & \vdash a : P}
    \end{prooftree}
    \quad
    \begin{prooftree}
    \hypo{[\Gamma]~\Phi & \vdash \alpha : Q}
    \infer1[\Ruleweaken]{[\Gamma]\ P, \Phi & \vdash \textsf{wk}_P(\alpha) : Q}
    \end{prooftree}
    \quad
    \begin{prooftree}
    \infer0[\Ruletop]{[\Gamma]~\Phi & \vdash {!} : \top}
    \end{prooftree}
    \\[0.8em]
    \begin{prooftree}
    \hypo{\Gamma^\op,\Gamma & \vdash F : \C}
    \hypo{[x: \C,\Gamma]\ {\Phi(\nxx)} & \vdash \alpha : Q(\nxx) }
    \infer2[\Rulereindex]{[\Gamma]~\Phi(F(x,\n x),F(\n x,x)) & \vdash F^*(\alpha) : Q(F(x,\n x),F(\n x,x))}
    \end{prooftree}
    \quad
    \begin{prooftree}
    \hypo{[\Gamma]~P, P, \Phi & \vdash \alpha : Q}
    \infer1[\Rulecontract]{[\Gamma]\ P, \Phi & \vdash \textsf{contr}_P(\alpha) : Q}
    \end{prooftree}
    \\[0.8em]
    \begin{prooftree}
    \hypo{[\Gamma]~\Phi \vdash P \x Q }
    \infer[double]1[\Ruleprod]{[\Gamma]~\Phi \vdash P, \qquad [\Gamma]~\Phi \vdash Q}
    \end{prooftree}
    \quad
    \begin{prooftree}
    \hypo{[x:\Gamma]\ A(\overline x,x), \Phi(\overline x,x) & \vdash B(\overline x,x)}
    \infer[double]1[\Ruleexp]{[x:\Gamma]~\takespace{A(\overline x,x), \Phi(\overline x,x)}{\Phi(\overline x,x)} & \vdash A^\op(x,\overline x) \Rightarrow B(\overline x,x)}
    \end{prooftree}
    \\[0.8em]
    \begin{prooftree}
    \infer[no rule]0{[a: \C, \Gamma]~\Phi & \vdash P(\naa)}
    \infer[double]1[\Ruleend]{\takespace{[a: \C, \Gamma]}{[\Gamma]}\ {\Phi} & \vdash \Endf{a:\C} P(\naa)}
    \end{prooftree}
    \qquad
    \begin{prooftree}
    \infer[no rule]0{[\Gamma]\ {\left(\Coendf{a:\C} P(\naa)\right)\!, \Phi} & \vdash Q}
    \infer[double]1[\Rulecoendfrobenius]{[a: \C, \Gamma]\ {P(\naa), \Phi} & \vdash Q}
    \end{prooftree}
    \\[1.9em]
    \nowidth[c]{
    \begin{prooftree}
    \hypo{\Gamma \textsf{\ unused in } P}
    \infer[no rule]1{[a\!:\!\Delta^\op,b\!:\!\Delta]~\Phi(a,b) & \vdash \takespace[l]{\gamma[\alpha]}{\alpha}\!:\!P(a,b)}
    \infer[no rule]1{[z\!:\!\Delta]\ k\!:\!\takespace{P(a,b)}{P(\nzz)},\takespace[l]{\Phi(a,b)}{\Phi(\nzz)} & \vdash \takespace[l]{\gamma[\alpha]}{\gamma[k]}\!:\!Q(\nzz)}
    \infer1[\Rulecutdin]{[z\!:\!\Delta]~\takespace[l]{\Phi(a,b)}{\Phi(\nzz)} & \vdash \gamma[\alpha]\!:\!Q(\nzz)}
    \end{prooftree}
    \quad
    \begin{prooftree}
    \hypo{\Gamma \textsf{\ unused in } P}
    \infer[no rule]1{[z\!:\!\Delta]~\Phi(\nzz) & \vdash \takespace[l]{\gamma[\alpha]}{\gamma}\!:\!P(\nzz)}
    \infer[no rule]1{[a\!:\!\Delta^\op,b\!:\!\Delta]\ k\!:\!P(a,b),\takespace[l]{\Phi(\n a,\n b)}{\Phi(\n a,\n b)} & \vdash \takespace[l]{\gamma[\alpha]}{\alpha[k]}\!:\!Q(a,b)}
    \infer1[\Rulecutnat]{[z\!:\!\Delta]~\takespace[l]{\Phi(a,b)}{\Phi(\nzz)} & \vdash \alpha[\gamma]\!:\!Q(\nzz)}
    \end{prooftree}
    }
    \\[1.9em]
    \begin{prooftree}
    \infer0[\Rulerefl]{[x: \C,\Gamma]~\Phi & \vdash \refl_\C : \hom_\C(\nxx)}
    \end{prooftree}
    \quad
    \begin{prooftree}
    \hypo{[z : \C, \Gamma]~\takespace{e:\hom_\C(a,b), \Phi(\overline b, \overline a)}{\Phi(\nzz)}~\vdash & h : P(\nzz)}
    \infer1[\RuleJ]{[a : \Cop, b : \C, \Gamma]\ e:\hom_\C(a,b), \Phi(\overline b, \overline a) \vdash & \J(h)[e] : P(a,b)}
    \end{prooftree}
    \\[1.1em]
    \hspace{-4.1em}{\fbox{$[\Gamma]~\Phi \vdash \alpha = \beta : P$}}
    \quad
    \begin{prooftree}
    \infer0[\RuleJcomp]{[z:\C,\Gamma]\ k:\Phi(\nzz) \vdash \J(h)[\refl_\C] = h : P(\nzz)}
    \end{prooftree}
    \\[0.7em]
    \begin{prooftree}
    \infer[no rule]0{[z:\C,\Gamma]~\Phi(\nzz) \vdash \alpha[\refl_\C] = \beta[\refl_\C] : P(\nzz)}
    \infer1[\RuleJeq]{[a:\Cop,b:\C,\Gamma]\ e:\hom_\C(a,b),\Phi(\overline b,\overline a) \vdash \alpha[e] = \beta[e] : P(a,b)}
    \end{prooftree}
\end{array}$
\end{adju}
\LinkRulereindex
\LinkRuleexp
\LinkRuleexpinv
\LinkRuleweaken
\LinkRuleprod
\LinkRulevar
\LinkRulecontract
\LinkRulerefl
\LinkRulecoendfrobenius
\LinkRuleop
\LinkRuleend
\LinkRulecoend
\LinkRulecoendinv
\LinkRuletop
\LinkRulecutnat
\LinkRulecutdin
\LinkRuleJ
\LinkRuleJcomp
\LinkRuleJeq
\vspace{-0.4em}
\Description{Main rules for the syntax of first-order dinatural directed type theory: rules for entailments and judgmental equality.}
\caption{Syntax of first-order dinatural directed type theory -- entailments and judgmental equality.}
\label{fig:syntax:entailments}
\end{figure}

The rules for entailments implicitly use the notion of variance for variables, described in \Cref{notation:nat}. Variance is captured formally in \Cref{fig:syntax:pos_neg_conditions_formulas} by the following judgments, all of which presuppose $\Gamma \ni x:\A$ for a variable $x$ of type $\A$ in context $\Gamma$:

\vspace{0.6em}

\noindent
\ \textbf{\Cref{fig:syntax:pos_neg_conditions_formulas}:}
$\left\{
\begin{minipage}{0.85\linewidth}
\begin{itemize}[leftmargin=*]
    \item \fbox{$\Gamma \ni x:\A \mathsf{\ unused\ in\,}t:\C$} for $x\!:\!\A$ does not syntactically appear in a term\nowidth{\ t\text{.}}
    \item ${\fbox{$\Gamma \ni x:\A \mathsf{\ cov\ in\,}P$}}$ states that $x\!:\!\A$ is \emph{covariant} in the predicate $[\Gamma]\ P$.
    \item \fbox{$\Gamma \ni x:\A \mathsf{\ contra\ in\,}P$} states that $x\!:\!\A$ is \emph{contravariant} in the predicate \nowidth{[\Gamma]\ P\text{.}}
\end{itemize}
\end{minipage}
\right.$

\vspace{0.7em}

To make the type theory non-trivial, our judgments are implicitly parameterized by a standard notion of signature $\Sigma := (\Sigma_B,\Sigma_T,\Sigma_P,\Sigma_A)$, i.e., sets of base type symbols, term symbols, predicate symbols, and base entailments respectively. Base predicates $P(s \mid t)$ for $P \in \Sigma_P$ are equipped with \emph{two} terms, a negative one $s : \textsf{neg}(P)^\op$ and a positive one $t : \textsf{pos}(P)$ typed in the same context $\Gamma^\op,\Gamma$. This choice is motivated by the fact that $\hom$ is similarly equipped with two sides. The judgments for equality of types are not extended by the signature. We omit the details of this extension.

\subsection{Polarity and Variance}
\label{sec:notation}

The main idea behind dinatural transformations is that variables in a predicate are allowed to be used irrespectively of the $\op$ in their type (or lack thereof). To give a taste for our type theory, we show what the statement and proof of transitivity of directed equality look like in our system:
\[
\begin{prooftree}
\infer0[\Rulevar]{[z:\C,c:\C]~\takespace{f:\hom(a,b),\, g:\hom(\n b,c)}{g : \hom(\n z, c)} & \vdash g : \hom(\n z,c)}
\infer1[\RuleJ]{[a:\Cop,b:\C,c:\C]\ {f:\hom(a,b),\, g:\hom(\n b,c)} & \vdash {J(g)} : \hom(a,c)}
\end{prooftree}
\]
Whenever a variable $b:\C$ is used with the ``wrong polarity'' we denote such use with $\n b:\Cop$, as in the above example. In order to make this intuition precise, we formally introduce the concepts of \emph{position}, \emph{polarity}, and \emph{variance} and their notation in the type theory. Variance is ultimately used to implement the syntactic restriction of directed equality elimination $\RuleJ$.

We use the term \emph{polarity of a type} to refer to the fact that types always come in pairs: whenever $\C$ is a type, its opposite $\C^\op$ is also a type. Polarity is a relative notion: we say the type $\C^\op$ is \emph{the negative} of $\C$ irrespectively of the fact that $\C$ itself might have an outermost syntactic $\op$.

\noindent Polarity is used in the syntax of the type theory in the following way:

\begin{itemize}
\item The $\op$ operation is also present in contexts, i.e., for a $\Gamma \ctx$ there is a \emph{negative context} $\Gamma^\op$ which is definitionally equal to the context obtained by adding $\op$ to each type.
\item In the formation rule for $[\Gamma]\ {\hom}_\C(s,t)$ in \Cref{fig:syntax:predicates}, the term $s$ is given return type $\C^\op$.
\item In the formation rule for $[\Gamma]\ P \Rightarrow Q$ in \Cref{fig:syntax:predicates}, the predicate $P$ is given type in $\Gamma^\op$.
\end{itemize}

\noindent The other crucial idea of our system is the above-mentioned fact that variables can appear at the same time irrespectively of their polarity. This is implemented by the following ideas:

\begin{itemize}
\item There are two cases where variables can appear in a predicate, namely the base cases $[\Gamma]\ {\hom}_\C(s,t)$ and $[\Gamma]\ P(s \mid t)$, where the two terms $s,t$ can use the variables from $\Gamma$.
\item The key idea is that both $s,t$ are not given type in $\Gamma$, but in the \emph{context concatenation} $\Gamma^\op,\Gamma$.
\item Intuitively, this allows for variables to be used in $s,t$ also in the ``wrong way'' (with respect to the original polarity of the context $\Gamma$ in which $P$ is given type).
\end{itemize}
\noindent We give a specific name to the terms of this shape in concatenated contexts $\Gamma^\op,\Gamma$, since they also play a crucial role in reindexing.
\begin{definition}
A \emph{diterm} is a term of the form $\Gamma^\op,\Gamma \vdash t : \C$ for some context $\Gamma$.
\end{definition}
\noindent We now capture the above intuitive ideas behind polarity and variance with precise terminology.
\begin{definition}[Positions in a predicate]
The name \emph{position} refers to a point in which a variable $x\!:\!\C$ can appear in a predicate, e.g., there are four possible positions $x,y,z,w$ for variables to appear in the predicate $\hom_\C(x,y) \x P(z,F(w))$.
\end{definition}

\begin{definition}[Variant use of a variable]
\label[definition]{def:variant_use}
For any predicate $[\Gamma]\ P$ and a position of type $\Cop$ in $P$, we say that a variable $\Gamma \ni x:\C$ (with no $\op$) is \emph{used contravariantly in that position} iff the variable used in that position is taken from the \emph{left} side $\Gamma^\op$ (in the context concatenation $\Gamma^\op,\Gamma$), i.e., with type $\n x:\Cop$. Accordingly, we will always denote variables taken from such left side of the context with an overbar $\n x$. Similarly, given a position of type $\C$ in $P$ we say that a variable $\Gamma \ni x:\C$ is \emph{used covariantly in that position} iff it is taken from the \emph{right} side $\Gamma$ (i.e. in the usual way), which we denote without any overbar.
\end{definition}

The notation $\n x$ is suggestive of the fact that $\n x:\C^\op$ and $x:\C$ will be identified with the same value when using dinatural transformations in the semantics of entailments.

\begin{example}[Derivation of a predicate]\label{example:predicate}
We provide an example derivation of a predicate in context combining the previously introduced ideas of co/contravariant variables, for a term $x:\C \vdash F(x) : \D$.
\[
\begin{prooftree}
\infer0{\n x\!:\!\C,\n y\!:\!\D,x\!:\!\C^\op,y\!:\!\D^\op \vdash y \!:\! \Dop}
\infer0{\n x\!:\!\C,\n y\!:\!\D,x\!:\!\C^\op,y\!:\!\D^\op \vdash \overline x : \C}
\infer1{\n x\!:\!\C,\n y\!:\!\D,x\!:\!\C^\op,y\!:\!\D^\op \vdash F(\overline x) : \D}
\infer2{[x\!:\!\C^\op,y\!:\!\D^\op]\ {\hom}_\D(y,F(\overline x)) \prop}
\hypo{\cdots\ \vdash x:\C}
\infer1{[x\!:\!\C,y\!:\!\D]\ P(x) \prop}
\infer2{[x\!:\!\C,y\!:\!\D]\ {\hom}_\D(y,F(\overline x)) \Rightarrow P(x) \prop}
\end{prooftree}
\]
\end{example}

\begin{definition}[Variance of a variable]
\label[definition]{def:variance}
Variables can occur in multiple positions at the same time: we say that a variable $\Gamma \ni x:\C$ is \emph{covariant} in a predicate $[\Gamma]\ P$ \emph{iff} it is \emph{always used covariantly} in the positions of $P$, i.e., it is always picked from the right side $\Gamma$ of the context $\Gamma^\op,\Gamma$ and is hence always used ``correctly'' with respect to $\Gamma$. Similarly, a variable $\Gamma \ni x:\C$ is said to be \emph{contravariant} in a predicate $[\Gamma]\ P$ when it is \emph{always used contravariantly} in the positions of $P$, i.e., it is always picked from the left side $\Gamma^\op$ of the context $\Gamma^\op,\Gamma$ and is hence always used ``in the wrong way'' with respect to $\Gamma$. A variable is said to be \emph{natural} when it is either covariant or contravariant, i.e., it is consistently used with the same variance. A variable is said to be \emph{dinatural} or \emph{mixed-variance} \emph{iff} it is neither covariant nor contravariant, i.e., it occurs at least once covariantly and at least once contravariantly in a predicate.
\end{definition}

\begin{example}[Variance]
In the predicate $[x\!:\!\Cop,y\!:\!\C]\ {\hom}_\C(x,y)$, both $x$ and $y$ are covariant.
In $[x\!:\!\C,y\!:\!\C,z\!:\!\C]\ {\hom}_\C(\n x,y) \times \hom_\C(\n y,z)$ the variable $x$ is contravariant, $y$ is dinatural, and $z$ is covariant. In $[x\!:\!\Cop,z\!:\!\Cop]\ {\hom}_\C(\n x,z) \Rightarrow \hom_\C(z,\n x)$, $x$ is contravariant and $z$ is covariant. Finally, for a term $\C^\op \vdash F : \D$ (i.e., a ``contravariant functor''), $x$ is covariant in $[x\!:\!\C]\ {\hom}_\D(F^\op(x),x)$.
\end{example}

The above definitions capture the way that natural and dinatural usage of variables is referred to in practice. Formally, variance of variables in predicates is captured using the judgments in \Cref{fig:syntax:pos_neg_conditions_formulas}. The actual implementation of variance is slightly different from the description above, but they are equivalent: the judgment $\Gamma \ni x:\A \mathsf{\ cov\ in\,} P$ is derivable, i.e., the variable $x$ is covariant, when \emph{its contravariant counterpart $\n x$ is not syntactically used anywhere in the predicate}. This last aspect is itself captured by a straightforward judgment, described in \Cref{fig:syntax:pos_neg_conditions_formulas}, which underapproximates syntactic unusedness of variables in terms. The well-formedness of these judgments occasionally relies on the fact that $\Gamma \ni x:\A$ implies that $\Gamma^\op \ni x:\A^\op$, and similarly $\Gamma^\op,\Gamma \ni x:\A$ and $\Gamma^\op,\Gamma \ni \n x:\Aop$ in the intuitive way.

\begin{example}[Variance, formally]\label{ex:variance_formally}
We give an example of a formal derivation for covariance using the predicate in \Cref{fig:syntax:pos_neg_conditions_formulas}, assuming for simplicity that the predicate $P$ does not have any variables:
\[
\begin{prooftree}
\infer0{[\n x\!:\!\C,\n y\!:\!\D,x\!:\!\Cop,y\!:\!\Dop] \ni \n y : \D \textsf{ unused in } y}
\infer0{[\cdots] \ni \n y : \D \textsf{ unused in } F(x)}
\infer2{[x\!:\!\Cop,y\!:\!\Dop] \ni y : \Dop \textsf{ cov in } {\hom}_\D(y,F(\overline x))}
\infer0{\cdots\vphantom{[x\!:\!\Cop,y\!:\!\Dop] \ni y : \Dop \textsf{ cov in } {\hom}_\D(y,F(\overline x))}}
\infer2{[x\!:\!\C,y\!:\!\D] \ni y : \D \textsf{ cov in } {\hom}_\D(y,F(\overline x)) \Rightarrow P}
\end{prooftree}
\]
\end{example}

\begin{remark}[Notation for variance in predicates]
We indicate with $[x:\C,y,\D,\Gamma]\ P(\nxx,\nyy)$ the fact that a predicate $P$ \emph{can} depend on $x,y$ both co- and contravariantly; we will often omit in $P$ the (unrestricted) presence of variables coming from a context $\Gamma$.
When either variance is omitted, e.g., as in $P(x,\n y)$, the predicate must depend \emph{only} on $x$ and $\n y$, i.e., $x$ is \emph{covariant} and $\n y$ is \emph{contravariant} in $P$. Variance for entire contexts is intuitively denoted as $[y:\Gamma]\ P(y)$, i.e., all variables in $\Gamma$ are covariant.

Formally, these restrictions are captured using the predicates for variance of \Cref{def:variance}. We use this convention in the rules for entailments of \Cref{fig:syntax:entailments}.
\end{remark}

There are many choices for the system of variances presented so far: the one presented here is a simple setup that closely matches the intuition for contravariance typically used in mathematics, denoting variables as contravariant precisely when one expects it as shown in \Cref{example:predicate}.

Mnemonically, \emph{p}ositions have \emph{p}olarity, and \emph{v}ariables have \emph{v}ariance. \emph{Co}variant variables are ``\emph{co}mpliant'' and they are used as they are told, while \emph{contra}variant variables are ``\emph{contra}rian'' and always reject well-typing laws.

For any predicate $[\Gamma]\ P$, there is an associated \emph{opposite predicate} $[\Gamma^\op]\ P^\op$, defined by induction on the derivation of $P$, obtained intuitively by inverting the variance of variables in each position: i.e., whenever $x$ was used in some position, $\n x$ is used instead, and vice versa. This operation is used in the rule for polarized implication $\Ruleexp$, described in \Cref{sec:rules}, and to define contravariance in \Cref{fig:syntax:pos_neg_conditions_formulas}. Note that this operation on predicates is defined metatheoretically: types and terms are the only two judgments for which there is a $-^\op$ \emph{in the syntax}.

We start by first defining a metatheoretical operation on diterms that simply swaps contexts:

\begin{definition}[Context swap of a term]\label{def:inversion_term}
Given a diterm $\Gamma^\op,\Gamma \vdash t : \C$, we indicate with $\Gamma,\Gamma^\op \vdash t^{\textsf{ctxswap}} : \C$ the \emph{context swap} of $t$, which is the term derivation obtained in the intuitive way by swapping the left and right side of its context; for example, $\left(\n x:\D^\op, x:\D \vdash x : \D\right)^{\textsf{ctxswap}} = \left(\n x:\D, x:\D^\op \vdash \n x : \D\right)$, and $\left(\n x:\C^\op, x:\C \vdash F(\n x) : \D\right)^{\textsf{ctxswap}} = \left(\n x:\C, x:\Cop \vdash F(x) : \D\right)$ for some term $\n x:\C^\op, x:\C \vdash F(x) : \D$. Crucially, the return type of the term does not change, which would be the case with the $t^\op$ operation internal to the syntax. Effectively this operation only rearranges the de Bruijn indices of variables, which is what the judgments for variance in \Cref{fig:syntax:pos_neg_conditions_formulas} use to detect co/contravariance.
\end{definition}

\begin{definition}[Opposite predicate]\label{def:inversion}
Given a predicate $[\Gamma]\ P$, there is a predicate in context $\Gamma^\op$ called \emph{the opposite of $P$} defined by (metatheoretical) induction on derivations of predicates:
  \[
  \begin{array}{l}
    -^\mathsf{op} : \set{[\Gamma]\ {-} \prop} \to \set{[\Gamma^\op]\ {-} \prop} \\
    \left( \top \right)^\mathsf{op} := \top \\
    \left( P \Rightarrow Q \right)^\mathsf{op} := P^\mathsf{op} \Rightarrow Q^\mathsf{op} \\
    \left( P \times Q \right)^\mathsf{op} := P^\mathsf{op} \times Q^\mathsf{op} \\
    \left( P(s \mid t) \right)^\mathsf{op} := P(s^{\textsf{ctxswap}} \mid  t^{\textsf{ctxswap}}) \\
    \left( \hom_\C(s, t) \right)^\mathsf{op} := \hom_{\C}(s^{\textsf{ctxswap}}, t^{\textsf{ctxswap}}) \\
    \textstyle \left( \Coendf{x:\C} P(\n x,x) \right)^\mathsf{op} := \Coendf{x:\C^\op} P(\n x,x)^\op \\
    \textstyle \left( \Endf{x:\C} P(\n x,x) \right)^\mathsf{op} := \Endf{x:\C^\op} P(\n x,x)^\op
  \end{array}
  \]
This operation can similarly be defined by inverting the polarity of a single variable: given a predicate $[x:\C,\Gamma]\ P(\n x,x)$ we denote with $[x:\Cop,\Gamma]\ P^{x \mapsto \op}(x,\n x)$ the predicate obtained by inverting the polarity of each position in $P$ where $x$ is used. A similar definition can be extended on propositional contexts $\Phi$. All these operations on predicates are clearly involutive.
\end{definition}

\begin{example}\label{example:predicate_inversion}
Taking the predicate of \Cref{example:predicate} and applying the predicate inversion operation $\left(\hom_\D(y,F(\overline x))\right)^\op$ produces the following derivation:
\[
\begin{prooftree}
\infer0{\n x\!:\!\C^\op,\n y\!:\!\D^\op,x\!:\!\C,y\!:\!\D \vdash \n y \!:\! \Dop}
\infer0{\n x\!:\!\C^\op,\n y\!:\!\D^\op,x\!:\!\C,y\!:\!\D \vdash x : \C}
\infer1{\n x\!:\!\C^\op,\n y\!:\!\D^\op,x\!:\!\C,y\!:\!\D \vdash F(x) : \D}
\infer2{[x\!:\!\C,y\!:\!\D]\ {\hom}_\D(\n y,F(x)) \prop}
\end{prooftree}
\]
\end{example}

The judgment for contravariance $\Gamma \ni x:\A \mathsf{\ contra\ in\,} P$ in \Cref{fig:syntax:pos_neg_conditions_formulas} is defined in terms of the covariant one and the notion of opposite predicate $P^\op$. Note that the well-formedness of this judgment relies on the fact that $\Gamma \ni x:\C$ implies $\Gamma^\op \ni x:\C^\op$.

\begin{example}[Contravariance, formally]
We give an example of a formal derivation for contravariance, following \Cref{ex:variance_formally}:
\[
\begin{prooftree}
\infer0{[\cdots] \ni \n x : \Cop \textsf{ unused in } \n y}
\infer0{[\n x\!:\!\Cop,\n y\!:\!\Dop,x\!:\!\C,y\!:\!\D] \ni \n x : \Cop \textsf{ unused in } x : \D}
\infer1{[\n x\!:\!\Cop,\n y\!:\!\Dop,x\!:\!\C,y\!:\!\D] \ni \n x : \Cop \textsf{ unused in } F(x) : \D}
\infer2{[x\!:\!\C,y\!:\!\D] \ni x : \C \textsf{ cov in } {\hom}_\D(\n y,F(x))}
\infer0{\cdots\vphantom{[x\!:\!\C,y\!:\!\D] \ni x : \C \textsf{ cov in } {\hom}_\D(\n y,F(x))}}
\infer2{[x\!:\!\Cop,y\!:\!\Dop] \ni x : \Cop \textsf{ cov in } {\hom}_\D(\n y,F(x)) \Rightarrow P}
\infer1{[x\!:\!\C,y\!:\!\D] \ni x : \C \textsf{ contra in } {\hom}_\D(y,F(\overline x)) \Rightarrow P}
\end{prooftree}
\]
\end{example}

\subsection{Rules}\label{sec:rules}

We now describe and give intuition for the main rules for entailments of our type theory in \Cref{fig:syntax:pos_neg_conditions_formulas}.

\begin{remark}[Notation for entailments]
    \label[remark]{notation:nat}
We use type-theoretic notation for entailments,
\[[x:\C, y:\D, ...]\ a:P(\overline x,x,\overline y,y,...),b:Q(\nxx,\nyy,...), ... \vdash \alpha[a,b,...] : R(\overline x,x,\overline y,y,...)\]
where we give names to each assumption in the list $\Phi := P,Q,...$. We overload square brackets $\alpha[a,b,...]$ both to indicate the assumptions and to denote composition of entailments in $\Rulecutdin$ and $\Rulecutnat$.
\end{remark}

\noindent Some of our rules are formulated in ``adjoint-form'' (e.g.~\cite[4.1.7, 4.1.8]{Jacobs1999categorical}), i.e., as {natural} \emph{bijections} between entailments. We use double lines in \Cref{fig:syntax:entailments} to indicate such isomorphisms of entailments, using judgmental equality of entailments to ensure that one direction is the inverse of the other. Naturality coincides with the fact that these isomorphisms commute with (both) the cut rules in the equational theory whenever possible: we use this in \Cref{sec:examples_coend_calculus} for the Yoneda technique. We give a spelled-out example of adjoint-form in \CiteAppendix{A}{appendix:syntax} for the $\Ruleend$ rule, describing precisely the naturality requirement for the rules in such form.

\begin{itemize}[leftmargin=*]
\item {\textbf{Structural rules.}} The rules $\Rulevar$, $\Ruleweaken$, $\Rulecontract$ capture the usual structural rules for assumptions, weakening, and contraction.
\item {\textbf{Products.}} The rule $\Ruleprod$ for conjunction $P \times Q$ is standard: reading the rule top-to-bottom, given a proof $[\Gamma]~\Phi \vdash P \times Q$ one can extract a proof $[\Gamma]~\Phi \vdash P$. Similarly, given two entailments with type $P$ and $Q$ in the same context one obtains an entailment with type $P \times Q$.
\item {\textbf{Polarized implication.}} Implication $\Ruleexp$ is similarly captured via the adjoint formulation, with a catch regarding polarity: the key idea is that a predicate $P(\n x,x)$ can be curried from one side to the other of the entailment \emph{by reversing the variance of all its variables}, i.e., using $P^\op$. Contrary to naturals and presheaves~\cite{Leinster2014basic}, dinaturals can be curried directly via the $\Ruleexp$ rule by currying each component of $\alpha$ in $\Set$. A similar idea is described in~\cite{Girard1992normal,Bainbridge1990functorial} as \emph{twisted exponential}.
\item {\textbf{(Co)ends.}} The rules $\Ruleend, \Rulecoendfrobenius$ capture the directed quantifiers of our type theory, i.e., (co)ends. These are also characterized in ``adjoint-form'', following precisely the same formulation of \cite[4.1.8]{Jacobs1999categorical}. Note that $\Phi$ is given type in $\Gamma$, and we do not make this weakening explicit.
\item {\textbf{Reindexing.}} Following the doctrinal presentation of logic (see~\cite{Jacobs1999categorical,Pitts1995categorical} for standard accounts), variables in entailments can be substituted with terms using the rule $\Rulereindex$: in particular, entailments can be substituted with \emph{diterms}, i.e., terms that are allowed to access the \emph{whole concatenation of contexts} $\Gamma^\op,\Gamma$. The fact that $F$ is a \emph{di}term is not a mere technical point, and it is used in \Cref{thm:dinat_collapse,thm:dinat_collapse,thm:op_of_entailments} to derive certain non-trivial structural rules related to variance.
\item {\textbf{Cut naturals-dinaturals.}} We present two restricted cut rules $\Rulecutdin,\Rulecutnat$ that allow entailments to be composed together. Associativity and identities for these is detailed in \CiteAppendix{A}{appendix:syntax}, along with a coherence condition that makes the two cuts agree whenever both entailments are \emph{naturals}. The occurrences $\na,\nb$ in $\Phi$ in $\Rulecutnat$ are needed to make sure that, in the semantics, $\alpha$ is natural in $a,b$ when the domain is \emph{just} $P$, i.e., by using $\Ruleexp$ to move $\Phi$ and invert the variance of $\n a,\n b$. Similarly, $P$ must also not syntactically depend on $\Gamma$ to ensure naturality in $a,b$, but both $\Phi$ and $Q$ can depend on $\Gamma$ without any restriction; we elaborate on this in the semantics of cuts in \Cref{sec:semantics}, which we use to state the naturality requirement for, e.g., ends in \CiteAppendix{A}{appendix:syntax}.
\item {\textbf{Directed equality elimination.}} The operational meaning behind $\RuleJ$ is the following: having identified two \emph{covariant} positions $a\!:\!\Cop$ and $b\!:\!\C$ in the predicate $P$, if there is a directed equality $\hom_\C(a,b)$ in context then it is enough to prove that $P$ holds ``on the diagonal'', where the two positions have been collapsed with the same dinatural variable $z:\C$; moreover, $a,b$ can be collapsed together in the context $\Phi$ \emph{only if they appear contravariantly}, i.e., as $\overline a$ and $\overline b$.
\item {\textbf{Dependent $\hom$ elimination.}} A \emph{dependent} version of directed $\J$, rule $\RuleJeq$, is needed to prove equational properties of maps definable with $\RuleJ$; this is done by allowing $\hom(a,b)$ to be contracted \emph{inside equality judgments}. Intuitively, given entailments $\alpha[e]$ and $\beta[e]$ with an equality in context $e:\hom_\C(a,b)$ which can be contracted using $\RuleJ$, we can deduce that $\alpha$ and $\beta$ are equal everywhere as soon as they are equal on $e = \refl_{\C,z}$ for every $z : \C$.
\end{itemize}

\section{Directed Equality \emph{à la} Martin-Löf}
\label{sec:examples_directed_equality}

We show how the rules for directed equality can be used to obtain the same terms definable with symmetric equality in \ML type theory, and proving properties about them follows precisely the steps of the usual proofs, i.e., by equality contraction and computation rules~\cite{UnivalentFoundationsProgram2013homotopy,Hofmann1997syntax}. All examples in this section satisfy the constraints for $\Rulecutnat,\Rulecutdin$ to be applied.

We start by showing transitivity of directed equality, i.e., categories have composition maps.

\begin{example}[Composition in a category]\label{ex:composition_example}%
The following derivation constructs the \emph{composition map for $\C$}, which is covariant in $a:\Cop,c:\C$ and dinatural in $b:\C$:
\[
\begin{prooftree}
\infer0[\Rulevar]{[z:\C,c:\C]~\takespace{f:\hom(a,b),\, g:\hom(\n b,c)}{g : \hom(\n z, c)} & \vdash g : \hom(\n z,c)}
\infer1[\RuleJ]{[a:\Cop,b:\C,c:\C]\ {f:\hom(a,b),\, g:\hom(\n b,c)} & \vdash {J(g)} : \hom(a,c)}
\end{prooftree}
\]
We contracted the first equality $f:\hom(a,b)$. Rule $\RuleJ$ can be applied since $a,b$ appear only contravariantly in context ($a$ does not appear) and covariantly in the conclusion ($\overline b$ does not). We now prove that $\mathsf{comp}[f,g] := J(g)$, denoted as ``$f \< g$'', is unital on identities (i.e., $\refl_\C$) and associative. Since we chose to contract $f$, the computation rule ensures unitality on the left:
\[
\begin{prooftree}
\infer0[\RuleJcomp]{[z:\C,c:\C]\ g:\hom(\n z,c) \vdash \refl_z \scomp g = g : \hom(\n z,c)}
\end{prooftree}
\]
On the other hand, to show that composition is right-unital we use dependent directed equality induction $\RuleJeq$, where now it is enough to just consider the case in which $a=z=w$ and $f = \refl_w$,
\[
\begin{prooftree}
    \infer0[\RuleJcomp]{[w:\C]\ \emptyctx & \vdash \refl_w \scomp \refl_w = \refl_w : \hom(\n w,w)}
    \infer 1[\RuleJeq]{[a:\Cop,z:\C]\ f:\hom(a,z) & \vdash {f \scomp \refl_z} = f : \hom(a,z)}
\end{prooftree}
\]
which follows by the computation rule for $\comp$ since $\refl_w$ is on the left. Similarly, to show associativity we just need to consider the case $a=b=z$ and $f=\refl_z$,
\vspace{-0.5em}
\begin{adju}[1]
\begin{prooftree}
\infer0[\RuleJcomp]{[z:\C,c:\C,d:\C]~\takespace[r]{f:\hom(\n a,b), g:\hom(\n b,c), h : \hom(\n c,d)}{g:\hom(\n z,c), h : \hom(\n c,d)}~\vdash & {\refl_z \scomp (g \scomp h)} = (\refl_z \scomp g) \scomp h : \hom(\n z,d)}
\infer1[\RuleJeq]{
[a:\C,b:\C,c:\C,d:\C]\ f:\hom(\n a,b), g:\hom(\n b,c), h : \hom(\n c,d) \vdash &{f \scomp (g \scomp h)} = (f \scomp g) \scomp h : \hom(\n a,d)}
\end{prooftree}
\end{adju}
\vspace{0.5em}

where in the top sequent both entailments are equal to $g \scomp h$ by the computation rules of $\comp$.
\end{example}

\begin{example}[Functorial action on morphisms]
For any term/functor $\C \vdash F : \D$, the functorial action on morphisms of $F$ corresponds with the fact that any term $F$ respects directed equality, i.e., directed equality is a congruence:
\[
\begin{prooftree}
\infer0[\Rulereindex+\Rulerefl]{[z:\C]\,\takespace{f:\hom_\C(\overline x,y)}{\emptyctx} & \vdash F^*(\refl_{\C}) : \hom_\D(F^\op(\overline z),F(z))}
\infer1[\RuleJ]{[x:\C,y:\C]\,f:\hom_\C(\overline x,y) & \vdash \J(F^*(\refl_{\C})) : \hom_\D(F^\op(\overline x),F(y))}
\end{prooftree}
\]
and thus we define $\map_F[f] := \J(F^*(\refl_{\C}))$, using $\Rulereindex$ with $F$ in the top sequent.

The computation rule states that $F$ maps identities to identities:
\[
\begin{prooftree}
\infer0[\RuleJcomp]{[z:\C]\,\top \vdash \map_F[\refl_\C] = F^*(\refl_\C) : \hom_\D(F^\op(\nx),F(x))}
\end{prooftree}
\]
The following shows functoriality for free; both top sides reduce to $\map_F[g]$ using $\RuleJcomp$:
\vspace{-0.5em}
\begin{adju}[1]
\begin{prooftree}
\infer0[\RuleJcomp]{
[z:\C,c:\C]~\takespace[r]{f:\hom(\n a,b), g:\hom(\n b,c)}{g:\hom(\n z,c)}~\vdash & {\map_F[\refl_z \scomp g]} = \refl_{F(z)}~\scomp \map_F[g] : \hom(\n z,d)}
\infer1[\RuleJeq]{
[a:\C,b:\C,c:\C]\ f:\hom(\n a,b), g:\hom(\n b,c) \vdash & {\map_F[f \scomp g]} = \map_F[f] \scomp \map_F[g] : \hom(\n a,d)}
\end{prooftree}
\end{adju}
\end{example}

\begin{example}[Transport]%
Transporting points of predicates along directed equalities~\cite[2.3.1]{UnivalentFoundationsProgram2013homotopy} is the functorial action of copresheaves $P : \C\!\>\!\Set$, i.e., predicates $[x : \C]\,P \prop$, for $x$ only positive:
\[
\begin{prooftree}
\infer0[\Rulevar]{[z:\C]\ k: P(z) & \vdash k:P(z)}
\infer1[\RuleJ]{[a:\Cop,b:\C]\ f:\hom(a,b), k:P(\overline a) & \vdash J(k) : P(b)}
\end{prooftree}
\]
The computation rule simply states that transporting a point of $P(a)$ along the identity morphism with $\subst[f, k] := \J(k)$ is the same as giving the point itself, i.e., $\subst[\refl_{\C},k] = k$.
\end{example}

\begin{example}[Pair of rewrites]%
Pairs of directed equalities induce directed equalities between pairs.
The other direction (i.e., ``directed injectivity of pairs'') follows from congruence of directed equality with the projections $\pi_1,\pi_2$ and then using the judgmental equality of terms.
\[
\begin{prooftree}
\infer0[\Rulereindex+\Rulerefl]{[z:\C,z':\D]\ \emptyctx & \vdash \hom_{\C\times\D}((\n z,\n z),(z,z))}
\infer1[\RuleJ]{[a':\Cop,b':\D,z:\C]\ g:\hom_{\D}(b,b')  & \vdash \hom_{\C\times\D}((\n z,b),(z,b'))}
\infer1[\RuleJ]{[a,a':\Cop,b,b':\D]\ f:\hom_{\C}(a,a'), g:\hom_{\D}(b,b')  & \vdash \hom_{\C\times\D}((a,b),(a',b'))}
\end{prooftree}
\]
\end{example}

\begin{example}[Higher-dimensional rewriting]%
The following shows that a directed equality between functors induces a natural transformation~\cite[1.4.1]{Loregian2021coend} (omitting the resulting term for simplicity):
\[
\begin{prooftree}
\infer0[\Rulereindex+\Rulerefl]{[H:[\C,\D],x:\C]\ \emptyctx & \vdash \hom_\D(\n H \cdot \n x,H \cdot x)}
\infer1[\Ruleend]{[H:[\C,\D]]\ \emptyctx & \vdash \Endf{x:\C} \hom_\D(\n H \cdot \n x,H \cdot x)}
\infer1[\RuleJ]{[F:[\C,\D]^\op,G:[\C,\D]]\ e:\hom_{[\C,\D]}(F,G) & \vdash \Endf{x:\C} \hom_\D(F \cdot \n x,G \cdot x)}
\end{prooftree}
\]
The opposite direction is not derivable in general, since in the case where $\C,\D$ are discrete categories (i.e., sets), it corresponds to function extensionality.
\end{example}

\begin{example}[Existence of singletons]\label{ex:singletons}%
The following derivation asserts that singleton subsets are inhabited~\cite[Remark 1.12.1]{UnivalentFoundationsProgram2013homotopy}, i.e., there is a proof for the first-order logic formula $\forall x.\exists y.x=y$:
\[
\begin{prooftree}
\infer0[\Rulevar]{[x:\Cop]\ k:\Coendf{y:\C} \hom_\C(x,y) & \vdash k:\Coendf{y:\C} \hom_\C(x,y)}
\infer1[\Rulecoendfrobenius]{[x:\Cop,y:\C]\ f:\hom_\C(x,y) & \vdash \textsf{coend}^{-1}(k)[f] : \Coendf{y:\C} \hom_\C(x,y)}
\infer1[\Rulecutnat]{[x:\C]\ \emptyctx & \vdash \textsf{coend}^{-1}(k)[\refl_x] : \Coendf{y:\C} \hom_\C(\n x,y)}
\infer1[\Ruleend]{[\,]\ \emptyctx & \vdash \textsf{end}(\textsf{coend}^{-1}(k)[\refl_x]) : \Endf{x:\C} \Coendf{y:\C} \hom_\C(\n x,y)}
\end{prooftree}
\]
This derivation is actually an isomorphism in the model, i.e., singletons are contractible. This follows from dependent directed equality contraction, which we show in detail in \CiteAppendix{B}{sec:appendix:dtt_other_derivations}.
\end{example}

The following theorems show that in our type theory both naturality and dinaturality follow ``for free'' from dependent directed equality contraction.  Cuts are allowed in both cases because of the \emph{natural} appearance of variables in $\textsf{subst}$.

\begin{example}[Internal naturality for entailments]
\label{ex:internal_naturality}
For any $[x:\C]\ P(x) \vdash \alpha : Q(x)$, an internal version of naturality for entailments holds via $\RuleJcomp$:
\vspace{-1.0em}
\begin{adju}[1.0]
\begin{prooftree}
\infer0[\RuleJcomp]{[z:\C]\,k:P(z) & \vdash \alpha[\subst_P[\refl_z,k]] = \subst_Q[\refl_z,\alpha[k]] : Q(z) }
\infer1[\RuleJeq]{[a:\Cop, b:\C]\,f:\hom_\C(a,b), k:P(\n a) & \vdash \alpha[\subst_P[f,k]] = \subst_Q[f,\alpha[k]] : Q(b)}
\end{prooftree}
\end{adju}
\end{example}

\begin{example}[Internal dinaturality for entailments]
\label{ex:internal_dinaturality}
For any $[x:\C]\ P(\nxx) \vdash \alpha : Q(\nxx)$, an internal version of (di)naturality for entailments, as in \Cref{def:dinaturality}, holds via $\RuleJcomp$:
\vspace{-0.8em}
\begin{adju}[1.0]
\begin{prooftree}
\infer0[\RuleJcomp]{[z:\C]\,k:P(\nzz) \vdash \subst_Q[(\refl_z,\refl_z),[\alpha[\subst_P[(\refl_z,\refl_z),k]]]] }
\infer[no rule]1{\phantom{aaaaaaaaaaaaaaaaaaaaa} = \subst_Q[(\refl_z,\refl_z),[\alpha[\subst_P[(\refl_z,\refl_z),k]]]] : Q(\nzz)}
\infer1[\RuleJeq]{[a:\Cop, b:\C]\,f:\hom_\C(a,b), k:P(\n b,\n a) & \vdash \subst_Q[(\refl_b,f),[\alpha[\subst_P[(f,\refl_a),k]]]] }
\infer[no rule]1{& = \subst_Q[(f,\refl_a),[\alpha[\subst_P[(\refl_b,f),k]]]] : Q(a,b)}
\end{prooftree}
\end{adju}
\end{example}
\noindent We elucidate more in detail why the above sequence of cuts is valid in \CiteAppendix{H}{sec:appendix:illustrate_internal_dinat}.

We show in \CiteAppendix{B}{sec:appendix:dtt_other_derivations} how natural transformations between \emph{terms} can be defined using ends~\cite[1.4.1]{Loregian2021coend}, which we use to capture the identity natural, composition of naturals, and internal naturality.

\subsection{On the Adjoint Formulation}
\label{sec:adjoint_formulation}

We elaborate how the adjoint formulation, i.e., the fact that rules are formulated as bijections of entailments, differs from the standard type-theoretical presentation of connectives in the style of natural deduction or sequent calculus~\cite[5.1.6]{Mimram2020program}. Since in both of these systems cut is either derivable or admissible, we cannot recover the usual rules for introduction/elimination for quantifiers and implication, since in the semantics this would enable us to compose any two entailments/dinatural transformations.
We give an example of introduction/elimination-like rules derivable from the adjoint formulation for (co)ends in \Cref{thm:elimination_rule_ends}.

\begin{example}[Rules for (co)ends with terms]\label{thm:elimination_rule_ends}
The following derivations capture an elimination rule for ends and, dually, an introduction rule for coends using a concrete diterm $\Gamma^\op,\Gamma \vdash F : \C$:\vspace{-0.4em}\[
\begin{prooftree}
\infer[no rule]0{[\Gamma]\ \Phi \vdash \alpha : \End{x:\C} P(\n x,x)}
\infer1[\Ruleendinv]{[x:\C,\Gamma]\ \Phi \vdash \textsf{end}^{-1}(\alpha) : P(\n x,x)}
\infer1[\Rulereindex]{[\Gamma]\ \Phi \vdash F^*(\textsf{end}^{-1}(\alpha)) : P(F, F)}
\end{prooftree}
\quad
\begin{prooftree}
\infer[no rule]0{[\Gamma]\ k:\Coend{x:\C} P(\n x, x), \Phi \vdash \alpha : Q}
\infer1[\Rulecoendinv]{[x:\C,\Gamma]\ k:P(\n x,x), \Phi \vdash \textsf{coend}^{-1}(\alpha) : Q}
\infer1[\Rulereindex]{[\Gamma]\ k:P(F, F), \Phi \vdash F^*(\textsf{coend}^{-1}(\alpha)) : Q}
\end{prooftree}
\]
We can recover the the projection and injection maps of (co)ends (i.e., the ``(co)units'' of the adjoint formulation) by picking $Q := \Coendf{x:\C} P(\n x, x)$, $\Phi := \Endf{x:\C} P(\n x,x), \Phi'$ and $\alpha := \Rulevar$ as follows:%
\vspace{-0.9em}%
\begin{adju}[1.0]%
\begin{prooftree}%
\infer0{[\Gamma]\ k:\End{x:\C} P(\n x,x), \Phi' \vdash F^*(\textsf{end}^{-1}(k)) : P(F, F)}
\end{prooftree}
\ \
\begin{prooftree}
\infer0{[\Gamma]\ k:P(F, F), \Phi \vdash F^*(\textsf{coend}^{-1}(k)) : \Coend{x:\C} P(\n x, x)}
\end{prooftree}
\end{adju}
\LinkRulecoendunit%
\LinkRuleendcounit%
\end{example}%
\vspace{-0.45em}%
The crucial aspect is that we cannot derive the above introduction/elimination rules where, instead, the end appears on the left, or the coend on the right: these would be the remaining rules for the quantifiers of \emph{sequent calculus}, and hence full cut would be admissible. In particular we only recover $\forall_R$ and $\exists_L$, but not $\forall_L$ and $\exists_R$, using the terminology of~\cite[5.1.8]{Mimram2020program}. We formally prove the non-admissibility of an unrestricted cut rule in \Cref{thm:no-full-cut}.

In standard accounts of logic, the adjoint-form is equivalent to the usual introduction and elimination rules for connectives, but only in the presence of \emph{cut}~\cite[4.1.8]{Jacobs1999categorical}. Hence, in our setting we can recover the usual rules only in contexts that are sufficiently \emph{natural} to allow for cuts to be applied. We give an example of this situation in \Cref{ex:cutfulcoends} to derive introduction/elimination-like rules for existentials in the style of natural deduction~\cite[5.1.6]{Mimram2020program}, and derive in \Cref{ex:cutfulexponential} transitivity of implication (which directly translates to an elimination rule).

\begin{example}[Natural deduction-style rules for coends]\label{ex:cutfulcoends}
The following derivations capture rules where coends are on the right of the turnstile: an elimination rule, an introduction rule with a concrete term $\Delta \vdash F : \C$ (not a \emph{di}term), and an introduction rule with two variables $x:\Cop,y:\C$:%
\vspace{-1.3em}%
\begin{adju}[1.0]%
\begin{minipage}{0.39\textwidth}\begin{prooftree}
\infer[no rule]0{[\Gamma,d\!:\!\Delta]\ \Phi(d) \vdash \Coendf{x:\C} P(\n x,x,d)}
\infer[no rule]1{[\Gamma,z:\C,d\!:\!\Delta]\ P(\n z,z,d),\Phi(d) \vdash Q(d)}
\infer1{[\Gamma,d:\Delta]\ \Phi(d) \vdash Q(d)}
\end{prooftree}
\end{minipage}%
\begin{minipage}{0.30\textwidth}\hspace{0.75em}%
\begin{prooftree}
\infer[no rule]0{[\Gamma,d\!:\!\Delta]\ \Phi(d) \vdash Q(F(d),d)}
\infer1{[\Gamma,d\!:\!\Delta]\ \Phi(d) \vdash \Coendf{x:\C} Q(x,d)}
\end{prooftree}
\end{minipage}
\begin{minipage}{0.34\textwidth}
\begin{prooftree}
\infer[no rule]0{[\Gamma,x\!:\!\Cop,y\!:\!\C]\ \Phi(x,y) \vdash R(x,y)}
\infer1{[\Gamma]\ \Phi(x,y) \vdash \Coendf{z:\C} R(\n z,z)}
\end{prooftree}
\end{minipage}
\end{adju}%
\vspace{0.3em}
Note that the variables of $\Delta$ are always used naturally, and $P,Q,R$ do not depend on $\Gamma$. $F$ cannot be a diterm since $Q(F(\n x,x))$ would make the top entailment dinatural in the variables of $\Delta$.
We report complete derivations for these rules in \CiteAppendix{C}{sec:appendix:adjoint_other_rules}.
\end{example}

\begin{example}[Transitivity of implication]\label{ex:cutfulexponential}
Implication is transitive \emph{in natural contexts}, with $[\Gamma]\ \Phi$:
\[
\begin{prooftree}
  \hypo{[a:\C]\ \Phi & \vdash \alpha : P(\n a) \Rightarrow Q(a)}
  \infer1[\Ruleexpinv]{[a:\C]\ P(a),\Phi & \vdash \textsf{exp}^{-1}(\alpha) : Q(a)}
  \hypo{[a:\C]\ \Phi & \vdash \beta : Q(\n a) \Rightarrow R(a)}
  \infer1[\Ruleexpinv]{[a:\C]\ Q(a),\Phi & \vdash \textsf{exp}^{-1}(\beta) : R(a)}
  \infer2[\Rulecutnat]{[a:\C]\ P(a),\Phi & \vdash \alpha\,;\beta := \textsf{exp}^{-1}(\beta)[\textsf{exp}^{-1}(\alpha)] : R(a)}
\end{prooftree}
\]
Polarized implication is in general not transitive, since, as we will see in \Cref{sec:semantics}, entailments are interpreted as dinaturals which do not compose in general; we show in \CiteAppendix{B}{sec:appendix:dtt_other_derivations} how one can use implication and ends to internalize the set of all entailments/dinaturals.
\end{example}

\subsection{Aspects of Directed Type Theory}
\label{sec:other_aspects}

We investigate in this section other proof-theoretical aspects of our directed type theory: in particular we show why symmetry is not immediately derivable and how all rules for directed equality can be equivalently characterized as a single isomorphism.

\begin{remark}[Syntactic failure of symmetry for directed equality]
\label[remark]{rem:failure_symmetry}
The restrictions in $\RuleJ$ illustrate why one \emph{cannot} derive that directed equality is symmetric, i.e., obtain a general map \[[a:\Cop,b:\C]\ e:\hom_\C(a,b) \vdash \mathsf{sym} : \hom_\C(\overline b,\overline a).\] The equality $e : \hom_\C(a,b)$ cannot be contracted because $\overline a$ appears in the conclusion contravariantly (similarly with $\overline b$), whereas $\RuleJ$ requires that the conclusion only has \emph{covariant} occurrences of the variables being contracted.
\end{remark}

The remark above merely illustrates why it is not derivable \emph{from the syntactic restriction}. We show in \Cref{thm:countermodel_symmetry} that the existence of a countermodel implies that it is not admissible in general.

As in the symmetric case, the rule for directed equality elimination is actually an isomorphism, and asking $\RuleJ$ to be an isomorphism fully characterizes all the rules of directed equality~\cite[3.2.3]{Jacobs1999categorical} (in the presence of the structural rules $\Rulecutnat$ and $\Rulevar$):
\begin{theorem}[Directed $\J$ as isomorphism] \agda[Dinaturality/J-Iso.agda]\label{thm:j_isomorphism}~\LinkRuleJinv%
Rule $\RuleJ$ is an isomorphism, and the inverse map is given by $J^{-1}(h) := h[\refl_\C]$ using $\Rulecutnat$ and $\Rulerefl$. Moreover, $\RuleJeq$ is logically equivalent to the rule $J(J^{-1}(\alpha)) = \alpha$ in the equational theory for every $\alpha$.
\end{theorem}
\begin{proof}
The computation rule states precisely that $J^{-1}(J(\alpha)) = \alpha$.
To show $J(J^{-1}(\alpha)) = \alpha$, we instantiate $\RuleJeq$ with $\alpha := J(\beta[\refl_\C])$ and use $\RuleJcomp$ in the hypothesis, i.e., $J(\beta[\refl_\C])[\refl_\C] = \beta[\refl_\C]$, to obtain $J(\beta[\refl_\C]) = \beta$ as desired. We show that $J(J^{-1}(\alpha)) = \alpha$ implies $\RuleJeq$: the hypothesis in $\RuleJeq$ states exactly $J^{-1}(\alpha) = \J^{-1}(\beta)$, hence $\alpha = \beta$ by applying $J$ on both sides.
\end{proof}

\begin{theorem}[$\J^{-1}\!\!\iff\!\!\refl$]
\label{thm:refl_from_hexagon}
\!Rule $\Rulerefl$ is logically equivalent to $\RuleJinv$.
\end{theorem}
\begin{proof}
Clearly $\Rulerefl$ implies $\RuleJinv$ by definition. Rule $\Rulerefl$ follows from $\RuleJinv$ in \Cref{thm:j_isomorphism} by picking $P(a,b):=\hom(a,b)$ and using the projection $\Rulevar$ to return the hypothesis $e:\hom_\C(a,b)$ as the bottom side map $h$, obtaining $\textsf{refl}_\C := J^{-1}(e)$. We leave the proof that the computation rule $J(h)[\textsf{refl}_\C] = h$ holds in \CiteAppendix{E}{appendix:computation_jinv}.
\end{proof}
The following derivations illustrate how dinaturality, intuitively, allows us to ``ignore'' polarity in the contexts of predicates, i.e., one can equivalently consider a \emph{contravariant} variable of type $\C$ as a \emph{covariant} variable of type $\Cop$, and viceversa.

\begin{theorem}[$\op$ of entailments]
\label{thm:op_of_entailments}
The following rule is derivable:
\[
\begin{prooftree}
\hypo{[x:\C,\Gamma]\ \Phi(\n x,x) \vdash \alpha : P(\n x, x)}
\infer1{[x:\Cop,\Gamma]\ \Phi^{x \mapsto \op}(x,\n x) \vdash \alpha^{{x \mapsto \op}} : P^{x \mapsto \op}(x, \n x)}
\end{prooftree}
\]
\end{theorem}
\begin{proof}
Follows by reindexing $\Rulereindex$ with the \emph{``negative projection'' diterm} $\n x:\C, x:\Cop \vdash \n x:\C$. The predicate obtained by substituting this term in $P$ coincides (metatheoretically) with $P^{x \mapsto \op}$. This reindexing is involutive in the sense that $\left(\alpha^{{x \mapsto \op}}\right)^{{x \mapsto \op}} = \alpha$ in the equational theory.
\end{proof}

In particular, the above derivation allows us to \emph{derive} different versions of $\RuleJ$ which adopt one or the other convention: for example $\RuleJ$ could be stated by requiring $a:\C$ (rather than $\Cop$) but then ask for contravariance of $a$ in the conclusion and covariance in $\Phi$. The formulation chosen in $\RuleJ$ with $a : \Cop,b : \C$ is simpler to state in terms of ``correct'' and ``incorrect'' appearances and emphasizes how the two variables play different asymmetric roles.

The following derivation shows how dinaturality allows us to capture a sort of ``mixed-variance reindexing'' $\C \to \C^\op\times\C$, since even variables with different polarities can be identified together.

\begin{theorem}[Dinatural collapse]
\label{thm:dinat_collapse}
The following rule is derivable:
\[
\begin{prooftree}
\hypo{[x:\Cop,y:\C,\Gamma]\ \Phi(\n x,x,\n y,y) \vdash \alpha : P(\n x,x,\n y,y)}
\infer1{[z:\C,\Gamma]\ \Phi(z,\n z,\n z,z) \vdash \alpha^{x,y \mapsto z} : P(z,\n z,\n z,z)}
\end{prooftree}
\]
\end{theorem}
\begin{proof}
Follows by reindexing $\Rulereindex$ with the \emph{``identity'' diterm} $\n x:\Cop, x:\C \vdash \langle \n x, x \rangle : \Cop \times \C$.
\end{proof}

The dinatural collapse operation can be used to ``downgrade'' natural transformations to dinatural transformations, which no longer compose; since we check for naturality syntactically, this allows for a situation in which two (dinatural) entailments do not compose in the syntax despite composing in the semantics (since the map being constructed remains unaltered).

\begin{remark}[Collapse loses compositionality]
\label[remark]{thm:no_compositionality}
We illustrate how dinatural collapse can make an entailment no longer composable. Recall the composition map $\textsf{\upshape comp}[f,g] := J(g)$ from \Cref{ex:composition_example}: then, the following entailments are not composable in the syntax, since both ${\textsf{\upshape comp}}^{a,b\mapsto z}$ and $\refl$ are dinatural in $z$; however, $\textsf{\upshape comp}[\refl_z,k]$ is a valid application of $\Rulecutnat$:
\[
\begin{prooftree}
\infer0{[z:\C]\ \Phi \vdash \refl : \hom_\C(\n z,z)}
\end{prooftree}\quad
\begin{prooftree}
\infer0{[a:\Cop,b,c:\C]\ {\hom_\C(a,b),\, \hom_\C(\n b,c)} & \vdash {\textsf{\upshape comp}} : \hom_\C(a,c)}
\infer1{[z,c:\C]\ {\hom_\C(\n z,z),\, \hom_\C(\n z,c)} & \vdash {\textsf{\upshape comp}}^{a,b\mapsto z} : \hom_\C(\n z,c)}
\end{prooftree}
\]
Note that one \emph{can} apply $\textsf{\upshape comp}$ to a constant dinatural $[\,]\ \emptyctx \vdash \alpha:\hom_\C(A,A)$ that selects some endomorphism for a concrete constant $[\,] \vdash A:\C$, since $\alpha$ would be natural in the empty context.
\end{remark}

We elucidate using $\Ruleexp$ why the exponential object in the category of presheaves and \emph{natural transformations} is non-trivial~\cite[6.3.20]{Leinster2014basic}, and is not the pointwise $\hom$ in $\Set$.

\begin{remark}[Exponentials for naturals]Given an entailment which is fully covariant in $x$ (i.e., a natural transformation) for predicates $[x:\C]\ F(x),G(x),H(x)$, by directly applying $\Ruleexp$,
\[
\begin{prooftree}
\hypo{[x: \C]\ F(x) \x G(x) & \vdash H(x) }
\infer[double]1[\Ruleexp]
    {[x: \C]~\takespace{F(x) \x G(x)}{G(x)} & \vdash F(\overline x) \= H(x)}
\end{prooftree}
\]
one has a natural transformation on top, but the bottom family of arrows is \emph{dinatural} in $x$.
\end{remark}
We show in \Cref{thm:derivation_presheaves_closed} how $\Ruleexp$ and the rules for directed equality can be used to give a logical proof that the usual definition of exponential for presheaves is indeed the correct one.

\section{Dinaturality}
\label{sec:dinaturality}

We recall some preliminary facts about dinatural transformations and (co)ends in order to present the semantics of our type theory. We will often abbreviate the term dinatural transformations simply as ``dinaturals'', and ordinary natural transformations as ``naturals''. %

\begin{definition}[Dipresheaves and difunctors]\label{def:dipresheaf}%
Consider the (strict) comonad $-^\diamond : \Cat \> \Cat$ defined by $\C \mapsto \Cop \x \C$, where the counit is given by projecting and comultiplication by duplicating and swapping. A \emph{dipresheaf} is simply a functor $\C^\diamond\>\Set$, i.e. a functor $\Cop\x\C\>\Set$.
\end{definition}
We always denote composition diagrammatically, i.e., $f\<g : a \> c$ for $f : a \> b, g : b \> c$.
\begin{definition}[Dinatural transformation~\cite{Dubuc1970dinatural}]\label{def:dinaturality}%
Given functors $F,G:\Cop \times \C\>\D$, a \emph{dinatural transformation} $\alpha : F \dinat G$ is a family of arrows $\alpha_x : F(x,x) \longrightarrow G(x,x)$ indexed by objects $x:\C$ such that
$\forall a,b:\C$, and $f : a \> b$ the following equation between arrows $F(b,a) \to G(a,b)$ holds:
\[F(\id_b,f)\< \alpha_b\< G(f,\id_b) = F(f, \id_a)\< \alpha_a\<G(\id_a,f).\]
\end{definition}

\begin{lemma}[Dinaturals generalize naturals~\cite{Dubuc1970dinatural}]\label[lemma]{rem:dinat_subsumes_nat}%
A natural transformation $\alpha : F \rightarrow G$ for $F,G : \C \> \D$ equivalently corresponds with a dinatural $\alpha : (\pi_2 \< F) \dinat (\pi_2\<G) :\Cop \times \C \> \D$.
\end{lemma}

The pointwise composition of two dinatural transformations is not necessarily dinatural (see~\cite{McCusker2021composing,Freyd1992functorial}), but dinaturals always compose with naturals on both the left and right side:
\begin{lemma}[Dinaturals compose with naturals~\cite{Dubuc1970dinatural}]\label[lemma]{thm:nat_dinat_compose}%
Given a dinatural transformation $\gamma : F \dinat G$ and natural transformations $\alpha : F' \> F, \beta : G \rightarrow G'$ for $F,F',G,G':\Cop\x\C\>\Set$, the map $\alpha \< \gamma \< \beta : F' \dinat G'$ defined by $(\alpha \< \gamma \< \beta)_x := \alpha_{xx} \< \gamma_{x} \< \beta_{xx}$ is dinatural.
\end{lemma}

Non-compositionality of  dinaturals  is an in\-trin\-sic property of \emph{directed proof-relevant} type theory, since in the groupoidal case they all compose (in the proof-irrelevant case, where $\Set$ is replaced by the preorder $\I := \set{0 \to 1}$, dinaturals compose trivially since there is no hexagon to check):
\begin{theorem}[Dinaturals in groupoids compose]\agda[Dinaturality/GroupoidCompose.agda]\label{thm:dinaturals_groupoid_compose}
Given a groupoid $\C$ and a category $\D$ for functors $F,G,H:\Cop \times \C\>\D$, any two dinaturals $\alpha:F\dinat G, \beta:G \dinat H$ compose.
\end{theorem}

The fundamental idea behind all rules for directed equality is given by the following elementary result, which connects dinatural transformations in $\Set$ with a corresponding natural one:

\begin{theorem}[Dinaturals and $\hom$-naturals]
\agda[Dinaturality/NaturalDinatural.agda]
\label{thm:dinatural_nat_hexagon}
For any $P,Q:\Cop\x\C\>\Set$, there is a bijection between set of dinatural transformations $P \dinat Q$ and certain natural transformations between functors $\Cop\x\C\>\Set$, as follows:
    \[
\begin{prooftree}
\label{rule:dinathomnat}
\hypo{\alpha_x : P(\nxx) \dinat Q(\nxx)}
\infer[double]1{\gamma_{ab} : \hom(a, b) \longrightarrow P^\op(b, a) \Rightarrow Q(a, b)}
\end{prooftree}
    \]
\end{theorem}
\begin{proof}
We describe the maps in both directions:
\begin{itemize}
    \item[($\Downarrow$)] Given a dinatural $\alpha : P \dinat Q$ and a morphism $f : \hom(a,b)$, the map $P(b,a) \> Q(a,b)$ corresponds precisely with the sides of the equation given in \Cref{def:dinaturality} for dinaturality, which is obtained by applying the functorial action of $P$ and $Q$.
    \item[($\Uparrow$)] Take $a=b$ and precompose with $\id_a\in\hom(a,a)$.
\end{itemize}
The fact that this is an isomorphism follows from the (di)naturality of both sets of maps. Note the similarity between the above argument and the proof of the Yoneda lemma, where the two central ideas are precisely applying the functorial action and instantiating at $\id$, with the isomorphism following from (di)naturality.
\end{proof}

\noindent We now recall definitions for the semantics of (co)ends, later used to give semantics to quantifiers.

\begin{definition}[(Co)wedges for $P$ {\cite[1.1.4]{Loregian2021coend}}]
Given $P:\Cop\x\C\>\D$, a \emph{wedge for $P$} is a pair object/dinatural $(X:\D,\alpha:\textsf K_X \dinat P)$, where $\textsf K_X$ is the constant functor in $X$. A \emph{wedge morphism} $(X,\alpha) \to (Y,\alpha')$ is an $f\!:\!X \to Y$ of $\D$ such that $\forall c:\C,\alpha_c = f \< \alpha'_c$. A \emph{cowedge} is a wedge in $\Dop$, denoting the categories of (co)wedges as $\Wedge(P),\Cowedge(P)$.
\end{definition}

\begin{definition}[(Co)ends {\cite[1.1.6]{Loregian2021coend}}]
Given a functor $P:\Cop\x\C\>\D$, the \emph{end} of $P$ is defined to be the terminal object of $\Wedge(P)$, whose object in $\D$ is denoted as $\Endf{x:\C} P(\nxx)$. Dually, the \emph{coend} of $P$ is the initial object of $\Cowedge(P)$, denoted similarly as $\Coendf{x:\D} P(\nxx)$. The integral symbol acts as a binder, in the sense that ``$\Endf{c:\C} P(c,c)$'' and ``$\Endf{x:\C} P(x,x)$'' are ($\alpha$-)equivalent; moreover, $P$ can depend on many parameters, e.g., if $P:(\Aop \x \A) \x (\Bop \x \B) \to \D$ then $\Endf{b:\B} P(\n a,a,\n b,b):\Aop \x \A \to \D$. (Co)ends exist when $\D$ is (co)complete~\cite{Loregian2021coend}.
\end{definition}

\section{Semantics}
\label{sec:semantics}

We now describe the categorical semantics of our directed type theory: the main idea behind categorical semantics is that we define functions that associate a certain mathematical object to each derivation tree, inductively. Whenever present, the symbol \agda\ links to the Agda formalization of the semantic interpretation of each rule.

The semantics for types, contexts, variables, terms, predicates and propositional contexts is given in \Cref{fig:semantics_rules}. The equality judgments associated to these are interpreted in a straightforward way, which we omit from this presentation; such equalities are only used to take care of involutions and the equational theory of terms, for which we therefore give a \emph{strict} semantics: equality of types and contexts is interpreted as \emph{isomorphisms} of categories, term equality is strict isomorphism of functors. Equality of predicates is similarly trivial since it only inherits congruence rules from the previous equality judgments.

The main rules of our type theory are those of entailments, for which we describe in detail the intuition behind the semantics of each rule and its soundness in dinatural transformations.

\begin{figure*}[h]
    \begin{adju}[1.0]
  \begin{tabular}{l@{}l@{}l}
    $\begin{array}{l}
        \begin{array}{l@{\,}l}
            \nowidth{\sem{-} : \set{{-} \type} \to \Cat} & \\
            \sem{\C^\op} & := \sem{\C}^\op \\
            \sem{\C \times \D} & := \sem{\C} \times \sem{\D} \\
            \sem{[\C,\D]} & := [\sem{\C},\sem{\D}] \\
            \sem{\top} & := \top
        \end{array} \\[3.0em]
        \begin{array}{l@{\,}l}
            \nowidth{\sem{-} : \set{{-} \ctx} \to \Cat} & \\
            \sem{[]} & := \top \\
            \sem{\Gamma^\op} & := \sem{\Gamma}^\op \\
            \sem{\Gamma, \C} & := \sem{\Gamma} \times \sem{\C}
        \end{array}
    \end{array}$
    &
    $\begin{array}{l}
        \begin{array}{l@{\,}l}
            \nowidth{\sem{-}^v : \set{\Gamma \ni - : \C} \to [\sem{\Gamma},\sem{\C}]} & \\
            \sem{\Gamma,x:\C \ni x : \C}^v & := \pi_2 \\
            \sem{\Gamma,y:\D \ni x : \C}^v & := \pi_1 \< \sem{x}^v
        \end{array} \\[2.0em]
        \begin{array}{l}
            \nowidth{\sem{-} : \set{\Gamma \vdash {-} : \C} \to [\sem{\Gamma},\sem{\C}]} \\
            \sem{x} := \sem{x}^v \\
            \sem{t^\op} := \sem{t}^\op \\
            \sem{\ang{s,t}} := \ang{\sem{s},\sem{t}} \\
            \sem{\pi_1(p)} := \sem{p} \< \pi_1 \\
            \sem{\pi_2(p)} := \sem{p} \< \pi_2 \\
            \sem{s \cdot t} := \ang{\sem{s},\sem{t}} \< \textsf{eval} \\
            \sem{\lambda x.t(x)} := \Lambda(t)
        \end{array}
    \end{array}$
    &
    $\begin{array}{l}
        \begin{array}{l@{\,}l}
            \nowidth{\sem{-} : \set{[\Gamma] - \prop} \to [\sem{\Gamma}^\op \times \sem{\Gamma},\Set]} & \\
            \sem{\top} & := \lambda \n \gamma, \gamma. \top_\Set \\
            \sem{P \times Q} & := \ang{\sem{P}, \sem{Q}} \,; \times_\Set \\
            \sem{P \Rightarrow Q} & := \ang{\sem{P}, \sem{Q}} \,; \Rightarrow_\Set \\
            \sem{\hom_\C(s,t)} & := \ang{\sem{s}, \sem{t}} \,; \hom_\C \\
            \sem{\Endf{x:\C} P(\n x,x)} & := \lambda \n \gamma, \gamma. \Endf{x:\C} P(\n x,x,\n \gamma,\gamma) \\
            \sem{\Coendf{x:\C} P(\n x,x)} & := \lambda \n \gamma, \gamma. \Coendf{x:\C} P(\n x,x,\n \gamma,\gamma) \\
        \end{array} \\[5.0em]
        \begin{array}{l@{\,}l}
            \nowidth{\sem{-} : \set{{-} \propctx} \to [\sem{\Gamma}^\op \times \sem{\Gamma},\Set]} & \\
            \sem{\emptyctx} & := \lambda \n \gamma, \gamma. \top_\Set \\
            \sem{\Phi,P} & := \ang{\sem{\Phi}, \sem{P}} \,; \times_\Set \\
        \end{array}
    \end{array}$
  \end{tabular}
\end{adju}
\Description{Semantics for types, contexts, variables, terms, predicates and propositional contexts of the dinatural directed type theory.}
\caption{Semantics for the main judgments of dinatural directed type theory.}
\label{fig:semantics_rules}
\end{figure*}

\begin{theorem}[Soundness in dinatural transformations]~\agda[All.agda]
Each rule in \Cref{fig:syntax:entailments} is validated using the semantics in categories, functors, dipresheaves, dinatural transformations. Inference rules are interpreted by functions between sets of dinaturals; these are isomorphisms when double-lines appear. Moreover, every function is \emph{natural} in all the dipresheaves (both predicates and propositional contexts) that appear in the rule.
\end{theorem}
\noindent We unpack this theorem by validating and describing the intuition behind each rule, using semantic brackets $\sem{-}$ to indicate the semantic object denoted by each constructor.

\begin{itemize}[leftmargin=*]
\item { \textbf{Structural rules.}~\agda[Dinaturality/Structural.agda]} Rule $\Rulevar$ is interpreted as the dinatural which projects away the predicate $P$. Moreover, $\Ruleweaken$ and $\Rulecontract$ state that dinaturals always compose on the left with, respectively, the projections and the diagonal map in $\Set$.

\item { \textbf{Products.}~\agda[Dinaturality/Product.agda]} Dinaturals validate the interpretation of conjunction in $\Ruleprod$ via the pointwise product of dipresheaves in $\Set$; the bottom sequent indicates the product of sets of dinaturals.

\item { \textbf{Polarized implication.}~\agda[Dinaturality/Implication.agda]} Contrary to naturals and presheaves~\cite{Leinster2014basic}, dinaturals can be curried directly via the $\Ruleexp$ rule by currying each component of $\alpha$ in $\Set$. In the semantics, the metatheoretical operation \Cref{example:predicate_inversion} corresponds to swapping arguments in a dipresheaf.

\item { \textbf{Reindexing with functors as terms.}~\agda[Dinaturality/Reindexing.agda]} Dinaturals can always be ``reindexed'' by plugging functors in each index of the component, preserving dinaturality.

\item { \textbf{Cuts naturals-dinaturals.}}~\agda[Dinaturality/Cut.agda] The two restricted cut rules $\Rulecutdin,\Rulecutnat$ correspond precisely to \Cref{thm:nat_dinat_compose}. Intuitively, both rules are stated in such a way that the dipresheaf $P$ (in the middle of the composition) only contains \emph{natural} occurrences of variables. The use of $\Gamma$ in $\Phi,Q$ is unproblematic since one can suitably take the (co)end over $\Gamma$ to ``hide'' these variables and compose naturals together. Associativity, unitality and coherence in \CiteAppendix{A}{appendix:syntax} are immediate.

The dinatural-into-natural rule $\Rulecutnat$ essentially corresponds to vertical composition in $\Prof$ as a virtual equipment~\cite{New2023formal,Cruttwell2010unified}: in this type theory, however, contravariant occurrences $\n a,\n b$ are allowed to appear in the \emph{same} predicate $P(\n a,\n b)$, but in the double-categorical setting they must be split as $P(...,a),Q(\n a,b),R(\n b,...)$. Note that composing a natural $\alpha$ with the identity dinatural still yields a \emph{di}natural, since the cut rules always collapse the composition as in \autoref{thm:dinat_collapse}. %
\item { \textbf{Directed equality introduction.}~\agda[Dinaturality/Refl.agda]} The rule $\Rulerefl$ states reflexivity of directed equality, and is validated semantically by $\alpha_x(h) := \id_x$. Dinaturality holds since $\forall f:a\>b, f \< \id_b = \id_a \< f$.

\item { \textbf{Directed equality elimination.}~\agda[Dinaturality/J.agda]} This rule and its syntactic restriction comes precisely from \Cref{thm:dinatural_nat_hexagon}: in the bottom side of the isomorphism, the dipresheaf $P$ is curried on the left of the turnside \emph{but inverting the polarity of $a,b$}. This is precisely the propositional context of $\RuleJ$. Hence, the restriction behind $\RuleJ$ comes from the naturality of the bottom map. Explicitly, given a dinatural $h$, the dinatural $\J(h)$ is defined as follows for indices $a:\sem{\C},b:\sem{\Cop},x:\sem{\Gamma}$: \[J(h)_{abx}:= \lambda e, k. (\sem{\Phi}(\id_{b},e, \id_{x}, \id_{x}) \< h_{bx} \< \sem{P}(e, \id_{b},\id_{x},\id_{x}))(k).\] The computation rule clearly holds when $a = b = z$ and $e = \id_{z}$, without the need for dinaturality.

\item { \textbf{Dependent $\hom$ elimination.}}~\agda[Dinaturality/J-Iso.agda] As shown in \Cref{thm:j_isomorphism}, the fact that $\J$ is an isomorphism characterizes directed equality. In particular, dependent equality elimination is the $J(J^{-1}(\alpha)) = \alpha$ direction, which uses naturality in the proof just like the Yoneda lemma~\cite[4.2]{Leinster2014basic}. %

\item { \textbf{(Co)ends.}}~\agda[Dinaturality/Quantifiers.agda]
The rules for (co)ends $\Ruleend$ and $\Rulecoendfrobenius$ express an adjoint-like (up to the non-composition of dinaturals) correspondence $\Coendf{\A[\C]}\!\dashv\,\pi^*_{\A[\C]}\,\dashv\!\Endf{\A[\C]}$ between the weakening functor $\pi^*_{\A[\C]}\!:\![\C^\diamond,\Set] \!\>\![\A^\diamond\x\C^\diamond,\Set]$ and the functors $\Coendf{\A[\C]},\Endf{\A[\C]} : [\A^\diamond\x\C^\diamond,\Set]\!\>\![\C^\diamond,\Set]$ sending dipresheaves to their (co)end in $A$. Semantically, these are simply the bijective correspondences between (co)wedges and morphisms (out of) into (co)ends, but parameterized by an additional context $\Gamma$. Quantifiers in categorical logic typically have to satisfy additional requirements in order to faithfully model logical operations: the Beck-Chevalley condition~\cite[1.9.4]{Jacobs1999categorical} states that ``quantifiers commute with substitution'', and the Frobenius condition~\cite[1.9.12]{Jacobs1999categorical} logically corresponds to having an additional context $\Phi$ in rules for colimit-like connectives~\cite[3.4.4]{Jacobs1999categorical}, as in $\Rulecoendfrobenius$. We show these technical conditions in \CiteAppendix{F}{appendix:bc_frobenius}.
\end{itemize}
\begin{theorem}[Symmetry is not admissible]\label{thm:countermodel_symmetry}
The statement of symmetry of directed equality in \Cref{rem:failure_symmetry} is not admissible in the type theory.
\end{theorem}
\begin{proof}
Add to the signature the category $\I := \{0 \to 1\}$ with a unique non-invertible morphism. By soundness, the lack of symmetry in $\I$ implies that symmetry cannot be derived in general.
\end{proof}

The set of all dinaturals can be characterized as an end $\mathsf{Dinat}(P,Q) \iso \Endf{x:\C} P(\xnx) \Rightarrow Q(\nxx)$; we prove this in \CiteAppendix{D}{appendix:sec:other_derivations_coend_calculus}. We internalize this idea to show that full cut cannot be derived:

\begin{theorem}[No full cut]\label{thm:no-full-cut}
A cut rule where $\Phi,P,Q$ are fully unrestricted is \emph{not} admissible.
\begin{proof}
Assuming full cut, the adjoint formulation is equivalent to the rules in natural deduction-style of first-order logic, which allows one to derive the following map in the empty context:%
\[%
\begin{prooftree}
\infer0{[\,]\ \Endf{x:\C} P(x,\n x) \Rightarrow Q(\n x,x), \Endf{x:\C} Q(x,\n x) \Rightarrow R(\n x,x) \vdash \Endf{x:\C} P(x,\n x) \Rightarrow R(\n x,x)}
\end{prooftree}
\]
by soundness of the semantics, this corresponds to composing \emph{all} dinatural transformations.
\end{proof}
\end{theorem}

\section{Coend Calculus via Dinaturality}
\label{sec:examples_coend_calculus}

We show how the rules for directed equality and (co)ends can be used to give concise proofs with a distinctly logical flavor to several central theorems of category theory. The technique we use mirrors the way (co)end calculus is applied in practical settings (e.g.,~\cite{Boisseau2018what,Hinze2012kan,Roman2020open}) via a ``Yoneda-like'' series of \emph{natural} isomorphisms of sets: to prove that two objects $A, B\!:\!\C$ are isomorphic, one can assume to have a generic object $\Phi$ and then apply a series of isomorphisms of sets \emph{natural} in $\Phi$ to establish that $\C(\Phi,A) \iso \C(\Phi,B)$, from which $A \iso B$ follows by the fully faithfulness of the Yoneda embedding~\cite{Boisseau2018what,Leinster2014basic}. The same technique can be used to show that \emph{functors} are naturally isomorphic, as well as adjunctions, e.g., \Cref{thm:derivation_presheaves_closed,exa:right_kan}. {We now show our main examples, with additional derivations of (co)end calulus in \CiteAppendix{D}{appendix:sec:other_derivations_coend_calculus}, which use Yoneda with $\Phi$ on the right side instead.}

\begin{remark}[Yoneda technique and naturality]\agda[Dinaturality/NaturalityExample.agda]\label[remark]{rem:dinatu_iso} All rules given in previous sections are \emph{natural} in each of the dipresheaves involved. In the following series of examples no proof ever involves a ``dinatural isomorphism'', since it would not be possible to state the final isomorphism with cuts; natural isomorphisms between sets of \emph{di}naturals are only used as intermediate steps. We report in \CiteAppendix{G}{sec:appendix:yoneda_technique} a spelled-out example of this Yoneda technique in the equational theory by explicitly constructing the isomorphisms and using naturality of the adjoint-form rules (i.e., they commute with cuts).%
\end{remark}

\begin{example}[(co)Yoneda lemma]\label[example]{thm:derivation_coyoneda}
For any predicate/copresheaf $[x:\C]\ P(x) \prop$, and a predicate/copresheaf $[x:\C]\ \Phi(x) \propctx$ acting as generic context, the following derivations capture the Yoneda lemma~\cite[Thm. 1]{Loregian2021coend} (using the characterization of naturals as an end) and coYoneda lemma~\cite[III.7, Theorem 1]{MacLane1998categories} (i.e., presheaves are isomorphic to a weighted colimit of representables) %
\vspace{-0.9em}%
\LinkRuleyoneda%
\LinkRulecoyoneda%
\[
\begin{minipage}{0.51\textwidth}
\begin{prooftree}
   \infer[no rule]0{[a\!:\!\C]~\Phi(a) \vdash \Endf{x:\C}\ {\hom}_\C(a,\overline x) \Rightarrow P(x)}
   \infer[double]1[\Ruleend]{[a\!:\!\C,x\!:\!\C]~\Phi(a) \vdash \hom_\C(a,\overline x) \Rightarrow P(x)}
   \infer[double]1[\Ruleexp]{[a\!:\!\C,x\!:\!\C]\hspace{1.6pt}\ \hom_\C(\overline a,x), \Phi(a) \hspace{1.3pt} \vdash P(x)}
   \infer[double]1[\RuleJ]{[z:\C]~\Phi(z) \vdash P(z)}
\end{prooftree}
\end{minipage}
\begin{minipage}{0.51\textwidth}
\begin{prooftree}
   \infer[no rule]0{[a\!:\!\C]~\Coendf{x:\C}\ {\hom}_\C(\overline x,a) \x P(x) & \vdash \Phi(a)}
   \infer[double]1[\Rulecoendfrobenius]{[a\!:\!\C,x\!:\!\C]\ {\hom}_\C(\overline x,a) \x P(x) & \vdash \Phi(a)}
   \infer[double]1[\RuleJ]{[z:\C]\ P(z) & \vdash \Phi(z)}
\end{prooftree}
\end{minipage}
\]
\end{example}
\begin{example}[Presheaves are cartesian closed]
\label{thm:derivation_presheaves_closed}
For any $[\C]\ A,B,\Phi$, the following derivation shows that the internal $\hom$ in the category of presheaves and naturals~\cite[6.3.20]{Leinster2014basic} defined by $(A \Rightarrow B)(x) := \Nat(\hom(x,-)\x A,B)$ is indeed the correct one. We show here the tensor/hom adjunction: %
\begin{adju}[1]
\begin{prooftree}
\infer[no rule]0{[x:\C]~\Phi(x) \vdash & (A \Rightarrow B)(x) := \Nat(\hom_\C(x,-) \times A, B)}
\infer[no rule]1{ = & \Endf{y:\C}\ {\hom}_\C(x,\overline y) \x A(\overline y) \Rightarrow B(y)}
\infer[double]1[\Ruleend]{[x:\C,y:\C]~\Phi(x) \vdash & \hom_\C(x,\overline y) \x A(\overline y) \Rightarrow B(y)}
\infer[double]1[\Ruleexp]{[x:\C,y:\C]\ A(y) \x \hom_\C(\overline x,y) \x \Phi(x)  \vdash & B(y)}
\infer[double]1[\Rulecoendfrobenius]{[y:\C]\ A(y) \x \left(\Coendf{x:\C}\ {\hom}_\C(\overline x,y) \x \Phi(x) \right)  \vdash & B(y)}
\infer[double]1[\Rulecoyoneda]{[y:\C]\ A(y) \x \Phi(y) \vdash & B(y)}
\end{prooftree}
\end{adju}
\vspace{0.3em}
 We precompose with the $\Rulecoyoneda$ isomorphism given in \Cref{thm:derivation_coyoneda} (which is a \emph{natural} isomorphism). Note that $\RuleJ$ cannot be applied immediately since $y$ appears positively in context in $A(y)$, whereas it should be negative to identify it with $x$. The above derivation is a simple application of our rules via dinaturality, but it is unclear how it can be captured using the proarrow equipment approach of~\cite{New2023formal,Wood1982abstract} as an abstract property of $\Prof$, due to the repetition of variables $y,\n y$.
\end{example}

\begin{example}[Pointwise formula for right Kan extensions]
\label{exa:right_kan}
Using our rules, we give a logical proof that the functor $\Ran_F : [\C,\Set] \> [\D,\Set]$ sending (co)pre\-she\-aves to their Kan extensions along $F : \C \> \D$ computed via ends~\cite[2.3.6]{Loregian2021coend} is right adjoint to precomposition $(F\<-) : [\D,\Set] \> [\C,\Set]$. We again precompose with the $\Rulecoyoneda$ isomorphism, which we reindex implicitly with $F$. Note the similarity between this derivation and the argument given in~\cite[5.6.6]{Pitts1995categorical} to compute adjoints in a general doctrine. For any $[x:\C]\ P(x)$, $[y:\D]\ \Phi(y)$, a functor/term $F:\C\>\D$:%
\vspace{-0.3em}%
\[%
\begin{prooftree}
\infer[no rule]0{[y:\D]~\Phi(y) \vdash & (\Ran_F P)(y) := \Endf{x:\C}\ {\hom}_\D(y,F(\overline x)) \Rightarrow P(x)}
\infer[double]1[\Ruleend]{[x:\C,y:\D]~\Phi(y) \vdash & \hom_\D(y,F(\overline x)) \Rightarrow P(x)}
\infer[double]1[\Ruleexp]{[x:\C,y:\D]\ {\hom}_\D(\overline y,F(x)) \x \Phi(y) \vdash & P(x)}
\infer[double]1[\Rulecoendfrobenius]{[x:\C]~\Coendf{y:\D}\ {\hom}_\D(\overline y,F(x)) \x \Phi(y) \vdash & P(x)}
\infer[double]1[\Rulecoyoneda]{[y:\C]~\Phi(F(x)) \vdash & P(x)}
\end{prooftree}
\]
\end{example}%
\vspace{-0.7em}%
\begin{example}[Fubini rule for ends]
\label{ex:fubini}%
For convenience we only show the case for ends. %
For $[\,]\  \Phi \propctx$ in the empty context (i.e., just an object $\sem{\Phi}:\Set$) and $[\C,\D]\ P\prop$ the following are all equivalent
thanks to the fact that certain structural properties of contexts hold by cartesianness of $\Cat$.
\[
\begin{minipage}{0.5\textwidth}
\begin{prooftree}
\infer[no rule]0{[\,]~\Phi \vdash & \Endf{x:\C}~\Endf{y:\D} P(\nxx,\overline y,y)}
\infer[double]1[\Ruleend]{[x:\C]~\Phi \vdash & \Endf{y:\D} P(\nxx,\overline y,y)}
\infer[double]1[\Ruleend]{[x:\C,y:\D]~\Phi \vdash & P(\nxx,\overline y,y)}
\infer[double]1[(\textsf{structural property})]{[y:\D,x:\C]~\Phi \vdash & P(\nxx,\overline y,y)}
\end{prooftree}
\end{minipage}
\begin{minipage}{0.5\textwidth}
\begin{prooftree}
\hypo{\cdots}
\infer[double]1[(\textsf{structural property})]{[p:\C\times\D]~\Phi \vdash & P(\overline p,p)}
\infer[double]1[\Ruleend]{[y:\D]~\Phi \vdash & \Endf{x:\C} P(\nxx,\overline y,y)}
\infer[double]1[\Ruleend]{[\,]~\Phi \vdash & \Endf{y:\D}~\Endf{x:\C} P(\nxx,\overline y,y)}
\infer[double]1[\Ruleend]{[\,]~\Phi \vdash & \Endf{p:\C\x\D} P(\nxx,\overline y,y)}
\end{prooftree}
\end{minipage}
\]
\end{example}
\vspace{-0.6em}%
\begin{example}[$\Rightarrow$ resp. limits]
  Ends are limits \cite{Loregian2021coend}, and functors $-\!\Rightarrow\!- : \Set^\op \x \Set \> \Set$ preserve them (ends/limits in $\Set^\op$, i.e., coends/colimits in $\Set$).
  For $[]\ \Phi \propctx,[]\ Q\prop, [\C]\ P \prop$:
  \[
  \begin{minipage}{0.5\textwidth}
  \begin{prooftree}
  \infer[no rule]0{[\,]\ \Phi \vdash & Q \Rightarrow \Endf{x:\C} P(\nxx)}
  \infer[double]1[\Ruleexp]{[\,]\ Q, \Phi \vdash & \Endf{x:\C} P(\nxx)}
  \infer[double]1[\Ruleend]{[x:\C]\ Q, \Phi \vdash & P(\nxx)}
  \infer[double]1[\Ruleexp]{[x:\C]\ \Phi \vdash & Q \Rightarrow P(\nxx)}
  \infer[double]1[\Ruleend]{[\,]\ \Phi \vdash & \Endf{x:\C} \left(Q \Rightarrow P(\nxx)\right)}
  \end{prooftree}
  \end{minipage}
  \begin{minipage}{0.5\textwidth}
  \begin{prooftree}
  \infer[no rule]0{[\,]\ \Phi \vdash (\Coendf{x:\C} P(\nxx) ) \Rightarrow Q}
  \infer[double]1[\Ruleexp]{[\,]\ (\Coendf{x:\C} P(\nxx)), \Phi \vdash Q}
  \infer[double]1[\Rulecoendfrobenius]{[x:\C]\ P(\nxx), \Phi \vdash & Q}
  \infer[double]1[\Ruleexp]{[x:\C]\ \Phi \vdash & P(\xnx) \Rightarrow Q}
  \infer[double]1[\Ruleend]{[\,]\ \Phi \vdash & \Endf{x:\C} P(\xnx) \Rightarrow Q}
  \end{prooftree}
  \end{minipage}
  \]
\end{example}
\vspace{-0.3em}%
\section{Conclusions and Future Work}\label{sec:conclusion}

In this paper we showed how dinaturality is the key notion to give a simple and natural description to a first-order directed type theory where types are interpreted as (1-)categories and directed equality as $\hom$-functors. Our type theory is powerful enough to express theorems about directed equality in a straightforward way, and to give a distinctly logical interpretation to well-known theorems in category theory by reinterpreting them under the light of directed type theory.

\textbf{Dinaturality.} The compositionality problem of dinatural transformations is a long-standing and famously difficult problem~\cite{Santamaria2019towards}, which both the category theory and computer science communities have relatively left unexplored since their introduction in the 1970s~\cite{Eilenberg1966generalization,Dubuc1970dinatural}. Our work gives a concrete motivation to further investigate this more than 50-years old mystery by connecting it to directed type theory. We conjecture that this connection could possibly hint to a deeper \emph{directed homotopical} reason~\cite{Fajstrup2016directed,Grandis2009directed} for why dinaturals fail to compose.
Strong dinaturals~\cite{Neumann2023paranatural,Pare1998dinatural} are one possible approach to deal with non-compositionality, but they lack in expressivity, e.g., they are not closed in general~\cite{Uustalu2010note}. Following \Cref{thm:dinaturals_groupoid_compose}, this non-compositionality is intrinsic to the directed proof-relevant setting, i.e., non-groupoidal categories. We leave investigating the relation between dinaturality and geometric models of $(\infty,1)$-categories in the spirit of~\cite{Riehl2017type,Gratzer2024directed,Weaver2020constructive} for future work.

\textbf{Type dependency.} Our treatment of directed equality via dinaturality is a first step towards understanding the precise interplay of polarity, directedness and variance in fully dependent \MLTT, especially with respect to how polarity of variables is influenced by their appearance in types, which we conjecture to be particularly non-trivial.

\textbf{Initiality.} The syntactic system presented in this paper could be axiomatized into a suitable initial object in a category of models that captures the behavior of variables in dinaturals and naturals (e.g., as in~\cite{Santamaria2019towards}): one possible approach could be to abstractly consider two classes of maps (dinaturals, naturals) and requiring such maps to interact as in \Cref{thm:nat_dinat_compose}.

\textbf{Doctrines.} All of our results can be specialized in the category of posets $\Pos$ rather than $\Cat$, where dinaturals compose trivially and our work provides a ``logic of posets'', captured via a bona fide doctrine, at the cost of trivializing (co)ends with (co)products. This posetal case could be axiomatized in the style of the doctrinal approach~\cite{Jacobs1999categorical,Maietti2015unifying}, with a notion of \emph{directed doctrine} capturing the roles pl\-ay\-ed by va\-ri\-an\-ce, the $-^\op$ involution, and (di)naturality. This would allow our syntactic rules to be organized in a well-known structure, with a suitable initiality result. %

\textbf{Internalizing Yoneda.} The Yoneda technique for isomorphisms follows from ``manually'' using naturality of isomorphisms in the equational theory. One could also get this naturality for free by making the theory second-order with a universe of propositions $\Set$ and adding a directed univalence statement $\hom_\Set(A,B) \iso A \Rightarrow B$ (as in~\cite{Altenkirch2024synthetic,Gratzer2024directed,Weaver2020constructive}): this would allow for implication to be represented as a directed equality, contractible with $\RuleJ$, and ``synthetically'' reproduce the same argument as in \Cref{ex:internal_dinaturality} by quantifying over all predicates involved.

\textbf{Higher (co)end calculus.} There are other conceptual examples of coend calculus which have not yet been interpreted in terms of directed equality: for instance, one should be able to express that composition maps exist \emph{for all} categories $\C : \Cat$, where this quantification can be expressed via a suitable pseudo-end in $\Cat$~\cite[7.1]{Loregian2021coend}; similarly, the category of elements of a functor, reminiscent of a $\Sigma$-type, can be given as the pseudo-coend $\El(F) \iso \Coendf{c:\C} c/\C \x F(c)$, where $c/\C$ is the coslice category and $F(c)$ is seen as a discrete category~\cite[4.2.2]{Loregian2021coend}. These examples could be captured by considering the category of small categories $\Cat$ as a suitable universe of types~\cite{Hofmann1997syntax}.

\textbf{Enrichment.} We do not rely on specific properties of $\Set$ (viewed as the base of enrichment of $\Cat$), other than cartesian closedness to have propositional implication/conjunction and the existence of (co)limits to express (co)ends. We conjecture that our analysis of dinaturals can be developed in more generality by taking enriched categories (over a sufficiently structured base of enrichment) as types, rather than simply categories (enriched over $\Set$).

\textbf{Implementation.} We remark how an implementation of the metatheory of our type-theoretical system in a proof assistant is non-trivial, since one has to push $-^\op$ down into connectives and ensure that $({X^\op})^\op \equiv X$ everywhere in the syntax: in types, contexts, terms, predicates, propositional contexts. This could be tackled in practice by using QITs~\cite{Altenkirch2016type} and the \texttt{-\hspace{0.1pt}-rewriting} feature of Agda~\cite{Cockx2021taming} to simplify $\op$ whenever necessary. Another solution would be to have $-^\op$ only at the level of base types, and then derive $-^\op$ as a metatheoretical operation on full types; this has the disadvantage that $-^\op$ is not a primitve type former that one can explicitly manipulate in the syntax.
\begin{acks}
The authors thank the reviewers for their detailed suggestions and {Pawe\l} Soboci\'nski for invaluable feedback on the presentation of this work. Loregian was supported by the Estonian Research Council grant PRG1210. Veltri was supported by the Estonian Research Council grant PSG749.
\end{acks}

\appendix

\clearpage

\section{Additional judgments for first-order dinatural directed type theory}
\label{appendix:syntax}

The rules to formally capture the variance of variables in predicates is given in \Cref{fig:syntax:pos_neg_conditions_formulas_both}, with the accompanying definition of unused variables in terms in \Cref{fig:syntax:unused_terms}.

We show in \Cref{sec:appendix:other_equational} the full rules in the equational theory regarding cuts. In \Cref{sec:appendix:end_bidirectional_example} we explicitly illustrate what a bidirectional rule in ``adjoint-form'' looks like, by showing the two directions, the isomorphisms and the naturality conditions.

  \begin{figure}[H]
  \[
  \begin{array}{c}
\fbox{$\Gamma \ni x:\A \mathsf{\ unused\ in\,}t:\C$}
\quad \inferenceTwo{\Gamma \ni x : \C}{x \not = y}{\Gamma \ni y:\C \mathsf{\ unused\ in\,} x:\C}
\quad \inference{\phantom{\Gamma\!\!\!\vdash s : \C}}{\Gamma \ni x:\A \mathsf{\ unused\ in\,} ! : \top}
\\[1em]
\quad \inference{\Gamma \ni x:\A \mathsf{\ unused\ in\,} t : \dom(f)}{\Gamma \ni x:\A \mathsf{\ unused\ in\,} f(t) : \cod(f)}
\quad \inference{\Gamma \ni x:\A \mathsf{\ unused\ in\,} t : \C}{\Gamma^\op \ni x:\Aop \mathsf{\ unused\ in\,} t^\op : \C^\op}
\\[1em]
\quad \inferenceTwo{\Gamma \ni x:\A \mathsf{\ unused\ in\,} s : \C}{\Gamma \ni x:\A \mathsf{\ unused\ in\,} t : \D}{\Gamma \ni x:\A \mathsf{\ unused\ in\,} \ang{s,t} : \C \times \D}
\\[1em]
\quad \inference{\Gamma \ni x:\A \mathsf{\ unused\ in\,} p : \C \times \D}{\Gamma \ni x:\A \mathsf{\ unused\ in\,} \pi_1(p) : \C}
\quad \inference{\Gamma \ni x:\A \mathsf{\ unused\ in\,} p : \C \times \D}{\Gamma \ni x:\A \mathsf{\ unused\ in\,} \pi_2(p) : \D}
\\[1em]
\quad \inferenceTwo{\Gamma \ni x:\A \mathsf{\ unused\ in\,} s : [\C,\D]}{\Gamma \ni x:\A \mathsf{\ unused\ in\,}t : \C}{\Gamma \ni x:\A \mathsf{\ unused\ in\,} s \cdot t : \D}
\quad \inference{\Gamma,z:\C \ni x:\A \mathsf{\ unused\ in\,} t(z) : \D}{\Gamma \ni x:\A \mathsf{\ unused\ in\,} \lambda z.t(z) : [\C,\D]}
  \end{array}
  \]
  \caption{Syntax of first-order dinatural directed type theory -- syntactically unused variables in terms.}
  \label{fig:syntax:unused_terms}
  \[
  \begin{array}{c}
\fbox{$\Gamma \ni x:\A \mathsf{\ cov\ in\,}\varphi$}
\\[1em]
\inferenceTwo{\Gamma \ni x:\A \mathsf{\ cov\ in\,} P}{\Gamma \ni x:\A \mathsf{\ cov\ in\,}Q}{\Gamma \ni x:\A \mathsf{\ cov\ in\,}P \times Q}
\quad \inferenceTwo{\Gamma^\op \ni x:\A^\op \mathsf{\ cov\ in\,} P}{\Gamma \ni x:\A \mathsf{\ cov\ in\,}Q}{\Gamma \ni x:\A \mathsf{\ cov\ in\,}P \Rightarrow Q}
\\[1em]
\inference{\phantom{\Gamma \ni x:\A \mathsf{\ cov\ in\,}\varphi}}{\Gamma \ni x:\A \mathsf{\ cov\ in\,}\top}
\quad \inference{\Gamma,y:\C \ni x:\A \mathsf{\ cov\ in\,}\varphi}{\Gamma \ni x:\A \mathsf{\ cov\ in\,}\Coendf{y:\C} \varphi(\n y,y)}
\quad \inference{\Gamma,y:\C \ni x:\A \mathsf{\ cov\ in\,}\varphi}{\Gamma \ni x:\A \mathsf{\ cov\ in\,}\Endf{y:\C} \varphi(\n y,y)}
\\[1.5em]
\quad \inferenceTwo{\Gamma^\textsf{op},\Gamma \ni \n x:\Aop \mathsf{\ unused\ in\,} s : \Cop}{\Gamma^\textsf{op},\Gamma \ni \n x:\A^\textsf{op} \mathsf{\ unused\ in\,} t : \C}{\Gamma \ni x:\A \mathsf{\ cov\ in\,}\hom_\C(s, t)}
\\[1.5em]
\quad \inferenceTwo{\Gamma^\textsf{op},\Gamma \ni \n x:\Aop \mathsf{\ unused\ in\,} s : \textsf{neg}(P)^\op}{\Gamma^\textsf{op},\Gamma \ni \n x:\A^\textsf{op} \mathsf{\ unused\ in\,} t : \textsf{pos}(P)}{\Gamma \ni x:\A \mathsf{\ cov\ in\,}P(s \mid t)}
\\[1.5em]
\inferenceThree{\A = \A'}{\varphi = \varphi'}{\Gamma \ni x:\A \mathsf{\ cov\ in\,} \varphi}{\Gamma \ni x:\A' \mathsf{\ cov\ in\,} \varphi'}
\\[1em]
\fbox{$\Gamma \ni x:\A \mathsf{\ contra\ in\,}\varphi$}
\\[1em]
\quad
\inference{\Gamma^\op \ni x:\A^\op \mathsf{\ contra\ in\,} \varphi^\op}{\Gamma \ni x:\A \mathsf{\ contra\ in\,} \varphi}
\end{array}
\]
\Description{This figure shows the syntax rules for determining whether a variable is covariant or contravariant in a formula, as well as rules for determining whether a variable is unused in a term. The rules for unused variables include cases for variables in product types, function types, and other constructs. The covariant and contravariant rules include cases for logical connectives, quantifiers, and hom-sets.}
\caption{Syntax of first-order dinatural directed type theory -- syntactic conditions for covariant/contravariant variables in predicates.}
\label{fig:syntax:pos_neg_conditions_formulas_both}
\end{figure}

\begin{figure}
\[
\begin{array}{c}
  \hspace{-4.1em}{\fbox{$[\Gamma]\ \Phi \vdash \alpha = \beta : P$}} \quad \cdots
  \\[1em]
  \begin{prooftree}
    \infer[no rule]0{\Gamma \textsf{ unused in }P\textsf{ and }Q}
    \infer[no rule]1{[a:\Delta^\op,b:\Delta,x:\Gamma]\ \takespace{\,k:Q(a,b),}{} \Phi(a,b,\nxx) & \vdash \takespace[l]{\gamma[\alpha]}{\alpha} : P(a,b)}
     \infer[no rule]1{[z:\Delta,x:\Gamma]\ k:\takespace{Q(a,b)}{P(\nzz)},\takespace[l]{\Phi(a,b,\nxx)}{\Phi(\nzz,\nxx)} & \vdash \takespace[l]{\gamma[\alpha]}{\gamma[k]} : Q(\nzz)}
    \infer[no rule]1{[a:\Delta^\op,b:\Delta,x:\Gamma]\ k:Q(a,b),\takespace[l]{\Phi(\n b,\n a,\nxx)}{\Phi(\n b,\n a,\nxx)} & \vdash \takespace[l]{\gamma[\beta]}{\beta[k]} : R(a,b,\nx,x)}
  \infer1[\Rulecutassoc]{[z:\Delta,x:\Gamma]\ \takespace{\,k:Q(a,b),}{} \Phi(\nzz,\nxx) & \vdash (\beta[\gamma])[\alpha] = \beta[\gamma[\alpha]] : R(\nzz,\nx,x)}
  \end{prooftree}
  \\[3.5em]
  \begin{prooftree}
    \infer[no rule]0{\Gamma \textsf{ unused in }P, \quad \Delta \textsf{ unused in }\Phi}
    \infer[no rule]1{[a:\Delta]\ \takespace{k:P(a),\Phi}{k:P(a),\Phi} & \vdash \takespace[l]{\gamma[\beta]}{\alpha[k]} : Q(a)}
    \infer[no rule]1{[a:\Delta]\ \takespace{k:Q(a),\Phi}{r:Q(a),\Phi} & \vdash \takespace[l]{\gamma[\beta]}{\beta[r]} : R(a)}
  \infer1[\Rulenatcutcoherence]{[a:\Delta]\ \takespace{k:P(a),\Phi}{k:P(a),\Phi} & \vdash \beta[\alpha]^\textsf{cut-nat} = \beta[\alpha]^\textsf{cut-din} : Q(a)}
  \end{prooftree}
  \\[3.5em]
  \begin{prooftree}
    \infer[no rule]0{\Gamma \textsf{ unused in }P}
    \infer[no rule]1{[z:\Delta,x:\Gamma]\ k:P(\n z,z),\Phi(\n z,z,\n x,x) & \vdash k : P(\n z,z)}
    \infer[no rule]1{[a:\Delta^\op,b:\Delta,x:\Gamma]\ \phantom{k:}P(a,b),\Phi(\n b,\n a,\n x,x) & \vdash \alpha : Q(a,b)}
  \infer1[\Rulecutnatidl]{[z:\Delta,x:\Gamma]\ \phantom{k:}P(\n z,z),\Phi(\n z,z,\n x,x) & \vdash \alpha[k] = \alpha^{a,b \mapsto z} : Q(\n z,z)}
  \end{prooftree}
  \\[2.4em]
  \begin{prooftree}
    \infer[no rule]0{\Gamma \textsf{ unused in }Q}
    \infer[no rule]1{[z:\Delta,x:\Gamma]\ \Phi(\n z,z,\n x,x) & \vdash \alpha : Q(\n z,z)}
    \infer[no rule]1{[a:\Delta^\op,b:\Delta,x:\Gamma]\ k:P(a,b),\Phi(\n b,\n a,\n x,x) & \vdash k : P(a,b)}
  \infer1[\Rulecutnatidr]{[z:\Delta,x:\Gamma]\ \Phi(\n z,z,\n x,x) & \vdash k[\alpha] = \alpha : Q(\n z,z)}
  \end{prooftree}
  \\[2.5em]
  \begin{prooftree}
    \infer[no rule]0{\Gamma \textsf{ unused in }P}
    \infer[no rule]1{[a:\Delta^\op,b:\Delta,x:\Gamma]\ k:P(a,b),\Phi(a,b,\n x,x) & \vdash k : P(a,b)}
    \infer[no rule]1{[z:\Delta,x:\Gamma]\ \takespace[r]{k:P(\n x,x)}{P(\n z,z)},\Phi(\n z,z,\n x,x) & \vdash \alpha : Q(\n z,z)}
  \infer1[\Rulecutdinidl]{[z:\Delta,x:\Gamma]\ \takespace[r]{k:P(\n x,x)}{P(\n z,z)},\Phi(\n z,z,\n x,x) & \vdash \alpha[k] = \alpha : Q(\n z,z)}
  \end{prooftree}
  \\[2.4em]
  \begin{prooftree}
    \infer[no rule]0{\Gamma \textsf{ unused in }Q}
    \infer[no rule]1{[a:\Delta^\op,b:\Delta,x:\Gamma]\ \Phi(a,b,\n x,x) & \vdash \alpha : Q(a,b)}
    \infer[no rule]1{[z:\Delta,x:\Gamma]\ k:Q(\n z,z),\Phi(\n z,z,\n x,x) & \vdash k : Q(\n z,z)}
  \infer1[\Rulecutdinidr]{[z:\Delta,x:\Gamma]\ \Phi(\n z,z,\n x,x) & \vdash k[\alpha] = \alpha^{a,b\mapsto z} : Q(\n z,z)}
  \end{prooftree}
\end{array}
\]
\Description{This figure shows the equational rules for cuts in first-order directed type theory. The rules include associativity for natural-dinatural-natural cuts, coherence for cuts between naturals, and left and right identities for cut. Each rule is presented in a prooftree format, illustrating the derivation steps leading to the final equality statement.}
\caption{Syntax of first-order directed type theory -- Equational rules for cuts: associativity for natural-dinatural-natural cuts, coherence for cuts between naturals, left and right identities for cut.}
\label{sec:appendix:other_equational}
\LinkRulecutassoc
\LinkRulenatcutcoherence
\LinkRulecutnatidl
\LinkRulecutnatidr
\LinkRulecutdinidl
\LinkRulecutdinidr
\end{figure}

\begin{figure}
\[
\begin{array}{c}
{\fbox{$[\Gamma]\ \Phi \vdash \alpha : P$}}
\quad \cdots
\\[1em]
\begin{prooftree}
\infer[no rule]0{[x: \C, \Gamma]\ \Phi & \vdash \alpha : P(\nxx)}
\infer1[\Ruleend]{\takespace{[x: \C, \Gamma]}{[\Gamma]}\ {\Phi} & \vdash \mathsf{end}(\alpha) : \Endf{x:\C} P(\nxx)}
\end{prooftree}
\quad
\begin{prooftree}
\infer[no rule]0{\takespace{[x: \C, \Gamma]}{[\Gamma]}\ \Phi & \vdash \alpha : \Endf{x:\C} P(\nxx)}
\infer1[\Ruleendinv]{[x: \C, \Gamma]\ \Phi & \vdash \mathsf{end}^{-1}(\alpha) : P(\nxx)}
\end{prooftree}
\\[2.5em]
  {\fbox{$[\Gamma]\ \Phi \vdash \alpha = \beta : P$}} \quad \cdots
\\[1em]
  \begin{prooftree}
  \infer[no rule]0{[x: \C, \Gamma]\ \Phi & \vdash \alpha : P(\nxx)}
  \infer1{[x: \C, \Gamma]\ {\Phi} & \vdash \mathsf{end}^{-1}(\mathsf{end}(\alpha)) = \alpha : P(\nxx)}
  \end{prooftree}
  \quad
  \begin{prooftree}
  \infer[no rule]0{[\Gamma]\ \Phi & \vdash \alpha : \Endf{x:\C} P(\nxx)}
  \infer1{[\Gamma]\ {\Phi} & \vdash \mathsf{end}(\mathsf{end}^{-1}(\alpha)) = \alpha : \Endf{x:\C} P(\nxx)}
  \end{prooftree}
\\[2.5em]
\begin{prooftree}
\infer[no rule]0{{[z:\Delta,\Gamma]}\ \takespace{k:Q(a,b),\Phi(\n a,\n b)}{\Phi(\n z,z)} & \vdash \beta : Q(\n z,z)}
\infer[no rule]1{{[a:\Delta^\op,b:\Delta,x:\C,\Gamma]}\ k:Q(a,b),\Phi(\n a,\n b) & \vdash \alpha : P(\n x,x,a,b)}
\infer1[\Ruleendnatl]{[x:\C,z:\Delta,\Gamma]\ \takespace{k:Q(a,b),\Phi(\n a,\n b)}{\Phi(\n z,z)} & \vdash \mathsf{end}(\alpha)[\beta] = \mathsf{end}(\alpha[\beta]) : \Endf{x:\C} P(\n x,x,\n z,z)}
\end{prooftree}
\\[2.5em]
\begin{prooftree}
\infer[no rule]0{{[a:\Delta^\op,b:\Delta,\Gamma]}\ \takespace{k:Q(\n z,z),\Phi(\n z,z)}{\Phi(a,b)} & \vdash \beta : Q(a,b)}
\infer[no rule]1{{[x:\C,z:\Delta,\Gamma]}\ k:Q(\n z,z),\Phi(\n z,z) & \vdash \alpha : P(\n x,x,\n z,z)}
\infer1[\Ruleenddinl]{[x: \C,z:\Delta,\Gamma]\ \takespace{k:Q(\n z,z),\Phi(\n z,z)}{\Phi(\n z,z)} & \vdash \mathsf{end}(\alpha)[\beta] = \mathsf{end}(\alpha[\beta]) : \Endf{x:\C} P(\n x,x,\n z,z)}
\end{prooftree}
\\[2.5em]
\begin{prooftree}
\infer[no rule]0{{[x_1\!:\!\Cop,x_2\!:\!\C,a\!:\!\Delta^\op,b\!:\!\Delta]}\ P(x_1,x_2,a,b) & \vdash \beta : P'(x_1,x_2,a,b)}
\infer[no rule]1{{[x:\C,z:\Delta]}\ \Phi(\nzz) & \vdash \alpha : P(\n x,x,\nzz)}
\infer1[\Ruleenddinr]{[z:\Delta]\ {\Phi(\n z,z)} & \vdash \mathsf{end}_F(\beta)[\mathsf{end}(\alpha)] = \mathsf{end}(\beta[\alpha]) : \Endf{x:\C} P'(\n x,x,\n z,z)}
\end{prooftree}
\\[2.5em]
\begin{prooftree}
\infer[no rule]0{{[x_1\!:\!\Cop,x_2\!:\!\C,z\!:\!\Delta]}\ Q(x_1,x_2,\n z,z) & \vdash \beta : P'(x_1,x_2,\n z,z)}
\infer[no rule]1{{[x:\C,a:\Delta^\op,b:\Delta]}\ \Phi(a,b) & \vdash \alpha : Q(\n x,x,a,b)}
\infer1[\Ruleendnatr]{[a:\Delta^\op,b:\Delta]\ {\Phi(a,b)} & \vdash \mathsf{end}_F(\beta)[\mathsf{end}(\alpha)] = \mathsf{end}(\beta[\alpha]) : \Endf{x:\C} P'(\n x,x,a,b)}
\end{prooftree}
\end{array}
\]
\Description{This figure shows the bidirectional rules for ends in first-order directed type theory. It includes the introduction and elimination rules for ends, as well as the associated isomorphisms and naturality conditions. Each rule is presented in a prooftree format, illustrating the derivation steps leading to the final judgment.}
  \caption{Syntax of first-order directed type theory -- Explicit description of a rule in ``adjoint-form'', e.g., for ends: rules, isomorphisms, and naturality in $\Phi,P$ for $\Ruleend$. Naturality in $P$ uses functoriality in \Cref{sec:appendix:functoriality_ends}.}
  \label{sec:appendix:end_bidirectional_example}
  \LinkRuleendinv
  \LinkRuleendnatl
  \LinkRuleenddinl
  \LinkRuleendnatr
  \LinkRuleenddinr
\end{figure}

\begin{figure}
\[
\begin{prooftree}
\infer[no rule]0{{[a:\Cop,b:\C,\Gamma]}\ k:P(a,b) & \vdash \alpha[k] : P(a,b)}
\infer1{[x: \C, \Gamma]\ p:\End{x:\C} P(\nxx) & \vdash \mathsf{end}_F(\alpha) := \mathsf{end}(\alpha[\mathsf{end}^{-1}(p)]) }
\infer[no rule]1{& = \mathsf{end}(\alpha)[\mathsf{end}^{-1}(p)] : \End{x:\C} P(\nxx)}
\end{prooftree}
\]
\Description{This figure illustrates the functoriality of ends for natural transformations in first-order directed type theory.}
\caption{Functoriality of ends for \emph{naturals} by precomposing with the counit of $\Ruleend$.}
\label{sec:appendix:functoriality_ends}

\end{figure}

\clearpage

\section{Directed type theory, other derivations}
\label{sec:appendix:dtt_other_derivations}

\begin{example}[Contractibility of singletons]
\label{thm:contractibility}
Recall the derivation for existence of singletons:
\[
\begin{prooftree}
\infer0{[ ]\ \cdot & \vdash \mathsf{end}(\mathsf{coend}^{-1}(k)[\refl_x]) : \End{x:\Cop} \Coend{y:\C} \hom_\C(x,y)}
\end{prooftree}
\]
We now show that singletons are actually contractible: assuming another element $k:\Coendf{y:\C} \hom_\C(x,y)$, we show that it is equal to the the one given in the first derivation (after removing the universal quantifier). Note that the right-hand side must cut away the hypothesis $k$ by precomposing with the constant dinatural $!$. In the bottom of the derivation we use the fact that the isomorphisms for coends are natural with respect to the cut rules of our type theory. In the top of the derivation we omit for simplicity an application of associativity of cuts and uniqueness of $!$ which is used to remove the application of $J^{-1}$.

\[
\begin{prooftree}
\infer0[\seqrule{refl}]{[z:\C]\ \emptyctx \vdash \mathsf{coend}^{-1}(k)[\refl_z] = \mathsf{coend}^{-1}(k)[\refl_z] : \cdots}
\infer1[\seqrule{$!$-unique}+\Rulecutassoc]{[z:\C]\ \emptyctx \vdash \mathsf{coend}^{-1}(k)[\refl_z] = \mathsf{coend}^{-1}(k)[\refl_z][!][\refl_z] : \cdots}
\infer1[\RuleJeq]{[x:\Cop,y:\C]\ k:\hom_\C(x,y) & \vdash \mathsf{coend}^{-1}(k) = \mathsf{coend}^{-1}(k)[\refl_x][!] : \cdots}
\infer1[\seqrule{$!$-unique}]{ & \takespace[r]{\ \mathsf{coend}^{-1}(s)}{\cdots} = \mathsf{coend}^{-1}(k)[\refl_x][\mathsf{coend}^{-1}(!)] : \cdots}
\infer1[\seqrule{coend-natural}]{[x:\Cop,y:\C]\ k:\hom_\C(x,y) & \vdash \mathsf{coend}^{-1}(s) = \mathsf{coend}^{-1}(\mathsf{coend}^{-1}(k)[\refl_x][!]) : \cdots}
\infer1[\Rulecoendfrobenius]{[x:\Cop]\ k:\Coend{y:\C} \hom_\C(x,y) & \vdash k = \mathsf{coend}^{-1}(k)[\refl_x][!] : \Coend{y:\C} \hom_\C(x,y)}
\end{prooftree}
\]
\end{example}

\begin{example}[Internal naturality for natural transformations]
  \label{ex:internal_naturality_terms}
  We show that naturality for natural transformations between terms, expressed as ends \cite[1.4.1]{Loregian2021coend}, holds internally by directed equality elimination. Given terms $\C \vdash F,G : \D$, we use the counit of $\Ruleend$ to extract the family of $\hom$-sets. We first explicitly show the rules used to construct the two sides of a naturality square:
  \[
  \begin{prooftree}
    \hypo{[a:\Cop, b:\C]\,f:\hom_\C(a,b), \eta : \Endf{x:\C} \hom_\D{F(\nx),G(x)} \vdash \eta : \Endf{x:\C} \hom(F(\n x), G(x))}
    \infer1[\Ruleendinv]{[a:\Cop, b:\C, x:\C]\, f:\hom_\C(a,b), \eta : ... \vdash \mathsf{end}^{-1}(\eta) : \hom(F(\n x), G(x))}
    \infer1[\Rulereindex]{[a:\Cop, b:\C]\, f:\hom_\C(a,b), \eta : ... \vdash \Delta^*(\mathsf{end}^{-1}(\eta)) : \hom(F(a), G(\n a))}
    \infer1[\Rulecutnat]{[a:\Cop, b:\C]\, f:\hom_\C(a,b), \eta : ... \vdash \comp[\Delta^*(\mathsf{end}^{-1}(\eta)),\congr_G[f]] : \hom(F(a), G(b))}
    \end{prooftree}
    \]
  where $\Delta^*$ is the reindexing functor which collapses $a,x$ to a single variable $a$, and $\Rulecutnat$ is used to apply $\comp$ on $\congr$ for $G$. This composition can be done since both $\congr$ and $\comp$ have the correct naturality shape that allows for $\Rulecutnat$ to be applied.

  The other derivation is obtained similarly:

  \[
  \begin{prooftree}
    \hypo{[a:\Cop, b:\C]\,f:\hom_\C(a,b), \eta : \Endf{x:\C} \hom_\D{F(\nx),G(x)} \vdash \eta : \Endf{x:\C} \hom(F(\n x), G(x))}
    \infer1[\Ruleendinv]{[a:\Cop, b:\C, x:\C]\, f:\hom_\C(a,b), \eta : ... \vdash \mathsf{end}^{-1}(\eta) : \hom(F(\n x), G(x))}
    \infer1[\Rulereindex]{[a:\Cop, b:\C]\, f:\hom_\C(a,b), \eta : ... \vdash \Delta^*(\mathsf{end}^{-1}(\eta)) : \hom(F(\n b), G(b))}
    \infer1[\Rulecutnat]{[a:\Cop, b:\C]\, f:\hom_\C(a,b), \eta : ... \vdash \comp[\congr_F[f],\Delta^*(\mathsf{end}^{-1}(\eta))] : \hom(F(a), G(b))}
    \end{prooftree}
  \]

  We show that the two maps constructed, corresponding to the two sides of a naturality square, are equal using directed equality elimination; let $K := \Delta^*(\mathsf{end}^{-1}(\eta))$:

  \[
  \begin{prooftree}
    \infer0{[z:\C]\, ... \vdash K = K : \hom(F(\n z), G(z))}
    \infer1[\RuleJcomp]{[z:\C]\, ... \vdash \comp[\refl_z,K] = \comp[K,\refl_z] : \hom(F(\n z), G(z))}
    \infer1[\RuleJcomp]{[z:\C]\, ... \vdash \comp[\congr_F[\refl_z],K] = \comp[K,\congr_G[\refl_z]] : \hom(F(\n z), G(z))}
    \infer1[\RuleJeq]{[a:\Cop, b:\C]\, f:\hom_\C(a,b), ... \vdash \comp[\congr_F[f],K] = \comp[K,\congr_G[f]]: \hom(F(a), G(b))}
    \end{prooftree}
  \]
  where the equations used follow by the computation rules for $\congr$ and left and right unitality of $\comp$. Note that $\RuleJeq$ can be used since $a,b$ appear precisely with the correct types that allow for $\RuleJ$ to be applied to contract the equality.

  This naturality can then be used to prove a suitable internal Yoneda lemma for the $\hom$ of categories by following the standard argument, e.g., given in \cite{Leinster2014basic}.
  \end{example}

\begin{example}[Identity natural transformation]
  \label{ex:identity_natural}
We show the existence of the identity natural transformation for terms, given a functor $\C \vdash F : \D$:
\[
\begin{prooftree}
  \infer0[\Rulerefl+\Rulereindex]{[x:\C]\ \emptyctx \vdash F^*(\textsf{refl}_x) : \hom_\D(F(\n x),F(x))}
  \infer1[\Ruleend]{[]\ \emptyctx \vdash \textsf{end}^{-1}(l) : \Endf{x:\C}\hom_\D(F(\n x),F(x))}
\end{prooftree}
\]
\end{example}

\begin{example}[Composition of natural transformations]
  \label{ex:compose_naturals}
We show that natural transformations between terms, expressed as an end~\cite[1.4.1]{Loregian2021coend}, can be composed. Take functors $\C \vdash F,G,H : \D$; first, consider the following elementary derivations:
\[
\begin{prooftree}
  \hypo{[\,]\,l :\Endf{x:\C}\hom_\C(F(\n x),G(x)), r : \Endf{x:\C}\hom_\C(G(\n x),H(x)) \vdash l : \Endf{x:\C}\hom_\C(F(\n x),G(x))}
  \infer1[\Ruleendinv]{[x:\C]\,l :\Endf{x:\C}\hom_\C(F(\n x),G(x)), r : \Endf{x:\C}\hom_\C(G(\n x),H(x)) \vdash \textsf{end}^{-1}(l) : \hom_\C(F(\n x),G(x))}
\end{prooftree}
\]
\[
\begin{prooftree}
  \hypo{[\,]\,l :\Endf{x:\C}\hom_\C(F(\n x),G(x)), r : \Endf{x:\C}\hom_\C(G(\n x),H(x)) \vdash l : \Endf{x:\C}\hom_\C(F(\n x),G(x))}
  \infer1[\Ruleendinv]{[x:\C]\,l :\Endf{x:\C}\hom_\C(F(\n x),G(x)), r : \Endf{x:\C}\hom_\C(G(\n x),H(x)) \vdash \textsf{end}^{-1}(r) : \hom_\C(G(\n x),H(x))}
\end{prooftree}
\]
Then, we take the statement for transitivity of directed equality, and reindex $a$ with $F(a)$, $b$ with $G(b)$, and $c$ with $H(c)$:
\[
\begin{prooftree}
  \infer0[\RuleJ]{[a:\Dop,b:\D,c:\D]\ {f:\hom_\D(a,b),\, g:\hom_\D(\n b,c)} & \vdash \textsf{comp} : \hom_\D(a,c)}
  \infer1[\Rulereindex]{[a:\Cop,b:\C,c:\C]\ {f:\hom_\D(F(a),G(b)),\, g:\hom(G(\n b),H(c))} & \vdash \textsf{comp}'[f,g] : \hom_\D(F(a),H(c))}
\end{prooftree}
\]
Now we can perform the composition of this map with the entailments above, which can be done because $\comp$ is individually \emph{natural} in $a,b$, and $b,c$. Composing $l$ into $\textsf{comp}$ contracts $a,b$ to the same variable $z$, while still allowing the other map to be later composed in the equality with $z,c$. Finally, we reintroduce the end quantifier.
\begin{adju}[1.0]
\begin{prooftree}
  \infer0[\Rulecutnat]{[z:\Cop,c:\C]\ {l:...\,,r:...\,,\, g:\hom(G(\n z),H(c))} & \vdash \textsf{comp}'[\textsf{end}^{-1}(l),g] : \hom_\D(F(z),H(c))}
  \infer1[\Rulecutnat]{[w:\C]\ {l:...\,,r:...\,,} & \vdash \textsf{comp}'[\textsf{end}^{-1}(l),\textsf{end}^{-1}(r)] : \hom_\D(F(\n w),H(w))}
  \infer1[\Ruleend]{[\,]\ {l:...\,,r:...\,,} & \vdash \textsf{end}(\textsf{comp}'[\textsf{end}^{-1}(l),\textsf{end}^{-1}(r)]) : \End{w:\C} \hom_\D(F(\n w),H(w))}
\end{prooftree}
\end{adju}
Associativity of the map above follows from associativity of $\comp$ as in the standard case.
\end{example}

\begin{example}[Directed equality in opposite categories]\label{thm:dir_eq_op}
  We do not ask that predicates $[x:\C, y:\Cop]\ {\hom}_{\Cop}(x,y)$ and $[x:\C, y:\Cop]\ {\hom}_{\C}(y,x)$ are definitionally equal in the equational theory (although this would arguably be a desirable choice), but we can prove by directed equality induction that they are isomorphic:
  \[
  \begin{prooftree}
  \infer0[\Rulerefl]{[z:\C,\Gamma]\ \phantom{f:{\hom}_{\Cop}(x,y), }\Phi & \vdash \mathsf{\refl}_z : {\hom}_\C(\n z, z)}
  \infer1[\RuleJ]{[x:\C,y:\Cop,\Gamma]\ f:{\hom}_{\Cop}(x,y), \Phi & \vdash J(\mathsf{\refl}_z)[f] : {\hom}_\C(y, x)}
  \end{prooftree}
  \]
  Rule $\RuleJ$ can be applied since $x,y$ appear covariantly in the conclusion. The inverse direction is identical:
  \[
  \begin{prooftree}
  \infer0[\Rulerefl]{[z:\C,\Gamma]\ \phantom{f:{\hom}_{\C}(y,x), }\Phi & \vdash \mathsf{\refl}_z : {\hom}_{\Cop}(z, \n z)}
  \infer1[\RuleJ]{[x:\C,y:\Cop,\Gamma]\ f:{\hom}_{\C}(y,x), \Phi & \vdash J(\mathsf{\refl}_z)[f] : {\hom}_{\Cop}(y, x)}
  \end{prooftree}
  \]
  In one direction, they compose (since they are both naturals) to the identity by directed equality induction:
  \[
  \begin{prooftree}
  \infer0[\RuleJcomp]{[z:\C,\Gamma]\ \Phi \vdash J(\mathsf{\refl}_z)[J(\mathsf{\refl}_z)[\mathsf{\refl}_z]] = J(\mathsf{\refl}_z)[\mathsf{\refl}_z] = \mathsf{\refl}_z : {\hom}_{\Cop}(z, \n z)}
  \infer1[\RuleJeq]{[x:\C,y:\Cop,\Gamma]\ f:{\hom}_{\C}(y,x), \Phi \vdash J(\mathsf{\refl}_z)[J(\mathsf{\refl}_z)[f]] = f : {\hom}_{\Cop}(y, x)}
  \end{prooftree}
  \]
  The other direction is analogous.
\end{example}

\section{Other rules derivable from the adjoint formulation}
\label{sec:appendix:adjoint_other_rules}

The following series of examples captures natural deduction-style rules for coends, where coends are on the right side of the turnstile.

\begin{example}[Elimination for coends]
The following derivation captures an elimination rule for coends, where $[\Gamma, d\!:\!\Delta]\ \Phi(d) \propctx, Q(d) \prop$, $[x\!:\!\Cop,y\!:\!\C,d\!:\!\Delta]\ P(x,y,d) \prop$, with variables in $\Delta$ always being used \emph{naturally}:%
\vspace{-1.0em}\begin{adju}[1.0]%
\begin{prooftree}
\infer[no rule]0{[\Gamma,d\!:\!\Delta]\ \Phi(d) \vdash \Coendf{x:\C} P(\n x,x,d)}
\infer1[\Rulecoendinv]{[d\!:\!\Delta]\ \Coendf{\gamma:\Gamma} \Phi(\n \gamma,\gamma,d) \vdash \Coendf{x:\C} P(\n x,x,d)}
\infer[no rule]0{[\Gamma,z:\C,d\!:\!\Delta]\ P(\n z,z,d),\Phi(\n \gamma, \gamma, d) \vdash Q(d)}
\infer1[\Rulecoendinv]{[\Gamma,d\!:\!\Delta]\ \Coendf{z:\C} P(\n z,z,d),\Phi(\n \gamma, \gamma, d) \vdash Q(d)}
\infer1[\hspace{-0.6em}$\begin{array}{l}\Rulecoendinv\,+\\\Ruleend\end{array}$]{[d\!:\!\Delta]\ P(\n z,z,d),\Coendf{\gamma:\Gamma} \Phi(\n \gamma, \gamma, d) \vdash \Endf{\gamma:\Gamma} Q(\n \gamma, \gamma,d)}
\infer2[\Rulecutnat]{[d:\Delta]\ \Coendf{\gamma:\Gamma} \Phi(\n \gamma, \gamma, d) \vdash \Coendf{\gamma:\Gamma} Q(\n \gamma, \gamma, d)}
\infer1[\Rulecoendfrobenius+\Ruleendinv]{[\Gamma,d:\Delta]\ \Phi(d) \vdash Q(d)}
\end{prooftree}
\end{adju}
\end{example}

\begin{example}[Introduction for coends with a term]
The following derivation captures an introduction rule for coends with a generic term $\Delta \vdash F : \C$ (not a \emph{di}term), for $[\Gamma, d:\Delta]\ \Phi(d) \propctx$, $[x:\C,d:\Delta]\ Q(x,d) \prop$:
\[
\begin{prooftree}
\infer[no rule]0{[\Gamma,d:\Delta]\ \Phi(d) \vdash Q(F(d),d)}
\infer0[\Rulecoendunit]{[x:\C,d:\Delta]\ Q(x,d) \vdash \Coendf{x:\C} Q(x,d)}
\infer1[\Rulereindex]{[d:\Delta]\ Q(F(d),d) \vdash \Coendf{x:\C} Q(x,d)}
\infer2{[\Gamma,d:\Delta]\ \Phi(d) \vdash \Coendf{x:\C} Q(x,d)}
\end{prooftree}
\]
In particular, we picked $\Rulecoendunit$ with $Q$ depending on just a single variable and reindexed with $F$, which ignores the negative context.  Note that variables in $\Delta$ are always used naturally.
\end{example}

\begin{example}[Introduction for coends with a dinatural variable]
The following derivation captures an introduction rule for coends with a dinatural variable $x$, for $[x:\Cop,y:\C, \Gamma, d:\Delta]\ \Phi(x,y,d) \propctx$, $[x:\Cop,y:\C,d:\Delta]\ Q(x,y,d) \prop$:
\[
\begin{prooftree}
\infer[no rule]0{[\Gamma,x:\Cop,y:\C]\ \Phi(x,y,d) \vdash Q(x,y,d)}
\infer0[\Rulecoendunit]{[d:\Delta]\ Q(\n z,z,d) \vdash \Coendf{z:\C} Q(\n z,z,d)}
\infer2[\Rulecutdin]{[\Gamma]\ \Phi \vdash \Coendf{z:\C} Q(\n z,z,d)}
\end{prooftree}
\]
In particular, we picked $\Rulecoendunit$ with $Q$ depending naturally on $x,y,z$. Note that variables in $\Delta$ are always used naturally.
\end{example}

\section{(Co)end calculus, other derivations}
\label{appendix:sec:other_derivations_coend_calculus}

We report here additional examples of derivations for (co)end calculus using our rules.

\begin{example}[Pointwise fomula for left Kan extensions] Dually to \Cref{exa:right_kan}, we give a logical proof that the functor $\Lan_F : [\C,\Set] \> [\D,\Set]$ sending (co)presheaves to their left Kan extensions along $F : \C \> \D$ computed via coends \cite[2.3.6]{Loregian2021coend} is left adjoint to precomposition $(F\<-) : [\D,\Set] \> [\C,\Set]$. For any $[x:\C]\ P(x)\prop$, a functor/term $\C \vdash F : \D$ and a generic $[y:\D]\ \varphi(y) \prop$:
\newlength{\coenddepth}
\settodepth{\coenddepth}{$\Coendf{y:\C}$}
\[
\begin{prooftree}
\infer[no rule]0{[y:\D]\ \takespace{\Coendf{y:\C} \hom_\D(F(y),\overline x) \x P(y)}{(\Lan_F P)(x)} := & }
\infer[no rule]1{\raisebox{0pt}[11pt][\coenddepth]{$\Coendf{x:\C}$} \hom_\C(F(\n x), y) \x P(x) \vdash & \varphi(y)}
\infer[double]1[\Rulecoendfrobenius]{[x:\C,y:\D]\ \hom_\C(F(\n x), y) \x P(x) \vdash & \varphi(y)}
\infer[double]1[\Ruleexp]{[x:\C,y:\D]\ P(x) \vdash & \hom_\C(F(x),\overline y)
 \Rightarrow \varphi(y)}
\infer[double]1[\Ruleend]{[x:\C]\ P(x) \vdash & \Endf{y:\D} \hom_\D(F(x),\overline y) \Rightarrow \varphi(y)}
\infer[double]1[\Ruleyoneda]{[x:\C]\ P(x) \vdash & \varphi(F(x))}
\end{prooftree}
\]
\end{example}

\begin{example}[Right rifts in profunctors]
We give a logical proof that composition (on both sides) in $\Prof$ has a right adjoint \cite[5.2.5 and Exercise 5.2]{Loregian2021coend}. This makes $\Prof$ a bicategory where \emph{right extensions} and \emph{right lifts} exist.
For simplicity we only treat precomposition, although postcomposition is completely analogous.
For any composable profunctors $[x:\Cop,y:\A]\ P(x,y)\prop$,$[x:\Aop,y:\D]\ Q(x,y)\prop$ and a generic $[x:\Cop,y:\D]\ \varphi(x,y)\prop$:
\[
\begin{prooftree}
\infer[no rule]0{[x:\Cop,z:\D]\takespace{\Coendf{y:\A} P(x,y) \times Q(\n y,z)}{(P \< -)(Q)(x,z)} := &}
\infer[no rule]1{\Coendf{y:\A} P(x,y) \times Q(\n y,z) \vdash & \varphi(x,z)}
\infer[double]1[\Rulecoendfrobenius]{[x:\Cop,y:\A,z:\D]\ P(x,y) \times Q(\n y,z) \vdash & \varphi(x,z)}
\infer[double]1[\Ruleexp]{[x:\Cop,y:\A,z:\D]\ Q(\n y,z) \vdash & P(\n x,\n y) \Rightarrow \varphi(x,z)}
\infer[double]1[\Ruleend]{[y:\A,z:\D]\ Q(\n y,z) \vdash & \Endf{x:\C} P(\n x,\n y) \Rightarrow \varphi(x,z)}
\infer[double]1[\Ruleop]{[y:\Aop,z:\D]\ Q(y,z) \vdash & \Endf{x:\C} P(\n x,y) \Rightarrow \varphi(x,z)}
\infer[no rule]1{ := & \Rift_P(\varphi)(y,z) \phantom{\Coendf{y:\A}}}%
\end{prooftree}
\]
where the last $\Ruleend$ can be applied since $x:\C$ does not appear on the left.
\end{example}

\begin{example}[Composition of profunctors is associative]
Using our approach relying on contextual operations we easily show that composition of profunctors, defined via a coend \cite{Loregian2021coend}, is associative and essentially follows from associativity of products. For composable profunctors $[x:\Aop,y:\B]\ P(x,y) \prop$, $[x:\Bop,y:\C]\ Q(x,y) \prop$, $[x:\Cop,y:\D]\ R(x,y) \prop$, and a generic $[x:\Aop,y:\D]\ \varphi(x,y) \prop$:

\[
\begin{prooftree}
\infer[no rule]0{[a:\A,d:\D]\ \Coendf{b:\B} P(\n a,b) \times \left(\Coendf{c:\C} Q(\n b,c) \times R(\n c,d) \right) \vdash & \varphi(\n a,d)}
\infer[double]1[\Rulecoendfrobenius]{[a:\A,b:\B,d:\D]\ P(\n a,b) \times \left(\Coendf{c:\C} Q(\n b,c) \times R(\n c,d) \right) \vdash & \varphi(\n a,d)}
\infer[double]1[\Rulecoendfrobenius]{[a:\A,b:\B,c:\C,d:\D]\ P(\n a,b) \times (Q(\n b,c) \times R(\n c,d)) \vdash & \varphi(\n a,d)}
\infer[double]1[\textsf{(structural property)}]{[a:\A,b:\B,c:\C,d:\D]\ (P(\n a,b) \times Q(\n b,c)) \times R(\n c,d) \vdash & \varphi(\n a,d)}
\infer[double]1[\Rulecoendfrobenius]{[a:\A,c:\C,d:\D]\ \left(\Coendf{b:\B} P(\n a,b) \times Q(\n b,c)\right) \times R(\n c,d) \vdash & \varphi(\n a,d)}
\infer[double]1[\Rulecoendfrobenius]{[a:\A,d:\D]\ \Coendf{c:\C} \left(\Coendf{b:\B} P(\n a,b) \times Q(\n b,c)\right) \times R(\n c,d) \vdash & \varphi(\n a,d)}
\end{prooftree}
\]
\end{example}

\begin{theorem}[Dinaturals as an end]
\label{appendix:dinats_as_ends}
The set of dinaturals $\mathsf{Dinat}(P,Q) := \set{P \dinat Q}$ between dipresheaves $P,Q:\Cop\x\C\>\Set$ can be characterized in terms of the following end \cite[Thm. 1]{Dubuc1970dinatural}, $\mathsf{Dinat}(P,Q) \iso \Endf{x:\C} P(\xnx) \Rightarrow Q(\nxx)$.
\end{theorem}
\begin{proof}
We give a simple derivation that characterizes all the points (i.e., dinaturals from the point in the empty term context) of the end above using our syntax:
\[
\takespace{\mathsf{Dinat}(P,Q) :=\ }{ }
\begin{prooftree}
\infer[no rule]0{\takespace{\,}{\smash{\mathsf{Dinat}(P,Q) :=\ }} [x: \C]\ P(\nxx) & \vdash Q(\nxx)}
\infer[double]1[\Ruleexp]{[x: \C]\ \emptyctx & \vdash P(\xnx) \Rightarrow Q(\nxx)}
\infer[double]1[\Ruleend]{[\, ]\ \emptyctx & \vdash \Endf{x:\C} P(\xnx) \Rightarrow Q(\nxx)}
\end{prooftree}\vspace{-0.2em}
\]
Since dinaturals generalize naturals, a similar derivation justifies the well-known description of natural transformations as ends shown in \Cref{sec:introduction} for $F,G:\C\>\Set$,\[\Nat(F,G) \iso \Endf{x:\C} F(\overline x) \Rightarrow G(x).\]
\end{proof}

\section{Computation rule via $J^{-1}$}
\label{appendix:computation_jinv}
We spell out the proof of the computation rule for the definition of $J^{-1}$ given in \Cref{thm:refl_from_hexagon}.

\begin{theorem}[$\J^{-1}\!\!\iff\!\!\refl$]
\!Rule $\Rulerefl$ is logically equivalent to $\RuleJinv$; in particular, assuming naturality of $J^{-1}$, if one defines $\refl_\C := J^{-1}(e)$ then the computation rule $J(h)[\refl_\C] = h$ holds in the equational theory.
\end{theorem}
\begin{proof}
We start by spelling out naturality of $J^{-1}$ in $P$, which is assumed: explicitly, naturality states that the following two derivations are equal in the equational theory for any $\alpha$ and $\beta$ (simplifying the context as much as possible for readability):
\[
\begin{prooftree}
\hypo{[a:\Cop,b:\C]\ e:\hom_\C(a,b), \Phi(\n b,\n a) & \vdash \alpha : P(a,b)}
\infer1{[z:\C]\ \Phi(\n z,z) & \vdash \J^{-1}(\alpha[e]) : P(\n z,z)}
\hypo{[z:\C]\ k:P(a,b), \Phi(\n a,\n b) & \vdash \beta[k] : Q(a,b)}
\infer2{[z:\C]\ \Phi(\n z,z) & \vdash \beta[\J^{-1}(\alpha)] : Q(\n z,z)}
\end{prooftree}
\]
and
\[
\begin{prooftree}
\hypo{[a:\Cop,b:\C]\ e:\hom_\C(a,b), \Phi(\n b,\n a) & \vdash \alpha : P(a,b)}
\hypo{[z:\C]\ k:P(a,b), \Phi(\n a,\n b) & \vdash \beta[k] : Q(a,b)}
\infer2{[z:\C]\ e:\hom_\C(a,b), \Phi(\n b,\n a) & \vdash \beta[\alpha] : P(\n z,z)}
\infer1{[z:\C]\ \Phi(\n z,z) & \vdash \J^{-1}(\beta[\alpha]) : Q(\n z,z)}
\end{prooftree}
\]
i.e., $\beta[\J^{-1}(\alpha)] = \J^{-1}(\beta[\alpha])$. In our particular case we take $P(a,b) := \hom(a,b)$ and $\alpha := e$ the projection with $\Rulevar$ and $\beta := J(h)$, from which we obtain that $J(h)[\textsf{refl}_\C] \equiv J(h)[J^{-1}(e)] = J^{-1}(J(h)[e]) = J^{-1}(J(h)) = h$ by the assumption that $J^{-1}(J(h)) = h$ and the fact that $\Rulevar$ is the identity for cut.
\end{proof}

\section{Frobenius and Beck-Chevalley conditions for (co)ends}\label{appendix:bc_frobenius}

\begin{theorem}[Beck-Chevalley and Frobenius condition for (co)ends]
\label{thm:theorem_bc_frobenius}
(Co)ends satisfy a \emph{Beck-Chevalley condition}, in the sense that for all $F: \C^\diamond\>\D$ there is a strict isomorphism \[\Endf{\A[\D]}~\< \,F^* \iso (\id_{\A^\diamond}~\times F)^* \< \Endf{\A[\C]}\] in the  (large) functor category $[[\A^\diamond\hspace{-1pt}\x\D^\diamond\hspace{-2pt},\Set],\hspace{-2pt}[\D^\diamond\hspace{-2pt},\Set]]$, where \[\Endf{\A[\C]}, \Coendf{\A[\C]} : [\A^\diamond \x \C^\diamond,\Set] \> [\C^\diamond,\Set]\] are the functors sending dipresheaves to their (co)end in $\A$ and $F^*: [\D^\diamond,\Set] \> [\C^\diamond,\Set]$ is precomposition with $F^\diamond$.

Moreover, a \emph{Frobenius condition} for coends is satisfied, in the sense that there is an isomorphism \[\Coendf{\A[\C]} (\pi_{\A[\C]}^*(P) \times \Phi) \iso \pi_{\A[\C]}^*(P) \times \Coendf{\A[\C]}(\Phi)\] \emph{natural} in $\Phi:\A^\diamond\times\C^\diamond\>\Set,P:\C^\diamond\>\Set$, where $- \times -\!:\![\C,\Set] \x [\C,\Set]\!\>\![\C,\Set]$ for any $\C$ is the product of (di)presheaves.
\end{theorem}
\begin{proof}
Beck-Chevalley is immediate. For Frobenius, our logical rules can be used to apply exactly the argument given in~\cite[1.9.12(i)]{Jacobs1999categorical}, detailed in \Cref{appendix:thm:frobenius}.
\end{proof}

\begin{theorem}[Frobenius condition for coends]
\label{appendix:thm:frobenius}
For any $\G:\A^\diamond\x\C^\diamond\>\Set$ and a generic $K:\C^\diamond\>\Set$, the following series of derivations gives a logical proof of the Frobenius condition given in \Cref{thm:theorem_bc_frobenius}, which we prove by following exactly the argument given in \cite[1.9.12(i)]{Jacobs1999categorical} in the case of fibrations with exponentials. In particular, we show that the Frobenius formulation of (co)ends follows from the non-Frobenius one combined with polarized exponentials. Note that we use the same Yoneda technique described in \Cref{rem:dinatu_iso}.
\[
\begin{prooftree}
\hypo{[\Gamma]\ \Coendf{x:\A[\Gamma]} (P \times \Phi(\n x,x)) & \vdash \varphi}
\infer[double]1[\Rulecoend]{[x:\A,\Gamma]\ P, \Phi(\n x,x) & \vdash \varphi}
\infer[double]1[\Ruleexp]{[x:\A,\Gamma]\ \Phi(\n x,x) & \vdash P \Rightarrow \varphi}
\infer[double]1[\Rulecoend]{[\Gamma]\ \Coendf{x:\A[\Gamma]} \Phi(\n x,x) & \vdash P \Rightarrow \varphi}
\infer[double]1[\Ruleexp]{[\Gamma]\ P, \Coendf{x:\A[\Gamma]} \Phi(\n x,x) & \vdash \varphi}
\end{prooftree}
\]
\end{theorem}
\begin{theorem}[$\Rulecoend\!\Rightarrow\!\Rulecoendfrobenius$]
The rule $\Rulecoendfrobenius$ can be directly justified using $\Rulecoend$, as follows:
\[
\begin{prooftree}
\infer[no rule]0{[\Gamma]\ {\left(\Coendf{a:\A} Q(\naa)\right)\!, \Phi} & \vdash \varphi}
\infer[double]1[\Ruleexp]{[\Gamma]\ {\Coendf{a:\A} Q(\naa)\! } & \vdash \Phi(\xnx) \Rightarrow \varphi}
\infer[double]1[\Rulecoend]{[y : \C, \Gamma]\ {Q(\naa)} & \vdash \Phi(\xnx) \Rightarrow \varphi}
\infer[double]1[\Ruleexp]{[\Gamma]\ {Q(\naa) ,\Phi} & \vdash \varphi}
\end{prooftree}
\]
\end{theorem}

\section{Yoneda technique}\label{sec:appendix:yoneda_technique}

We show how the Yoneda technique described in \Cref{rem:dinatu_iso} can be used to prove a derivation of (co)end calculus. We show the case of Yoneda \Cref{thm:derivation_coyoneda}.
\[
\begin{prooftree}
   \infer[no rule]0{[a\!:\!\C]~\Phi(a) \vdash \Endf{x:\C}\ {\hom}_\C(a,\overline x) \Rightarrow P(x)}
   \infer[double]1[\Ruleend]{[a\!:\!\C,x\!:\!\C]~\Phi(a) \vdash \hom_\C(a,\overline x) \Rightarrow P(x)}
   \infer[double]1[\Ruleexp]{[a\!:\!\C,x\!:\!\C]\hspace{1.6pt}\ \hom_\C(\overline a,x) \x \Phi(a) \hspace{1.3pt} \vdash P(x)}
   \infer[double]1[\RuleJ]{[z:\C]~\Phi(z) \vdash P(z)}
\end{prooftree}
\]
Explicitly, the two entailments witnessing the isomorphism are obtained by picking $\Phi$ to be the context with a single formula and the $\Rulevar$ case at the top of the derivation, i.e.,
\[
\begin{prooftree}
   \infer0[\Rulevar]{[z:\C]~k:P(z) \vdash k : P(z)}
   \infer1[\RuleJ]{[a\!:\!\C,x\!:\!\C]~k:P(a), {\hom}_\C(\n a, x) \vdash \J(k) : P(x)}
   \infer1[\Ruleexp]{[a\!:\!\C,x\!:\!\C]~k:P(a) \vdash \textsf{exp}(\J(k)) : {\hom}_\C(a,\overline x) \Rightarrow P(x)}
   \infer1[\Ruleend]{[a\!:\!\C]~k:P(a) \vdash \mathsf{end}(\textsf{exp}(\J(k))) : \Endf{x:\C}\ {\hom}_\C(a,\overline x) \Rightarrow P(x)}
\end{prooftree}
\]
and
\[
\begin{prooftree}
   \infer0[\Rulevar]{[a\!:\!\C]~k:\Endf{x:\C}\ {\hom}_\C(a,\overline x) \Rightarrow P(x) \vdash k : \Endf{x:\C}\ {\hom}_\C(a,\overline x) \Rightarrow P(x)}
   \infer1[\Ruleendinv]{[a\!:\!\C,x\!:\!\C]~k:\cdots \vdash \hom_\C(a,\overline x) \Rightarrow P(x)}
   \infer1[\Ruleexpinv]{[a\!:\!\C,x\!:\!\C]\hspace{1.6pt}\ k:\cdots, \hom_\C(\overline a,x) \hspace{1.3pt} \vdash P(x)}
   \infer1[\RuleJinv]{[z:\C]~k:\Endf{x:\C}\ {\hom}_\C(z,\overline x) \Rightarrow P(x) \vdash \J^{-1}(\textsf{exp}^{-1}(\mathsf{end}^{-1}(k))) : P(z)}
\end{prooftree}
\]
These two entailments can clearly be composed since they are both natural transformations. They compose to the identity in both directions by using the same approach when proving fully faithfulness of the Yoneda embedding~\cite{Leinster2014basic}, i.e., using naturality of each rule in $\Phi$ to make them commute with cuts and then using the fact that all rules are invertible:
\[
\begin{prooftree}
   \infer0{[a\!:\!\C]~k:P(a) \vdash & \J^{-1}(\textsf{exp}^{-1}(\mathsf{end}^{-1}(k)))[k \mapsto \mathsf{end}(\textsf{exp}(\J(k)))]}
   \infer[no rule]1{ = & \J^{-1}(\textsf{exp}^{-1}(\mathsf{end}^{-1}(k))[k \mapsto \mathsf{end}(\textsf{exp}(\J(k)))])}
   \infer[no rule]1{ = & \J^{-1}(\textsf{exp}^{-1}(\mathsf{end}^{-1}(k)[k \mapsto \mathsf{end}(\textsf{exp}(\J(k)))]))}
   \infer[no rule]1{ = & \J^{-1}(\textsf{exp}^{-1}(\mathsf{end}^{-1}(k[k \mapsto \mathsf{end}(\textsf{exp}(\J(k)))])))}
   \infer[no rule]1{ = & \J^{-1}(\textsf{exp}^{-1}(\mathsf{end}^{-1}(\mathsf{end}(\textsf{exp}(\J(k))))))}
   \infer[no rule]1{ = & \J^{-1}(\textsf{exp}^{-1}(\textsf{exp}(\J(k))))}
   \infer[no rule]1{ = & \J^{-1}(\J(k))}
   \infer[no rule]1{ = & k: P(a)}
\end{prooftree}
\]

Note that we are propagating the cut along the hypothesis $k$ in context (this is only ambiguous in the rule $\Ruleexp$ since there are two hypotheses, where we leave $f:\hom(a,b)$ untouched).

The other direction is obtained analogously.

\section{Composite in \Cref{ex:internal_dinaturality}}
\label{sec:appendix:illustrate_internal_dinat}

Given a dinatural transformation
\[
[z:\C]\,k:P(\nzz) \vdash \alpha : Q(\nzz)
\]
we illustrate how the composite
\[
[a:\Cop, b:\C]\,f:\hom_\C(a,b), k:P(\n b,\n a) \vdash \subst_Q[(f,\refl_a),[\alpha[\subst_P[(\refl_b,f),k]]]] : Q(a,b)
\]
in \Cref{ex:internal_dinaturality} is indeed allowed by the cut rules of our type theory, i.e., that dinaturals compose. The well-formedness of \Cref{ex:internal_naturality} follows similarly since it is a special case of the one below. We construct one of the two sides of the equation, with the other one following similarly.

The key idea is that \textsf{subst} is essentially a natural transformation when saturated in the function $f$ (even partially). The \textsf{subst} of a predicate $[a:\Cop,b:\C]\ Q(z,b)$ depending on two variables corresponds to the following entailment:
\[
[a',b:\Cop, a,b':\C]\,f:\hom_{\C}(a',a), g:\hom_\C(b,b'), k:Q(\n a,\n b) \vdash \subst_P[f,g,k] : P(a',b')
\]
After precomposing $f$ with $\refl$ and renaming variables via \Cref{thm:op_of_entailments} note that the resulting map is \emph{natural in $z,b$} after currying the equality $g$ to the right.
\[
[b,z:\Cop, b':\C]\,g:\hom_\C(b,b'), k:P(z,\n b) \vdash \subst_P[\refl_z,g,k] : P(z,b')
\]
This map \emph{can be precomposed} with $\alpha$ by picking $b$ to be part of the variables of $\Gamma$ in the rule $\Rulecutdin$. The intuition for this, described in \Cref{sec:semantics} for the semantics of cut, is that one can take the (co)end over $b$ and obtain the above family as \emph{natural} in $z$ and $b'$, without $b$ appearing, which then \emph{can} be composed with $\alpha$ in the expression $\alpha[\textsf{subst}_P[(\textsf{refl}_b,f),k]]$. The remaining part of the term is then obtained by using $\Rulecutnat$ to compose with $\textsf{subst}_Q$ in an analogous way.

\bibliographystyle{ACM-Reference-Format}
\bibliography{iwilare}

@misc{Laretto2025di,
  author        = {Laretto, Andrea and Loregian, Fosco and Veltri, Niccolò},
  title         = {Di{-} is for Directed: First-Order Directed Type Theory via Dinaturality},
  doi           = {10.48550/arxiv.2409.10237},
  eprinttype    = {arxiv},
  year          = {2025},
  eprint        = {2409.10237}
}

@misc{Laretto2025artifact,
  author        = {Laretto, Andrea},
  title         = {Artifact for ``Di{-} is for Directed: First-Order Directed Type Theory via Dinaturality''},
  doi           = {10.5281/zenodo.17788596},
  year          = {2025},
  howpublished       = {Zenodo}
}

@book{Jacobs1999categorical,
  author    = {Jacobs, Bart P. F.},
  title     = {Categorical Logic and Type Theory},
  publisher = {{North-Holland}},
  series    = {Studies in Logic and the Foundations of Mathematics},
  volume    = {141},
  bibsource = {dblp computer science bibliography, https://dblp.org},
  year      = {1999}
}

@inproceedings{Roman2020open,
  author        = {Román, Mario},
  booktitle     = {Electronic {Proceedings} in {Theoretical} {Computer} {Science}},
  title         = {Open {Diagrams} via {Coend} {Calculus}},
  doi           = {10.4204/EPTCS.333.5},
  eprint        = {2004.04526v4},
  language      = {en},
  pages         = {65--78},
  volume        = {333},
  archiveprefix = {arxiv},
  creationdate  = {2023-07-31T18:29:41},
  issn          = {2075-2180},
  month         = feb,
  readstatus    = {read},
  year          = {2020}
}

@book{Shulman2016categorical,
  author       = {Michael Shulman},
  title        = {Categorical Logic from a Categorical Point of View},
  url          = {https://github.com/mikeshulman/catlog},
  creationdate = {2023-07-31T19:04:45},
  year         = {2016}
}

@article{Lawvere1969adjointness,
  author       = {Lawvere, Francis W.},
  title        = {Adjointness in {Foundations}},
  issn         = {0012-2017},
  number       = {3/4},
  pages        = {281--296},
  url          = {https://www.jstor.org/stable/42969800},
  urldate      = {2023-07-31},
  volume       = {23},
  creationdate = {2023-07-31T19:16:57},
  journal      = {Dialectica},
  publisher    = {Wiley},
  year         = {1969}
}

@book{Leinster2014basic,
  author       = {Leinster, Tom},
  title        = {Basic {Category} {Theory}},
  doi          = {10.1017/CBO9781107360068},
  isbn         = {9781107044241},
  publisher    = {Cambridge University Press},
  series       = {Cambridge {Studies} in {Advanced} {Mathematics}},
  url          = {https://www.cambridge.org/core/books/basic-category-theory/A72533879BBC7BD956CC415777B7DA99},
  urldate      = {2023-07-31},
  address      = {Cambridge},
  creationdate = {2023-07-31T19:21:55},
  year         = {2014}
}

@book{Loregian2021coend,
  author       = {Loregian, Fosco},
  title        = {({Co})end {Calculus}},
  doi          = {10.1017/9781108778657},
  isbn         = {9781108746120},
  publisher    = {Cambridge University Press},
  series       = {London {Mathematical} {Society} {Lecture} {Note} {Series}},
  url          = {https://www.cambridge.org/core/books/coend-calculus/C662E90767358B336F17B606D19D8C43},
  urldate      = {2023-07-31},
  address      = {Cambridge},
  creationdate = {2023-07-31T19:28:04},
  year         = {2021}
}

@inproceedings{Caccamo2001higher,
  author       = {Cáccamo, Mario and Winskel, Glynn},
  booktitle    = {Theorem {Proving} in {Higher} {Order} {Logics}},
  title        = {A {Higher}-{Order} {Calculus} for {Categories}},
  doi          = {10.1007/3-540-44755-5\_11},
  editor       = {Boulton, Richard J. and Jackson, Paul B.},
  isbn         = {9783540447559},
  language     = {en},
  pages        = {136--153},
  publisher    = {Springer},
  series       = {Lecture {Notes} in {Computer} {Science}},
  address      = {Berlin, Heidelberg},
  creationdate = {2023-07-31T22:23:20},
  year         = {2001}
}

@book{Borceux1994handbook,
  author       = {Borceux, Francis},
  title        = {Handbook of {Categorical} {Algebra}: {Volume} 1: {Basic} {Category} {Theory}},
  doi          = {10.1017/CBO9780511525858},
  isbn         = {9780521441780},
  publisher    = {Cambridge University Press},
  series       = {Encyclopedia of {Mathematics} and its {Applications}},
  url          = {https://www.cambridge.org/core/books/handbook-of-categorical-algebra/A0B8285BBA900AFE85EED8C971E0DE14},
  urldate      = {2023-08-01},
  volume       = {1},
  address      = {Cambridge},
  creationdate = {2023-08-01T00:17:44},
  shorttitle   = {Handbook of {Categorical} {Algebra}},
  year         = {1994}
}

@book{UnivalentFoundationsProgram2013homotopy,
  author       = {The {Univalent Foundations Program}},
  title        = {Homotopy Type Theory: Univalent Foundations of Mathematics},
  publisher    = {\url{https://homotopytypetheory.org/book}},
  address      = {Institute for Advanced Study},
  creationdate = {2023-08-01T00:26:52},
  year         = {2013}
}

@incollection{Baez2010physics,
  author        = {Baez, John and Stay, Michael},
  booktitle     = {New Structures for Physics},
  title         = {Physics, Topology, Logic and Computation: A {Rosetta} Stone},
  doi           = {10.48550/arxiv.0903.0340},
  eprint        = {0903.0340},
  eprinttype    = {arxiv},
  pages         = {95--172},
  publisher     = {Springer},
  archiveprefix = {arXiv},
  creationdate  = {2023-08-01T15:14:44},
  langid        = {english},
  month         = mar,
  shorttitle    = {Physics, {{Topology}}, {{Logic}} and {{Computation}}},
  year          = {2010}
}

@article{McCusker2021composing,
  author       = {Guy McCusker and Alessio Santamaria},
  title        = {Composing dinatural transformations: Towards a calculus of substitution},
  doi          = {10.1016/j.jpaa.2021.106689},
  number       = {10},
  pages        = {106689},
  volume       = {225},
  creationdate = {2023-08-01T15:16:24},
  journal      = {Journal of Pure and Applied Algebra},
  month        = oct,
  publisher    = {Elsevier {BV}},
  year         = {2021}
}

@misc{Clarke2022profunctor,
  author        = {Bryce Clarke and Derek Elkins and Jeremy Gibbons and Fosco Loregian and Bartosz Milewski and Emily Pillmore and Mario Román},
  title         = {Profunctor Optics, a Categorical Update},
  doi           = {10.48550/arxiv.2001.07488},
  eprint        = {2001.07488},
  archiveprefix = {arXiv},
  creationdate  = {2023-08-01T15:21:58},
  primaryclass  = {cs.PL},
  year          = {2022}
}

@article{Wood1982abstract,
  author       = {Wood, Robert J.},
  title        = {{A}bstract proarrows {I}},
  number       = {3},
  pages        = {279--290},
  volume       = {23},
  creationdate = {2023-08-01T16:07:26},
  journal      = {Cahiers de topologie et géometrie différentielle categoriques},
  publisher    = {Dunod éditeur, publié avec le concours du CNRS},
  year         = {1982}
}

@article{Boisseau2018what,
  author       = {Boisseau, Guillaume and Gibbons, Jeremy},
  title        = {What you needa know about {Yoneda}: profunctor optics and the {Yoneda} lemma (functional pearl)},
  doi          = {10.1145/3236779},
  number       = {ICFP},
  pages        = {84:1--84:27},
  urldate      = {2023-08-01},
  volume       = {2},
  creationdate = {2023-08-01T16:58:06},
  journal      = {Proceedings of the ACM on Programming Languages},
  month        = jul,
  shorttitle   = {What you needa know about {Yoneda}},
  year         = {2018}
}

@inproceedings{Boisseau2020string,
  author       = {Guillaume Boisseau},
  booktitle    = {5th International Conference on Formal Structures for Computation and Deduction (FSCD 2020)},
  title        = {{String Diagrams for Optics}},
  doi          = {10.4230/LIPIcs.FSCD.2020.17},
  editor       = {Zena M. Ariola},
  isbn         = {978-3-95977-155-9},
  pages        = {17:1--17:18},
  publisher    = {Schloss Dagstuhl--Leibniz-Zentrum f{\"u}r Informatik},
  series       = {Leibniz International Proceedings in Informatics (LIPIcs)},
  url          = {https://drops.dagstuhl.de/opus/volltexte/2020/12339},
  volume       = {167},
  address      = {Dagstuhl, Germany},
  creationdate = {2023-08-01T16:58:52},
  issn         = {1868-8969},
  urn          = {urn:nbn:de:0030-drops-123399},
  year         = {2020}
}

@article{Cruttwell2010unified,
  author        = {Geoff S.H. Cruttwell and Michael Shulman},
  title         = {A unified framework for generalized multicategories},
  doi           = {10.48550/arxiv.0907.2460},
  eprint        = {0907.2460},
  pages         = {580--655},
  volume        = {24},
  archiveprefix = {arxiv},
  creationdate  = {2023-08-01T17:04:44},
  journal       = {Theory Appl. Categ.},
  year          = {2010}
}

@article{Uustalu2020eilenberg,
  author       = {Uustalu, Tarmo and Veltri, Niccolò and Zeilberger, Noam},
  title        = {Eilenberg-{Kelly} {Reloaded}},
  doi          = {10.1016/j.entcs.2020.09.012},
  issn         = {1571-0661},
  language     = {en},
  pages        = {233--256},
  series       = {The 36th {Mathematical} {Foundations} of {Programming} {Semantics} {Conference}, 2020},
  url          = {https://www.sciencedirect.com/science/article/pii/S1571066120300633},
  urldate      = {2023-08-01},
  volume       = {352},
  creationdate = {2023-08-01T18:19:20},
  journal      = {Electronic Notes in Theoretical Computer Science},
  month        = oct,
  year         = {2020}
}

@phdthesis{Lawvere1963functorial,
  author       = {Lawvere, Francis W.},
  title        = {{Functorial Semantics of Algebraic Theories}},
  creationdate = {2023-08-04T01:23:44},
  school       = {Columbia University},
  year         = {1963}
}

@inproceedings{Pistone2021yoneda,
  author       = {Paolo Pistone and Luca Tranchini},
  booktitle    = {29th EACSL Annual Conference on Computer Science Logic (CSL 2021)},
  title        = {{The Yoneda Reduction of Polymorphic Types}},
  doi          = {10.4230/LIPIcs.CSL.2021.35},
  editor       = {Christel Baier and Jean Goubault-Larrecq},
  isbn         = {978-3-95977-175-7},
  pages        = {35:1--35:22},
  publisher    = {Schloss Dagstuhl--Leibniz-Zentrum f{\"u}r Informatik},
  series       = {Leibniz International Proceedings in Informatics (LIPIcs)},
  url          = {https://drops.dagstuhl.de/opus/volltexte/2021/13469},
  volume       = {183},
  address      = {Dagstuhl, Germany},
  creationdate = {2023-08-04T01:50:21},
  issn         = {1868-8969},
  urn          = {urn:nbn:de:0030-drops-134696},
  year         = {2021}
}

@inproceedings{Pistone2018proof,
  author       = {Paolo Pistone},
  booktitle    = {Proceedings Joint International Workshop on Linearity {\&} Trends in Linear Logic and Applications, Linearity-TLLA@FLoC 2018, Oxford, UK, 7-8 July 2018},
  title        = {Proof nets, coends and the Yoneda isomorphism},
  doi          = {10.4204/EPTCS.292.9},
  editor       = {Thomas Ehrhard and Maribel Fern{\'{a}}ndez and Valeria de Paiva and Lorenzo Tortora de Falco},
  pages        = {148--167},
  series       = {{EPTCS}},
  volume       = {292},
  bibsource    = {dblp computer science bibliography, https://dblp.org},
  creationdate = {2023-08-04T01:51:00},
  year         = {2018}
}

@phdthesis{Santamaria2019towards,
  author       = {Alessio Santamaria},
  title        = {Towards a Godement calculus for dinatural transformations},
  url          = {https://ethos.bl.uk/OrderDetails.do?uin=uk.bl.ethos.787523},
  bibsource    = {dblp computer science bibliography, https://dblp.org},
  creationdate = {2023-08-04T02:28:43},
  school       = {University of Bath},
  year         = {2019}
}

@book{Mimram2020program,
  author       = {Mimram, Samuel},
  title        = {Program = Proof},
  isbn         = {9798615591839},
  publisher    = {Independently Published},
  url          = {https://books.google.ee/books?id=nzZzzgEACAAJ},
  creationdate = {2023-08-04T16:40:36},
  year         = {2020}
}

@inproceedings{New2023formal,
  author       = {New, Max S. and Licata, Daniel R.},
  booktitle    = {Foundations of {Software} {Science} and {Computation} {Structures}},
  title        = {A {Formal} {Logic} for {Formal} {Category} {Theory}},
  doi          = {10.1007/978-3-031-30829-1\_6},
  editor       = {Kupferman, Orna and Sobocinski, Pawel},
  isbn         = {9783031308291},
  language     = {en},
  pages        = {113--134},
  publisher    = {Springer Nature Switzerland},
  series       = {Lecture {Notes} in {Computer} {Science}},
  creationdate = {2023-08-04T22:00:36},
  year         = {2023}
}

@inproceedings{Dubuc1970dinatural,
  author       = {Dubuc, Eduardo and Street, Ross},
  booktitle    = {Reports of the {Midwest} {Category} {Seminar} {IV}},
  title        = {Dinatural transformations},
  doi          = {10.1007/BFb0060443},
  editor       = {MacLane, S. and Applegate, H. and Barr, M. and Day, B. and Dubuc, E. and {Phreilambud} and Pultr, A. and Street, R. and Tierney, M. and Swierczkowski, S.},
  isbn         = {9783540362920},
  language     = {en},
  pages        = {126--137},
  publisher    = {Springer},
  series       = {Lecture {Notes} in {Mathematics}},
  address      = {Berlin, Heidelberg},
  creationdate = {2023-08-04T22:21:14},
  year         = {1970}
}

@inproceedings{Asada2010arrows,
  author       = {Asada, Kazuyuki},
  booktitle    = {Proceedings of the third {ACM} {SIGPLAN} workshop on {Mathematically} structured functional programming},
  title        = {Arrows are strong monads},
  doi          = {10.1145/1863597.1863607},
  isbn         = {9781450302555},
  pages        = {33--42},
  publisher    = {Association for Computing Machinery},
  series       = {{MSFP} '10},
  urldate      = {2023-08-05},
  address      = {New York, NY, USA},
  creationdate = {2023-08-05T01:26:46},
  month        = sep,
  year         = {2010}
}

@book{Heunen2019categories,
  author       = {Heunen, Chris and Vicary, Jamie},
  title        = {{Categories for Quantum Theory: An Introduction}},
  doi          = {10.1093/oso/9780198739623.001.0001},
  eprint       = {https://academic.oup.com/book/43710/book-pdf/50991591/9780191060069\_web.pdf},
  isbn         = {9780198739623},
  publisher    = {Oxford University Press},
  creationdate = {2023-08-05T03:38:21},
  month        = nov,
  year         = {2019}
}

@book{MacLane1998categories,
  author       = {Mac Lane, Saunders},
  title        = {Categories for the Working Mathematician},
  doi          = {10.1007/978-1-4757-4721-8},
  edition      = {2nd},
  isbn         = {0-387-98403-8},
  pages        = {xii+314},
  publisher    = {Springer-Verlag New York},
  series       = {Graduate Texts in Mathematics},
  volume       = {5},
  creationdate = {2023-08-05T15:01:58},
  year         = {1998}
}

@incollection{Hofmann1997syntax,
  author       = {Hofmann, Martin},
  booktitle    = {Semantics and Logics of Computation},
  title        = {Syntax and {Semantics} of {Dependent} {Types}},
  doi          = {10.1017/CBO9780511526619.004},
  editor       = {Pitts, Andrew M. and Dybjer, Peter},
  isbn         = {9780521580571},
  pages        = {79--130},
  publisher    = {Cambridge University Press},
  series       = {Publications of the {Newton} {Institute}},
  url          = {https://www.cambridge.org/core/books/semantics-and-logics-of-computation/syntax-and-semantics-of-dependent-types/119C8085C6A1A0CD7F24928EF866748F},
  urldate      = {2023-08-05},
  address      = {Cambridge},
  creationdate = {2023-08-05T15:15:59},
  year         = {1997}
}

@inproceedings{Ahrens2022semantics,
  author       = {Ahrens, Benedikt and North, Paige Randall and van der Weide, Niels},
  booktitle    = {Proceedings of the 37th {Annual} {ACM}/{IEEE} {Symposium} on {Logic} in {Computer} {Science}},
  title        = {Semantics for two-dimensional type theory},
  doi          = {10.1145/3531130.3533334},
  isbn         = {9781450393515},
  pages        = {1--14},
  publisher    = {Association for Computing Machinery},
  series       = {{LICS} '22},
  urldate      = {2023-08-05},
  address      = {New York, NY, USA},
  creationdate = {2023-08-05T17:37:57},
  month        = aug,
  year         = {2022}
}

@inproceedings{Hinze2012kan,
  author       = {Hinze, Ralf},
  booktitle    = {Mathematics of {Program} {Construction}},
  title        = {Kan {Extensions} for {Program} {Optimisation} {Or}: {Art} and {Dan} {Explain} an {Old} {Trick}},
  doi          = {10.1007/978-3-642-31113-0\_16},
  editor       = {Gibbons, Jeremy and Nogueira, Pablo},
  isbn         = {9783642311130},
  language     = {en},
  pages        = {324--362},
  publisher    = {Springer},
  series       = {Lecture {Notes} in {Computer} {Science}},
  address      = {Berlin, Heidelberg},
  creationdate = {2023-08-07T02:04:57},
  shorttitle   = {Kan {Extensions} for {Program} {Optimisation} {Or}},
  year         = {2012}
}

@inproceedings{Galal2020profunctorial,
  author       = {Zeinab Galal},
  booktitle    = {5th International Conference on Formal Structures for Computation and Deduction (FSCD 2020)},
  title        = {{A Profunctorial Scott Semantics}},
  doi          = {10.4230/LIPIcs.FSCD.2020.16},
  editor       = {Zena M. Ariola},
  isbn         = {978-3-95977-155-9},
  pages        = {16:1--16:18},
  publisher    = {Schloss Dagstuhl--Leibniz-Zentrum f{\"u}r Informatik},
  series       = {Leibniz International Proceedings in Informatics (LIPIcs)},
  url          = {https://drops.dagstuhl.de/opus/volltexte/2020/12338},
  volume       = {167},
  address      = {Dagstuhl, Germany},
  creationdate = {2023-08-07T02:28:51},
  issn         = {1868-8969},
  urn          = {urn:nbn:de:0030-drops-123387},
  year         = {2020}
}

@inproceedings{Altenkirch2016type,
  author       = {Altenkirch, Thorsten and Kaposi, Ambrus},
  booktitle    = {Proceedings of the 43rd {Annual} {ACM} {SIGPLAN}-{SIGACT} {Symposium} on {Principles} of {Programming} {Languages}},
  title        = {Type theory in type theory using quotient inductive types},
  doi          = {10.1145/2837614.2837638},
  isbn         = {9781450335492},
  pages        = {18--29},
  publisher    = {Association for Computing Machinery},
  series       = {{POPL} '16},
  urldate      = {2023-08-07},
  address      = {New York, NY, USA},
  creationdate = {2023-08-07T03:08:30},
  month        = jan,
  year         = {2016}
}

@article{Licata20112,
  author       = {Licata, Daniel R. and Harper, Robert},
  title        = {2-{Dimensional} {Directed} {Type} {Theory}},
  doi          = {10.1016/j.entcs.2011.09.026},
  issn         = {1571-0661},
  language     = {en},
  pages        = {263--289},
  series       = {Twenty-seventh {Conference} on the {Mathematical} {Foundations} of {Programming} {Semantics} ({MFPS} {XXVII})},
  url          = {https://www.sciencedirect.com/science/article/pii/S1571066111001174},
  urldate      = {2023-08-07},
  volume       = {276},
  creationdate = {2023-08-07T13:28:51},
  journal      = {Electronic Notes in Theoretical Computer Science},
  month        = sep,
  year         = {2011}
}

@article{Asada2010categorifying,
  author       = {Asada, Kazuyuki and Hasuo, Ichiro},
  title        = {Categorifying {Computations} into {Components} via {Arrows} as {Profunctors}},
  doi          = {10.1016/j.entcs.2010.07.012},
  issn         = {1571-0661},
  language     = {en},
  number       = {2},
  pages        = {25--45},
  series       = {Proceedings of the {Tenth} {Workshop} on {Coalgebraic} {Methods} in {Computer} {Science} ({CMCS} 2010)},
  url          = {https://www.sciencedirect.com/science/article/pii/S157106611000071X},
  urldate      = {2023-08-07},
  volume       = {264},
  creationdate = {2023-08-07T14:36:48},
  journal      = {Electronic Notes in Theoretical Computer Science},
  month        = aug,
  year         = {2010}
}

@book{Harper2016practical,
  author       = {Harper, Robert},
  title        = {Practical {Foundations} for {Programming} {Languages}},
  edition      = {2nd},
  isbn         = {9781107150300},
  publisher    = {Cambridge University Press},
  address      = {USA},
  creationdate = {2023-08-07T19:01:19},
  month        = mar,
  year         = {2016}
}

@incollection{Hofmann1998groupoid,
  author       = {Hofmann, Martin and Streicher, Thomas},
  booktitle    = {Twenty-five years of constructive type theory ({V}enice, 1995)},
  title        = {The groupoid interpretation of type theory},
  doi          = {10.1093/oso/9780198501275.003.0008},
  editor       = {Sambin, Giovanni and Smith, Jan M},
  isbn         = {9780198501275},
  pages        = {83--111},
  publisher    = {Oxford Univ. Press, New York},
  series       = {Oxford Logic Guides},
  urldate      = {2023-08-17},
  volume       = {36},
  creationdate = {2023-08-17T02:00:36},
  month        = oct,
  mrclass      = {03B15 (68N15 68Q55)},
  mrnumber     = {1686862},
  year         = {1998}
}

@article{Cockx2021taming,
  author       = {Cockx, Jesper and Tabareau, Nicolas and Winterhalter, Théo},
  title        = {The taming of the rew: a type theory with computational assumptions},
  doi          = {10.1145/3434341},
  number       = {POPL},
  pages        = {60:1--60:29},
  urldate      = {2023-08-22},
  volume       = {5},
  creationdate = {2023-08-22T15:25:31},
  journal      = {Proceedings of the ACM on Programming Languages},
  month        = jan,
  shorttitle   = {The taming of the rew},
  year         = {2021}
}

@misc{Castellan2020categories,
  author        = {Castellan, Simon and Clairambault, Pierre and Dybjer, Peter},
  title         = {Categories with {Families}: {Unityped}, {Simply} {Typed}, and {Dependently} {Typed}},
  doi           = {10.48550/arXiv.1904.00827},
  eprint        = {1904.00827},
  archiveprefix = {arxiv},
  creationdate  = {2023-09-16T01:51:04},
  month         = jul,
  school        = {arXiv},
  shorttitle    = {Categories with {Families}},
  year          = {2020}
}

@mastersthesis{Nuyts2015towards,
  author       = {Andreas Nuyts},
  title        = {Towards a Directed Homotopy Type Theory based on 4 Kinds of Variance},
  creationdate = {2023-09-18T09:07:12},
  school       = {KU Leuven},
  year         = {2015}
}

@incollection{Lawvere1970equality,
  author       = {Lawvere, Francis W.},
  booktitle    = {Applications of Categorical Algebra},
  title        = {Equality in hyperdoctrines and comprehension schema as an adjoint functor},
  editor       = {Heller, A.},
  pages        = {1--14},
  publisher    = {American Mathematical Society},
  address      = {Providence, R.I.},
  creationdate = {2023-08-07T00:39:18},
  mrclass      = {18.10},
  mrnumber     = {MR0257175 (41 \#1829)},
  mrreviewer   = {H. Gonshor},
  year         = {1970}
}

@book{Lambek1986introduction,
  author       = {Lambek, Joachim and Scott, Philip J.},
  title        = {Introduction to {Higher}-{Order} {Categorical} {Logic}},
  language     = {en},
  publisher    = {Cambridge University Press},
  series       = {Cambridge Studies in Advanced Mathematics},
  url          = {https://www.cambridge.org/ee/academic/subjects/mathematics/logic-categories-and-sets/introduction-higher-order-categorical-logic, https://www.cambridge.org/ee/academic/subjects/mathematics/logic-categories-and-sets},
  urldate      = {2023-09-28},
  volume       = {7},
  creationdate = {2023-09-28T22:45:54},
  year         = {1986}
}

@article{North2019towards,
  author       = {North, Paige Randall},
  title        = {Towards a {Directed} {Homotopy} {Type} {Theory}},
  doi          = {10.1016/j.entcs.2019.09.012},
  issn         = {1571-0661},
  pages        = {223--239},
  series       = {Proceedings of the {Thirty}-{Fifth} {Conference} on the {Mathematical} {Foundations} of {Programming} {Semantics}},
  url          = {https://www.sciencedirect.com/science/article/pii/S1571066119301288},
  urldate      = {2023-10-07},
  volume       = {347},
  creationdate = {2023-10-07T00:41:13},
  journal      = {Electronic Notes in Theoretical Computer Science},
  month        = nov,
  year         = {2019}
}

@misc{Neumann2023paranatural,
  author        = {Neumann, Jacob},
  title         = {Paranatural {Category} {Theory}},
  doi           = {10.48550/arXiv.2307.09289},
  eprint        = {2307.09289},
  archiveprefix = {arxiv},
  creationdate  = {2023-10-27T15:06:32},
  month         = jul,
  school        = {arXiv},
  year          = {2023}
}

@article{Ahrens2023bicategorical,
  author       = {Ahrens, Benedikt and North, Paige Randall and van der Weide, Niels},
  title        = {Bicategorical type theory: semantics and syntax},
  doi          = {10.1017/S0960129523000312},
  issn         = {0960-1295, 1469-8072},
  language     = {en},
  pages        = {1--45},
  url          = {https://www.cambridge.org/core/journals/mathematical-structures-in-computer-science/article/bicategorical-type-theory-semantics-and-syntax/725F2E17B25094145F9D037D9465A534},
  urldate      = {2023-11-07},
  creationdate = {2023-11-07T13:24:51},
  journal      = {Mathematical Structures in Computer Science},
  month        = oct,
  publisher    = {Cambridge University Press},
  shorttitle   = {Bicategorical type theory},
  year         = {2023}
}

@article{Moggi1991notions,
  title   = {Notions of computation and monads},
  journal = {Information and Computation},
  volume  = {93},
  number  = {1},
  pages   = {55 - 92},
  year    = {1991},
  note    = {Selections from 1989 IEEE Symposium on Logic in Computer Science},
  author  = {Eugenio Moggi}
}

@inproceedings{Chu2025dependent,
  author  = {North, Paige Randall and Chu, Fernando},
  title   = {Dependent two-sided fibrations for directed type theory},
  booktitle = {TYPES 2025},
  year    = {2025}
}

@book{Casadio2021Joachim,
  address   = {Cham},
  editor    = {Claudia Casadio and Philip J. Scott},
  publisher = {Springer Verlag},
  title     = {Joachim Lambek: The Interplay of Mathematics, Logic, and Linguistics},
  year      = {2021}
}

@article{Riehl2017type,
  author  = {Emily Riehl and Michael Shulman},
  title   = {A type theory for synthetic $\infty$-categories},
  journal = {Higher structures},
  year    = 2017,
  volume  = 1,
  number  = 1,
  note    = {arXiv:1705.07442}
}

@article{Petric2003g,
  author       = {Petrić, Zoran},
  title        = {G-dinaturality},
  doi          = {10.1016/S0168-0072(03)00003-4},
  issn         = {0168-0072},
  number       = {1},
  pages        = {131--173},
  url          = {https://www.sciencedirect.com/science/article/pii/S0168007203000034},
  urldate      = {2023-11-22},
  volume       = {122},
  creationdate = {2023-11-22T15:16:39},
  journal      = {Annals of Pure and Applied Logic},
  month        = aug,
  year         = {2003}
}

@inproceedings{Plotkin1993logic,
  author    = {Plotkin, Gordon and Abadi, Mart{\'i}n},
  booktitle = {Typed Lambda Calculi and Applications},
  title     = {A logic for parametric polymorphism},
  year      = {1993},
  address   = {Berlin, Heidelberg},
  editor    = {Bezem, Marc and Groote, Jan Friso},
  pages     = {361--375},
  publisher = {Springer Berlin Heidelberg},
  isbn      = {978-3-540-47586-6}
}

@unpublished{Kavvos2019quantum,
  author = {Kavvos, Alex},
  year   = {2019},
  title  = {A Quantum of Direction},
  url    = {https://seis.bristol.ac.uk/~tz20861/papers/meio.pdf}
}

@phdthesis{Weaver2024bicubical,
  author      = {Weaver, Matthew},
  year        = {2024},
  title       = {Bicubical Directed Type Theory},
  institution = {Princeton University},
  url         = {https://dataspace.princeton.edu/handle/88435/dsp017s75dg778}
}

@phdthesis{Neumann2025generalized,
  author      = {Neumann, Jacob},
  title       = {A Generalized Algebraic Theory of Directed Equality},
  year        = {2025},
  institution = {University of Nottingham}
}

@article{Bainbridge1976feedback,
  author       = {Bainbridge, Edwin S.},
  title        = {Feedback and generalized logic},
  doi          = {10.1016/S0019-9958(76)90390-9},
  issn         = {0019-9958},
  number       = {1},
  pages        = {75--96},
  url          = {https://www.sciencedirect.com/science/article/pii/S0019995876903909},
  urldate      = {2024-01-26},
  volume       = {31},
  creationdate = {2024-01-26T22:25:37},
  journal      = {Information and Control},
  month        = may,
  year         = {1976}
}

@book{Crole1994categories,
  author       = {Crole, Roy L.},
  title        = {Categories for {Types}},
  doi          = {10.1017/CBO9781139172707},
  isbn         = {9780521450928},
  publisher    = {Cambridge University Press},
  url          = {https://www.cambridge.org/core/books/categories-for-types/258841251C62FED1DACD20884E59E61C},
  urldate      = {2024-01-26},
  address      = {Cambridge},
  creationdate = {2024-01-26T23:53:58},
  year         = {1994}
}

@article{Hefford2023coend,
  author        = {Hefford, James and Comfort, Cole},
  title         = {Coend {Optics} for {Quantum} {Combs}},
  doi           = {10.4204/EPTCS.380.4},
  eprint        = {2205.09027},
  issn          = {2075-2180},
  pages         = {63--76},
  volume        = {380},
  archiveprefix = {arxiv},
  creationdate  = {2024-05-27T16:01:46},
  journal       = {Electronic Proceedings in Theoretical Computer Science},
  month         = aug,
  year          = {2023}
}

@article{Eilenberg1966generalization,
  author       = {Eilenberg, Samuel and Kelly, G. Max},
  title        = {A generalization of the functorial calculus},
  doi          = {10.1016/0021-8693(66)90006-8},
  issn         = {0021-8693},
  language     = {en},
  number       = {3},
  pages        = {366--375},
  url          = {https://www.sciencedirect.com/science/article/pii/0021869366900068},
  urldate      = {2023-08-07},
  volume       = {3},
  creationdate = {2023-08-07T01:39:01},
  journal      = {Journal of Algebra},
  month        = may,
  year         = {1966}
}

@inproceedings{Girard1992normal,
  author       = {Girard, Jean-Yves and Scedrov, Andre and Scott, Philip J.},
  booktitle    = {Logic from {Computer} {Science}},
  title        = {Normal {Forms} and {Cut}-{Free} {Proofs} as {Natural} {Transformations}},
  doi          = {10.1007/978-1-4612-2822-6\_8},
  editor       = {Moschovakis, Yiannis N.},
  isbn         = {9781461228226},
  language     = {en},
  pages        = {217--241},
  publisher    = {Springer},
  address      = {New York, NY},
  creationdate = {2024-05-29T01:14:22},
  year         = {1992}
}

@article{Bainbridge1990functorial,
  author       = {Bainbridge, Edwin S. and Freyd, Peter J. and Scedrov, Andre and Scott, Philip J.},
  title        = {Functorial polymorphism},
  doi          = {10.1016/0304-3975(90)90151-7},
  issn         = {0304-3975},
  number       = {1},
  pages        = {35--64},
  series       = {Special {Issue} {Fourth} {Workshop} on {Mathematical} {Foundations} of {Programming} {Semantics}, {Boulder}, {CO}, {May} 1988},
  url          = {https://www.sciencedirect.com/science/article/pii/0304397590901517},
  urldate      = {2024-05-29},
  volume       = {70},
  creationdate = {2024-05-29T03:40:12},
  journal      = {Theoretical Computer Science},
  month        = jan,
  year         = {1990}
}

@article{Pistone2019completeness,
  author       = {Pistone, Paolo},
  title        = {On completeness and parametricity in the realizability semantics of {System} {F}},
  doi          = {10.23638/LMCS-15(4:6)2019},
  issn         = {1860-5974},
  url          = {https://lmcs.episciences.org/5878},
  urldate      = {2024-06-30},
  volume       = {Volume 15, Issue 4},
  creationdate = {2024-06-30T22:31:31},
  journal      = {Logical Methods in Computer Science},
  month        = oct,
  publisher    = {Episciences.org},
  year         = {2019}
}

@inproceedings{Pistone2017dinaturality,
  author    = {Pistone, Paolo},
  booktitle = {2nd International Conference on Formal Structures for Computation and Deduction (FSCD 2017)},
  title     = {{On Dinaturality, Typability and beta-eta-Stable Models}},
  doi       = {10.4230/LIPIcs.FSCD.2017.29},
  editor    = {Miller, Dale},
  isbn      = {978-3-95977-047-7},
  pages     = {29:1--29:17},
  publisher = {Schloss Dagstuhl -- Leibniz-Zentrum f{\"u}r Informatik},
  series    = {Leibniz International Proceedings in Informatics (LIPIcs)},
  volume    = {84},
  address   = {Dagstuhl, Germany},
  issn      = {1868-8969},
  urn       = {urn:nbn:de:0030-drops-77291},
  year      = {2017}
}

@inproceedings{Weaver2020constructive,
  author       = {Weaver, Matthew Z. and Licata, Daniel R.},
  booktitle    = {Proceedings of the 35th {Annual} {ACM}/{IEEE} {Symposium} on {Logic} in {Computer} {Science}},
  title        = {A {Constructive} {Model} of {Directed} {Univalence} in {Bicubical} {Sets}},
  doi          = {10.1145/3373718.3394794},
  isbn         = {9781450371049},
  pages        = {915--928},
  publisher    = {Association for Computing Machinery},
  series       = {{LICS} '20},
  urldate      = {2024-06-30},
  address      = {New York, NY, USA},
  creationdate = {2024-06-30T23:38:31},
  month        = jul,
  year         = {2020}
}

@unpublished{Nuyts2023higher,
  author = {Nuyts, Andreas},
  school = {KU Leuven},
  series = {HoTT/UF 2023},
  title  = {Higher Pro-arrows: Towards a Model for Naturality Pretype Theory},
  year   = {2023}
}

@incollection{Freyd1992dinaturality,
  author       = {Freyd, Peter J. and Robinson, Edmund P. and Rosolini, Giuseppe},
  booktitle    = {Applications of Categories in Computer Science: Proceedings of the London Mathematical Society Symposium, Durham 1991},
  title        = {Dinaturality for free},
  doi          = {10.1017/CBO9780511525902.007},
  editor       = {Pitts, A. M. and Fourman, M. P. and Johnstone, P. T.},
  isbn         = {9780521427265},
  pages        = {107--118},
  publisher    = {Cambridge University Press},
  series       = {London {Mathematical} {Society} {Lecture} {Note} {Series}},
  url          = {https://www.cambridge.org/core/books/applications-of-categories-in-computer-science/dinaturality-for-free/E6E7F60E4A647D401B901BB936104359},
  urldate      = {2024-07-02},
  address      = {Cambridge},
  creationdate = {2024-07-02T00:13:52},
  year         = {1992}
}

@inproceedings{Freyd1992functorial,
  author       = {Peter J. Freyd and Edmund P. Robinson and Giuseppe Rosolini},
  booktitle    = {Proceedings of the Seventh Annual Symposium on Logic in Computer Science {(LICS} '92), Santa Cruz, California, USA, June 22-25, 1992},
  title        = {Functorial Parametricity},
  doi          = {10.1109/LICS.1992.185555},
  pages        = {444--452},
  publisher    = {{IEEE} Computer Society},
  bibsource    = {dblp computer science bibliography, https://dblp.org},
  creationdate = {2024-07-02T00:15:45},
  year         = {1992}
}

@inproceedings{Cohen2015cubical,
  author       = {Cohen, Cyril and Coquand, Thierry and Huber, Simon and M\"{o}rtberg, Anders},
  booktitle    = {21st International Conference on Types for Proofs and Programs (TYPES 2015)},
  title        = {{Cubical Type Theory: A Constructive Interpretation of the Univalence Axiom}},
  doi          = {10.4230/LIPIcs.TYPES.2015.5},
  editor       = {Uustalu, Tarmo},
  isbn         = {978-3-95977-030-9},
  pages        = {5:1--5:34},
  publisher    = {Schloss Dagstuhl -- Leibniz-Zentrum f{\"u}r Informatik},
  series       = {Leibniz International Proceedings in Informatics (LIPIcs)},
  volume       = {69},
  address      = {Dagstuhl, Germany},
  creationdate = {2024-07-17T17:28:07},
  issn         = {1868-8969},
  urn          = {urn:nbn:de:0030-drops-84754},
  year         = {2015}
}

@incollection{Scott2000some,
  author       = {Scott, Philip J.},
  title        = {Some aspects of categories in computer science},
  doi          = {10.1016/S1570-7954(00)80027-3},
  editor       = {Hazewinkel, M.},
  pages        = {3--77},
  publisher    = {North-Holland},
  series       = {Handbook of Algebra},
  urldate      = {2024-08-22},
  volume       = {2},
  creationdate = {2024-08-22T15:56:31},
  month        = jan,
  year         = {2000}
}

@article{Blute1993linear,
  author       = {Blute, Richard F.},
  title        = {Linear logic, coherence and dinaturality},
  doi          = {10.1016/0304-3975(93)90053-V},
  issn         = {0304-3975},
  number       = {1},
  pages        = {3--41},
  url          = {https://www.sciencedirect.com/science/article/pii/030439759390053V},
  urldate      = {2024-08-22},
  volume       = {115},
  creationdate = {2024-08-22T15:59:31},
  journal      = {Theoretical Computer Science},
  month        = jul,
  year         = {1993}
}

@inproceedings{Kudasov2024formalizing,
  author       = {Kudasov, Nikolai and Riehl, Emily and Weinberger, Jonathan},
  booktitle    = {Proceedings of the 13th ACM SIGPLAN International Conference on Certified Programs and Proofs},
  date         = {2024},
  title        = {Formalizing the $\infty$-Categorical Yoneda Lemma},
  doi          = {10.1145/3636501.3636945},
  isbn         = {9798400704888},
  location     = {London, UK},
  pages        = {274--290},
  publisher    = {Association for Computing Machinery},
  series       = {CPP 2024},
  address      = {New York, NY, USA},
  creationdate = {2024-09-06T12:48:57},
  numpages     = {17},
  year         = {2024}
}

@inproceedings{Voigtlaender2020free,
  author       = {Voigtländer, Janis},
  booktitle    = {Declarative {Programming} and {Knowledge} {Management}},
  title        = {Free {Theorems} {Simply}, via {Dinaturality}},
  doi          = {10.1007/978-3-030-46714-2\_16},
  editor       = {Hofstedt, Petra and Abreu, Salvador and John, Ulrich and Kuchen, Herbert and Seipel, Dietmar},
  isbn         = {9783030467142},
  language     = {en},
  pages        = {247--267},
  publisher    = {Springer International Publishing},
  address      = {Cham},
  creationdate = {2024-09-07T15:37:50},
  year         = {2020}
}

@article{Blute1996linear,
  author       = {Blute, Richard F. and Scott, Philip J.},
  title        = {Linear {Läuchli} semantics},
  doi          = {10.1016/0168-0072(95)00017-8},
  issn         = {0168-0072},
  number       = {2},
  pages        = {101--142},
  urldate      = {2024-09-07},
  volume       = {77},
  creationdate = {2024-09-07T15:44:16},
  journal      = {Annals of Pure and Applied Logic},
  month        = jan,
  year         = {1996}
}

@article{Blute1998shuffle,
  author       = {Blute, Richard F. and Scott, Philip J.},
  title        = {The {Shuffle} {Hopf} {Algebra} and {Noncommutative} {Full} {Completeness}},
  doi          = {10.2307/2586659},
  number       = {4},
  pages        = {1413--1436},
  volume       = {63},
  creationdate = {2024-09-07T15:48:40},
  journal      = {Journal of Symbolic Logic},
  publisher    = {Association for Symbolic Logic},
  year         = {1998}
}

@article{Maietti2015unifying,
  author       = {Maietti, Maria Emilia and Rosolini, Giuseppe},
  title        = {Unifying {Exact} {Completions}},
  doi          = {10.1007/s10485-013-9360-5},
  issn         = {1572-9095},
  language     = {en},
  number       = {1},
  pages        = {43--52},
  urldate      = {2024-09-08},
  volume       = {23},
  creationdate = {2024-09-08T16:12:45},
  journal      = {Applied Categorical Structures},
  month        = feb,
  year         = {2015}
}

@inproceedings{Uustalu2010note,
  author       = {Tarmo Uustalu},
  booktitle    = {7th Workshop on Fixed Points in Computer Science, {FICS} 2010, Brno, Czech Republic, August 21-22, 2010},
  title        = {A Note on Strong Dinaturality, Initial Algebras and Uniform Parameterized Fixpoint Operators},
  editor       = {Luigi Santocanale},
  pages        = {77--82},
  publisher    = {Laboratoire d'Informatique Fondamentale de Marseille},
  url          = {https://hal.archives-ouvertes.fr/hal-00512377/document\#page=78},
  bibsource    = {dblp computer science bibliography, https://dblp.org},
  creationdate = {2024-09-13T19:46:27},
  year         = {2010}
}

@article{Berg2010types,
  author       = {van den Berg, Benno and Garner, Richard},
  title        = {{Types are weak $\omega$-groupoids}},
  doi          = {10.1112/plms/pdq026},
  eprint       = {https://academic.oup.com/plms/article-pdf/102/2/370/4487337/pdq026.pdf},
  issn         = {0024-6115},
  number       = {2},
  pages        = {370--394},
  volume       = {102},
  creationdate = {2024-09-30T14:15:47},
  journal      = {Proceedings of the London Mathematical Society},
  month        = oct,
  year         = {2010}
}

@article{Awodey2009homotopy,
  author       = {Awodey, Steve and Warren, Michael A.},
  title        = {Homotopy theoretic models of identity types},
  doi          = {10.1017/S0305004108001783},
  number       = {1},
  pages        = {45--55},
  volume       = {146},
  creationdate = {2024-09-30T15:07:16},
  journal      = {Mathematical Proceedings of the Cambridge Philosophical Society},
  year         = {2009}
}

@article{Pare1998dinatural,
  author       = {Paré, Robert and Román, Leopoldo},
  title        = {Dinatural numbers},
  doi          = {10.1016/S0022-4049(97)00036-4},
  issn         = {0022-4049},
  number       = {1},
  pages        = {33--92},
  urldate      = {2024-09-30},
  volume       = {128},
  creationdate = {2024-09-30T16:14:12},
  journal      = {Journal of Pure and Applied Algebra},
  month        = jun,
  year         = {1998}
}

@misc{Gratzer2024directed,
  author       = {Gratzer, Daniel and Weinberger, Jonathan and Buchholtz, Ulrik},
  title        = {Directed univalence in simplicial homotopy type theory},
  doi          = {10.48550/arXiv.2407.09146},
  creationdate = {2025-01-08T17:39:09},
  year         = {2024},
  month        = jul
}

@incollection{Pitts1995categorical,
  author       = {Pitts, Andrew M.},
  booktitle    = {Handbook of Logic in Computer Science: Volume 5: Logic and Algebraic Methods},
  date         = {1995},
  title        = {Categorical logic},
  editor       = {Abramsky, S. and Gabbay, D. M. and Maibaum, T. S. E.},
  isbn         = {9780198537816},
  pages        = {39--123},
  publisher    = {Oxford University Press},
  urldate      = {2023-07-31},
  creationdate = {2025-01-08T17:38:37},
  month        = may,
  numpages     = {85},
  year         = {1995}
}

@inproceedings{Chu2024directed,
  author       = {Chu, Fernando and Mangel, \'{E}l\'{e}onore and North, Paige Randall},
  booktitle    = {30th {International} {Conference} on {Types} for {Proofs} and {Programs} {TYPES} 2024–{Abstracts}},
  title        = {A directed type theory for 1-categories},
  pages        = {205},
  url          = {https://types2024.itu.dk/abstracts.pdf#page=215},
  urldate      = {2024-09-05},
  creationdate = {2025-01-08T17:39:45},
  year         = {2024}
}

@misc{Altenkirch2024synthetic,
  author        = {Altenkirch, Thorsten and Neumann, Jacob},
  title         = {Synthetic 1-{Categories} in {Directed} {Type} {Theory}},
  doi           = {10.48550/arXiv.2410.19520},
  eprint        = {2410.19520},
  archiveprefix = {arxiv},
  creationdate  = {2025-01-17T23:41:25},
  month         = oct,
  publisher     = {arXiv},
  year          = {2024}
}

@book{Fajstrup2016directed,
  author       = {Fajstrup, Lisbeth and Goubault, Eric and Haucourt, Emmanuel and Mimram, Samuel and Raussen, Martin},
  title        = {Directed Algebraic Topology and Concurrency},
  doi          = {10.1007/978-3-319-15398-8},
  isbn         = {9783319153988},
  publisher    = {Springer International Publishing},
  creationdate = {2025-01-22T22:09:26},
  year         = {2016}
}

@book{Grandis2009directed,
  place      = {Cambridge},
  series     = {New Mathematical Monographs},
  title      = {Directed Algebraic Topology: Models of Non-Reversible Worlds},
  publisher  = {Cambridge University Press},
  author     = {Grandis, Marco},
  year       = {2009},
  collection = {New Mathematical Monographs}
}

@misc{Gratzer2025yoneda,
  author = {Daniel Gratzer and Jonathan Weinberger and Ulrik Buchholtz},
  date   = {2025},
  title  = {The {Yoneda} embedding in simplicial type theory},
  doi    = {10.48550/arXiv.2501.13229},
  year   = 2025,
  month  = jan
}

@inbook{Guerrini2004proof,
  place      = {Cambridge},
  series     = {London Mathematical Society Lecture Note Series},
  title      = {Proof Nets and the $\lambda$-Calculus},
  booktitle  = {Linear Logic in Computer Science},
  publisher  = {Cambridge University Press},
  author     = {Guerrini, Stefano},
  editor     = {Ehrhard, Thomas and Girard, Jean-Yves and Ruet, Paul and Scott, Philip J.},
  year       = {2004},
  pages      = {65–118},
  collection = {London Mathematical Society Lecture Note Series}
}

\end{document}